\documentclass{article}

\usepackage{amsmath,amssymb,latexsym,amsthm,mathrsfs,gensymb,graphicx,subcaption,cleveref,authblk}
\usepackage{tikz,color,bm,pgfplots}
\usetikzlibrary{pgfplots.fillbetween}

\usepackage{chngcntr}
\counterwithin{figure}{section}

\newtheorem{theo}{Theorem}[section]
\newtheorem{pro}[theo]{Proposition}
\newtheorem{lem}[theo]{Lemma}
\newtheorem{cor}[theo]{Corollary}

\newcommand{\ra}{\rightarrow}
\theoremstyle{definition}
\newtheorem{defin}[theo]{Definition}
\newtheorem{exa}[theo]{Example}

\newtheorem*{nota}{Notation}
\theoremstyle{remark}
\newtheorem{rem}[theo]{Remark}

\title{Analytic-geometric methods for finite Markov chains with applications to quasi-stationarity}

\author[1]{Persi Diaconis}
\author[2]{Kelsey Houston-Edwards\thanks{kedwards@olin.edu}}
\author[3]{Laurent Saloff-Coste\thanks{lsc@math.cornell.edu}}
\affil[1]{\textit{\small Departments of Mathematics and Statistics, Stanford University}}
\affil[2]{\textit{\small The Olin College of Engineering}}
\affil[3]{\textit{\small Department of Mathematics, Cornell University}}

\begin{document}

\maketitle

\begin{abstract}
For a relatively large class of well-behaved absorbing  (or killed) finite Markov chains, we give detailed quantitative estimates regarding the behavior of the chain before it is absorbed (or killed). Typical examples are random walks on box-like finite subsets of the square lattice $\mathbb Z^d$ absorbed (or killed) at the boundary. The analysis is based on Poincar\'e, Nash, and Harnack inequalities, moderate growth, and on the notions of John and inner-uniform domains. 
\end{abstract}

\tableofcontents

\section{Introduction} \setcounter{equation}{0}

\subsection{Basic ideas and scope}
Markov chains that are either absorbed or killed at boundary points are important in many applications. We refer to \cite{CMSM,DM} for entries to the vast literature regarding such chains and their applications. Absorption and killing are distinguished by what happens to the chain when it exits its domain $U$. In the killing case, it simply ceases to exist. In the absorbing case, the chain exits $U$ and gets absorbed at a specific boundary point which, from a classical viewpoint, is still part of the state space of the chain. 
In this paper we study the behavior of chains until they are either absorbed or killed, which means that there is no significant difference between the two cases. 
For simplicity, we will phrase the present work in the language of  Markov chains killed at the boundary.

The goal of this article is to explain how to apply to finite Markov chains a well-established circle of ideas developed for and used in the study of the heat equation with Dirichlet boundary condition in Euclidean domains and manifolds with boundary, or, equivalently, for Brownian motion killed at the boundary. By applying these techniques to some finite Markov chains, we can provide good estimates for the behavior of these chains until they are killed. These estimates are 
also very useful for computing probabilities concerning the exit position of the process, that is, the position when the chain is killed. Such probabilities are related to harmonic measure and time-constrained variants. This is discussed by the authors in a follow-up article~\cite{DHSZ2}.

In \cite{DM}, a very basic example of this sort is discussed, lazy simple random walk on $\{0,1\dots,N\}$ with absorption at $0$ and reflection at $N$. This served as a starting point for the present work. Even for such a simple example, the techniques developed below provide improved estimates.

The present approach utilizes powerful tools: Harnack, Poincar\'e and Nash inequalities. It leads to good results even for  domains whose boundaries are quite rugged, namely,  inner-uniform domains and John domains. The notions of  ``Harnack inequality'' and ``John domain'' are quite unfamiliar in the context of finite Markov chains and their installment in this context is non-trivial and interesting mostly when a quantitative viewpoint is implemented carefully. The main contribution of this work is to provide such an implementation. 

\begin{figure}[h]
\begin{center} 

\begin{picture}(200,140)(-50,-20)

{\color{blue} \multiput(7.6,4)(10,10){11}{\circle*{3}}}
{\color{blue} \multiput(4,4)(10,0){11}{\circle*{3} }}
\multiput(4,4)(10,0){11}{\line(0,1){100}}
\multiput(4,4)(0,10){11}{\line(1,0){100}}
\put(-4,-4){\makebox(4,4){$\scriptstyle (0,0)$}}  \put(110,-4){\makebox(4,4){$\scriptstyle (N,0)$}}

\multiput(24,14)(10,0){9}{\circle*{3}} \multiput(34,24)(10,0){8}{\circle*{3}} \multiput(44,34)(10,0){7}{\circle*{3}}
\multiput(54,44)(10,0){6}{\circle*{3}} \multiput(64,54)(10,0){5}{\circle*{3}} 
\multiput(74,64)(10,0){4}{\circle*{3}}
\multiput(84,74)(10,0){3}{\circle*{3}} \multiput(94,84)(10,0){2}{\circle*{3}} \put(104,94){\circle*{3}}
\multiput(104,14)(0,10){9}{\circle*{4}}
\end{picture}
\caption{The forty-five degree finite cone in $\mathbb Z^2$}\label{D0}
\end{center}\end{figure}
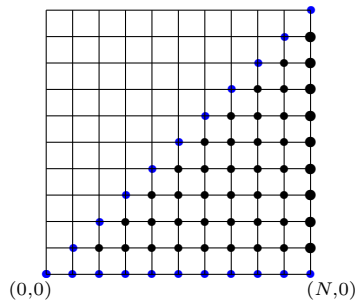

The type of finite Markov chains---more precisely, the type of families of finite Markov chains---to which these methods apply is, depending of one's perspective, both quite general and rather restrictive. First, we will mostly deal with reversible Markov chains. Second, the most technical part of this work applies only to families of finite Markov chains whose state spaces have a common  ``finite-dimensional'' nature. Our basic geometric assumptions require that all Markov chains in the general family under consideration have, roughly speaking, the same dimension. The model examples are families of finite Markov chains whose state spaces are subsets of $\mathbb{Z}^d$ for some fixed $d$, such as the family of forty-five degree cones parametrized by $N$ shown in Figure~\ref{D0}. Many interesting families of finite Markov chains evolve on state spaces that have an ``infinite-dimensional nature,'' e.g., the hypercube $\{0,1\}^d$ or the symmetric group $\mathbb S_d$ where $d$ grows to infinity. Our main results do not apply well to these ``infinite-dimensional'' families of Markov chains, although some intermediary considerations explained in this paper  do apply to such examples. See Section~\ref{sec-Doob}.

The simple example depicted in Figure \ref{D0} illustrates the aim of this work.  Start with simple random walk on the square grid in the plane. For each integer $N>2$, consider the subgraph of the square grid consisting of those vertices $(p,q)$  such that
$$  q<p\le N  \mbox{ and } 0<q<N,$$ which are depicted by black dots on Figure~\ref{D0}.  Call this set of vertices~${U=U_N}$.  The boundary where the chain is killed (depicted in blue) consists of the bottom and diagonal sides of the cone, i.e., the vertices with either $q=0$ or $p=q$ for $0\le p,q\le N$.  Call this set $\partial U =\partial U_N$ and set $\mathfrak X=\mathfrak X_N=U_N\cup \partial U_N$. The vertices along the right side of the cone, $\{(N,q), 1\le q<N\}$, have one less neighboring vertex, so we add a loop at each of these vertices. (In Figure~\ref{D0}, these vertices are depicted with larger black dots and the loops are omitted for simplicity.) 

We are interested in understanding the behavior of the simple random walk on $\mathfrak X$ killed at the boundary $\partial U$, before its random killing time $\tau_U$. In particular, we would like to have good approximations of quantities
such as 
\begin{equation}
\label{eq:est1}
\mathbf P_x(\tau_U >\ell), \;\; \mathbf P_x( X_t=y| \tau_U >\ell ), \;\; \mathbf P_x( X_t=y \mbox{ and } \tau_U>n),
\end{equation}
for $x,y\in U, \ 0\le t\le \ell,$ and
\begin{equation}
\label{eq:est2}
\lim_{\ell\ra \infty}  \mathbf P_x( X_t=y| \tau_U >\ell ),
\end{equation}
for $ x,y\in U, \ 0\le t <+\infty$ where the time parameter $t$ is integer valued.
This limit, if it exists, can be interpreted as the iterated transition probability at time $t$ for the chain conditioned to never be absorbed. We chose the example in Figure \ref{D0} because it is a rather simple domain, but already demonstrates some of the complexities in approximating the above quantities.

\subsection{The Doob-transform technique}
Before looking at this example in detail, consider a general irreducible aperiodic Markov kernel $K$ on a finite or countable state space $\mathfrak X$. Let $U$ be a finite subset of $\mathfrak X$ such that the kernel $K_U(x,y)=K(x,y)\mathbf 1_U(x)\mathbf 1_U(y)$ is still irreducible and aperiodic. Let $(X_t)$ be the (discrete time) random walk on $\mathfrak X$ driven by $K$, and let $\tau_U$ be the stopping time equal to the time of the first exit from $U$ as above. 

A rather general result explained in Section~\ref{sec-Dir} implies that the limit 
$$\lim_{\ell\ra \infty}  \mathbf P_x( X_t=y| \tau_U >\ell ), \;\;x,y\in U,\;\;t \in \mathbb{N}_{\geq 0}$$ exists and so we can define $K_{\mbox{\tiny Doob}}^t(x,y)$ for any $x, y \in U$ and $t \in \mathbb{N}_{\geq 0}$ as
$$K_{\mbox{\tiny Doob}}^t(x,y) = \lim_{\ell\ra \infty}  \mathbf P_x( X_t=y| \tau_U >\ell).$$
It is not immediately clear that this collection of $t$-dependent kernels, $$K^1_{\mbox{\tiny Doob}},K^2_{\mbox{\tiny Doob}},K^3_{\mbox{\tiny Doob}},\dots,$$  has special properties but, it turns out that it is nothing other than the collection of the iterated  kernels of the kernel $K_{\mbox{\tiny Doob}}=K^1_{\mbox{\tiny Doob}}$ itself, i.e.,
$$K_{\mbox{\tiny Doob}}^t(x,y)=\sum_z K^{t-1}_{\mbox{\tiny Doob}}(x,z)K_{\mbox{\tiny Doob}}(z,y).$$
Moreover, $K_{\mbox{\tiny Doob}}$ is an irreducible aperiodic Markov kernel.

To see why this is true, let us explicitly find the kernel $K_{\mbox{\tiny Doob}}$. Recall that, by the Perron-Frobenius theorem, the irreducible, aperiodic, non-negative kernel $K_U$ has a real eigenvalue $\beta_0 \in [0,1]$ which is simple and such that $|\beta| < \beta_0$ for every other eigenvalue $\beta$.
This top eigenvalue $\beta_0$ has a right eigenfunction $\phi_0$ and a left eigenfunction $\phi_0^*$ which are both positive functions on $U$.  Set
$$   K_{\phi_0}(x,y)= \beta_0^{-1}  \phi_0(x)^{-1} K_U(x,y)\phi_0(y)$$
and observe that this is an irreducible aperiodic  Markov kernel with invariant probability measure proportional to $\phi^*_0\phi_0$. These facts all follow 
from the definition and elementary algebra.

In Section~\ref{sec-IU}, we show that
$$\lim_{\ell\ra \infty}  \mathbf P_x( X_t=y| \tau_U >\ell ) = K^t_{\phi_0}(x,y),$$
and hence
$$K_{\mbox{\tiny Doob}}^t(x,y) = K^t_{\phi_0}(x,y).$$
 This immediately implies that
$$\mathbf P_x(X_t=y \mbox{ and } \tau_U >t)=K^t_U(x,y)= \beta_0^t K^t_{\mbox{\tiny Doob}}(x,y) \phi_0(x)\phi_0(y)^{-1}.$$

If we assume---this is a big and often unrealistic assumption---that we know the eigenfunction $\phi_0$, either via an explicit formula or via ``good two-sided estimates,'' then any question about
$$\mathbf P_x(X_t=y \mbox{ and } \tau_U >t) \ \text{or, equivalently,} \  K^t_U(x,y)$$
can be answered by studying $$K^t_{\mbox{\tiny Doob}}(x,y)$$ and vice-versa.  The key point of this technique is that $K_{\mbox{\tiny Doob}}$ is an irreducible aperiodic Markov kernel with invariant measure  proportional to $\phi_0^*\phi_0$ and its ergodic properties can be investigated using a wide variety of classical tools.  

The notation $K_{\mbox{\tiny Doob}} $ refers to the fact that this well-established circle of ideas is known as the Doob-transform technique. From now on, we will use the name $K_{\phi_0}$ instead, to remind the reader about the key role of the eigenfunction $\phi_0$.

\subsection{The 45 degree finite discrete cone} 
In our specific example depicted in Figure~\ref{D0}, $K_U$ is symmetric 
in $x,y$ so that~${\phi_0^*=\phi_0}$. We let $\pi_U\equiv 2/N(N-1) $ denote the uniform measure on $U$ and normalize $\phi_0$ by the natural condition 
$\pi_U(\phi_0^2)=1$. Then, $\pi_{\phi_0}=\phi_0^2 \pi_U$ is the invariant probability measure of $K_{\phi_0}$ and this pair 
$(K_{\phi_0} , \pi_{\phi_0})$ is irreducible, aperiodic, and reversible. By applying known quantitative methods to this particular aperiodic, irreducible, ergodic Markov chain, we can approximate the quantities~\eqref{eq:est1} and~\eqref{eq:est2} as follows. 

For any  $x=(p,q)\in U$
 and any  $t$,   set  $x_{\sqrt{t}}=(p_{\sqrt{t}},q_{\sqrt{t}})$ where 
$$p_{\sqrt{t}}= (p+ 2\lfloor \sqrt{t/4}\rfloor)\wedge N \mbox{ and }q_{\sqrt{t}}=(q+ \lfloor  \sqrt{t/4}\rfloor)\wedge (N/2). $$
The transformation $x=(p,q)\mapsto x_{\sqrt{t}}=(p_{\sqrt{t}},q_{\sqrt{t}})$  takes any vertex $x=(p,q)$ and pushes it inside $U$ and away from the boundary at scale $\sqrt{t}$
(at least as long as $t\le N$). The two key properties of  $x_{\sqrt{t}}$ are that it is  at distance at most $\sqrt{t}$ from $x$ and at a distance from the boundary $\partial U$ of order at least 
$\sqrt{t}\wedge N$.  

The following six statements can be proven using the techniques in this paper. The first five of these statements generalize to a large class of examples that will be described in detail. The last statement takes advantage of the particular structure of the example in Figure~\ref{D0}. Note that the constants $c,C$ may change from line to line but are independent of $N,t$ and $x,y\in U=U_N.$
\begin{enumerate}
\item For all $N$, $ cN^{-2} \le 1-\beta_0 \le   CN^{-2}.$ This eigenvalue estimate gives a basic rate at which mass disappears from $U$. For a more precise statement, see item 5 below.
\item All eigenvalues of $K_U$ are real, the smallest one, $\beta_{\mbox{\tiny min}}$, satisfies 
$$ \beta_0+\beta_{\mbox{\tiny min}}\ge c\beta_0 N^{-2}$$  and, for any eigenvalue $\beta$ other than $\beta_0$, 
$$\beta_0-\beta\ge c\beta_0 N^{-2}.$$
This inequality shows that $\frac{\beta_{\mbox{\tiny min}}}{\beta_0}$, the smallest eigenvalue of $K_{\phi_0}$, is strictly larger than~$-1$, which implies the aperiodicity of $K_{\phi_0}$.
\item   For all $x,y,t,N$ with  $t\ge N^2$   
$$ \max_{x,y} \left\{ \left| \frac{N(N-1)\mathbf P_x(X_t=y \mbox{ and } \tau_U >t)}{2\beta_0^t \phi_0(x)\phi_0(y) } -1\right| \right\} \le Ce^{- c t/N^2}.$$
A simple interpretation of this (and the following) statement is that $$\mathbf{P}_x(X_t = y \ \text{and} \ \tau_U > t)\;\; (\mbox{resp. } \mathbf{P}_x(\tau_U > t))$$ is asymptotic to a known function expressed in terms of $\beta_0$ and $\phi_0$.
\item  For all $x,t,N$ with  $t\ge N^2$,  
$$ \max_{x} \left\{ \left| \frac{N(N-1)\mathbf P_x(\tau_U >t)}{2\beta_0^t \phi_0(x) \pi_U(\phi_0) } -1\right| \right\} \le Ce^{- c t/N^2}.$$
\item    For all $x, t,N$,  $$ c\beta_0^t\frac{\phi_0(x)}{\phi_0(x_{\sqrt{t}})}\le \mathbf P_x(\tau_U>t)  \le   C \beta_0^t \frac{\phi_0(x)}{\phi_0(x_{\sqrt{t}})} .$$
Unlike the third and fourth statements on this list, which give asymptotic expressions for $$ \mathbf P_x(X_t=y\mbox{ and } \tau_U>t) \mbox{ and }\mathbf{P}_x(\tau_U > t)$$ for times greater than $N^2$, the fifth statement provides a two-sided bound of the survival probability $\mathbf P_x(\tau_U>t)$ that holds true uniformly for every starting point $x$ and  time $t>0$.
\item   For all $N$ and  $x=(p,q)\in U$, where $U$ is described in Figure~\ref{D0},     
$$ cpq(p+q)(p-q) N^{-4} \le \phi_0(x) \le Cpq(p+q)(p-q) N^{-4}  .$$
Observe that this detailed description of the somewhat subtle behavior of $\phi_0$ in all of $U$, together with the previous estimate of $\mathbf P_x(\tau_U>t)$, provides precise information for the survival probability of the process $(X_t)_{t>0}$ started at any given point in $U$. 
 \end{enumerate}
In general, it is hard to get detailed estimates on $\phi_0$, although some non-trivial and useful  properties of $\phi_0$ can be derived for large classes of examples. Even in the example given in Figure~\ref{D0}, the behavior of $\phi_0$ is not easily explained. In this case, it is actually possible to explicitly compute $\phi_0$:
\begin{eqnarray*}\phi_0(x) =4\kappa_N \sin \frac{ \pi p}{2N+1}\sin\frac{\pi q}{2N+1}
\left(\sin^2 \frac{\pi p}{2N+1} - \sin^2 \frac{\pi q}{2N+1}  \right).
\end{eqnarray*}
The constant $\kappa_N $ which makes this eigenfunction have $L^2(\pi_U)$-norm equal to $1$ can be computed to be $\kappa_N= \frac{\sqrt{8N(N-1)}}{2N+1}$. The eigenvalue $\beta_0$ is 
$$\beta_0= \frac{1}{2}\left(\cos\frac{\pi}{2N+1}+\cos\frac{3\pi}{2N+1}\right).$$

\subsection{A short guide}
Because some of the key techniques in this paper have a geometric flavor, we have chosen to emphasize the fact that all our examples are subordinate to some preexisting geometric structure. This underlying geometric structure introduces some of the key parameters that must remain fixed (or appropriately bounded) in order to obtain families of examples to which the results we seek to obtain apply uniformly. 

Generally, we use the language of graphs, and the most basic example of such a structure is a $d$-dimensional square grid. Throughout, the underlying 
space  is denoted by $\mathfrak X$. It is finite or countable and its elements are called vertices. It is equipped with an edge set $\mathfrak E$ which is a set of pairs undirected $\{x,y\}$ of distinct vertices (note that this excludes loops). Vertices in such pairs are called neighbors. For each $x\in \mathfrak X$, the number of pairs in $\mathfrak E$ that contain $x$ is supposed to be finite, i.e., the graph is locally finite. The structure $(\mathfrak X,\mathfrak E)$ yields a natural notion of a discrete path joining two vertices  and we assume that any two points in $\mathfrak X$  can indeed be joined by such a path.    

Two rather subtle types of finite subsets of $\mathfrak X$ play a key role in this work: \mbox{$\alpha$-John domains} 
and $\alpha$-inner-uniform domains. Inner-uniform domains are always John domains, but John domains are not always inner-uniform.  The number $\alpha\in (0,1]$ is a geometric parameter, and we will mostly consider families of subsets which are all either $\alpha$-John or $\alpha$-inner-uniform for one fixed $\alpha>0$.  John domains, named after Fritz John, are discussed in Section \ref{sec-JD} whereas the discussion and use of inner-uniform domains is postponed until Section \ref{sec-IU}. Our most complete results are for inner-uniform domains. These notions are well known in the context of (continuous) Euclidean domains, in particular in the field of conformal and quasi-conformal geometry. We provide a discrete version. See Figures \ref{D3}, \ref{IU-J}, and~\ref{IU-U} for simple examples.

Whitney coverings are a key tool used in proofs about John and inner-uniform domains. These are collections of inner balls within some domain that are nearly disjoint and have a radius that is proportional to the distance of the center to the boundary. These collections of balls are not themselves a covering of the domain, but their triples are, i.e., they generate a covering. See Section \ref{sec-WC} for the formal definition and Figure \ref{fig:whitney-covering} for an example. Whitney coverings are absolutely essential to the analysis presented in this paper. For instance, a Whitney covering of a given finite John domain $U$ is used to obtain good estimates for the second largest eigenvalue of a Markov chain (e.g., simple random walk on our graph) forced to remained in the finite domain $U$. See, e.g., Theorem \ref{th-KNU}.   

\begin{figure}[h]
\begin{center}
\begin{tikzpicture}[scale=.2]

\foreach \x in {0,30}
\draw [thick]   (15+\x,10)--(20+\x,5) (5+\x,5) -- (10+\x,5) -- (15+\x,5) -- (20+\x,5)-- (20+\x,10) --(15+\x,10) -- (10+\x,10) -- (10+\x,5) -- (15+\x,10)-- (15+\x,15)--((20+\x,10)--(20+\x,15)--(15+\x,15)--(10+\x,15)  (15+\x,15)-- (15+\x,20)  (20+\x,15)--(25+\x,15);

\foreach \x in {0,30}
\draw [fill]    (5+\x,5) circle [radius=.2] (10+\x,5) circle [radius=.2]  (15+\x,5) circle [radius=.2]  (20+\x,5) circle [radius=.2] (20+\x,10)  circle [radius=.2] (15+\x,10) circle [radius=.2]  (10+\x,10) circle [radius=.2] (10+\x,5) circle [radius=.2] (15+\x,10)  circle [radius=.2]  ((20+\x,10) circle [radius=.2] (20+\x,15)
circle [radius=.2]  (15+\x,15) circle [radius=.2]  (10+\x,15) circle [radius=.2]   (15+\x,15) circle [radius=.2]  (15+\x,20) circle [radius=.2]   (20+\x,15)circle [radius=.2] (25+\x,15)circle [radius=.2] ;

\node [red] at (4.5,4) {$2$};  \node [red]  at (10,4) {$5$};   \node [red]  at (15,4) {$2$}; \node [red] at (20.5,4) {$4$}; 
 \node [red]  at (9,10) {$3$};   \node [red]  at (15,9) {$5$}; \node [red] at (21,10) {$5$}; 
  \node [red]  at (9,15) {$3$};   \node [red]  at (14,14) {$7$};  \node [red] at (20,16) {$5$};  \node [red]  at (25.5,16) {$1$};  
   \node [red]  at (15,21) {$2$};
  
\node at (7,4) {$1$};  \node  at (13,4) {$1$};   \node at (18,4) {$1$};
\node at (9,7.5) {$2$};  \node  at (13,7) {$1$};   \node at (17,7) {$1$};  \node at (20.7,7.5) {$1$};  
\node at (12.15,9.2) {$1$};  \node  at (17.8,9.2) {$1$};
\node at (14.3,12.2) {$1$};  \node  at (17.1,11.8) {$1$}; \node at (20.7,12.5) {$1$};  
\node at (12.4,15.8) {$1$};  \node  at (17.5,15.8) {$3$}; \node at (22.5,15.8) {$1$};  
\node at (14.3,17.6) {$1$};

\draw [fill] (35,5) circle [radius=.4]; \node at (34,6) {$\scriptscriptstyle \frac{1}{2}$}; 
\node at (36,4) {$\scriptscriptstyle \frac{1}{2}$};  \node at (39,4) {$\scriptscriptstyle \frac{1}{5}$}; 

\node  at (41,4) {$\scriptscriptstyle \frac{1}{5}$};  \node at (44,4) {$\scriptscriptstyle \frac{1}{2}$};  \node at (46,4) {$\scriptscriptstyle \frac{1}{2}$}; 
 \node at (48.7,4) {$\scriptscriptstyle \frac{1}{4}$};  \draw [fill] (50,5) circle [radius=.4]; 
  \node at (51,4) {$\scriptscriptstyle \frac{1}{4}$};  \node at (50.7,6.3) {$\scriptscriptstyle \frac{1}{4}$};  \node at (48.9,7.1) {$\scriptscriptstyle \frac{1}{4}$};  
     
\node at (39.3,6.2) {$\scriptscriptstyle \frac{2}{5}$};  \node at (42.3,6.2) {$\scriptscriptstyle \frac{2}{5}$};
\node at (39.3,8.6) {$\scriptscriptstyle \frac{2}{3}$};  \node at (41.2,10.9) {$\scriptscriptstyle \frac{1}{3}$};  

\node at (43.9,8) {$\scriptscriptstyle \frac{1}{5}$};   \node at (43.7,10.9) {$\scriptscriptstyle \frac{1}{5}$}; \node at (45.9,8) {$\scriptscriptstyle \frac{1}{5}$}; 
 \node at (44.5,11.7) {$\scriptscriptstyle \frac{1}{5}$};  \node at (46.1,10.85) {$\scriptscriptstyle \frac{1}{5}$};

 \draw [fill] (50,10) circle [radius=.4];  \node at (51,10) {$\scriptscriptstyle \frac{1}{5}$};   \node at (50.7,8.5) {$\scriptscriptstyle \frac{1}{5}$}; 
  \node at (47.8,11.15) {$\scriptscriptstyle \frac{1}{5}$};   \node at (48.9,9.1) {$\scriptscriptstyle \frac{1}{5}$};    \node at (50.7,11.5) {$\scriptscriptstyle \frac{1}{5}$}; 
  
  \node at (44.5,13.4) {$\scriptscriptstyle \frac{1}{7}$};   \node at (43.7,15.9) {$\scriptscriptstyle \frac{1}{7}$}; \node at (47.2,13.8) {$\scriptscriptstyle \frac{1}{7}$}; 
 \node at (44.5,16.7) {$\scriptscriptstyle \frac{1}{7}$};  \node at (46.1,15.85) {$\scriptscriptstyle \frac{3}{7}$};

  \node at (50.7,13.8) {$\scriptscriptstyle \frac{1}{5}$}; 
 \node at (51.2,15.85) {$\scriptscriptstyle \frac{1}{5}$};  \node at (48.7,15.85) {$\scriptscriptstyle \frac{3}{5}$};  \node at (54,15.5) {$\scriptscriptstyle 1$};

 \draw [fill] (40,15) circle [radius=.4];   \node at (41.7,15.85) {$\scriptscriptstyle \frac{1}{3}$};  \node at (39,15) {$\scriptscriptstyle \frac{2}{3}$};   
 \draw [fill] (45,20) circle [radius=.4];    \node at (45,21.3) {$\scriptscriptstyle \frac{1}{2}$};  \node at (45.5,18.4) {$\scriptscriptstyle \frac{1}{2}$};    
    
    \end{tikzpicture}
\caption{A graph with weights ${\color{red} \pi},\mu$ ($\mu$ subordinated to $\pi$) and the resulting Markov kernel (with invariant measure $\pi$).  On the right, each edge $\{x,y\}$ carries two numbers, $K(x,y)$ and $K(y,x)$, with $K(x,y)$ written next to $x$. Large dots indicate non-zero holding and the holding value is indicated nearby.} \label{WM}
\end{center}
\end{figure}

With the geometric graph structure of Section~\ref{sec-JDW} fixed, we add vertex weights, $\pi(x)$ for each $x \in \mathfrak{X}$, and (positive) edge weights, $\mu_{xy}$ for each $\{x,y\} \in \mathfrak{E}$, with the requirement that $\mu$ is subordinated to $\pi$, i.e., $\sum_{y \in \mathfrak{X}} \mu_{xy} \leq \pi(x)$ (often, $\mu_{xy}$ is extended to all pairs by setting $\mu_{xy}=0$ when $\{x,y\}\not\in \mathfrak E$). Section \ref{sec-pimuK} explains how each choice of such weights defines a Markov chain and Dirichlet form adapted to the geometric structure $(\mathfrak X,\mathfrak E)$.  This is illustrated in Figure~\ref{WM} where the Markov kernel $K=K_\mu$ 
is obtained by seting $K(x,y) = \mu_{xy}/\pi(x)$ for $x \neq y$ and $K(x,x) = 1 - \sum_{y} \mu_{xy}/\pi(x)$. We will generally refer to the geometric structure of $(\mathfrak{X}, \mathfrak{E})$ with weights $(\pi,\mu)$ instead of the Markov chain.

Section \ref{sec-DMPN} introduces the important known concepts of volume doubling, moderate growth, various Poincar\'e inequalities, and Nash inequalities. These notions depend on the underlying structure $(\mathfrak X,\mathfrak E)$ and the weights $(\pi,\mu)$. There is a very large literature on volume doubling, Poincar\'e inequalities
and Nash inequalities in the context of harmonic analysis, global analysis and partial differential equations (see, e.g., \cite{GrigBook,LSCAspects} and the references therein for pointers to the literature) and analysis on countable graphs (see, \cite{BalrlowLMS,Coulhon,GrigoryanAMS,LSCStF}). The notion of moderate growth is from \cite{DSCmod,DSCNash} which also cover volume doubling and Poincar\'e and Nash inequalities in the context of finite Markov chains. 

Section \ref{sec-PQPJD} is one of the key technical sections of the article. Given an underlying structure $(\mathfrak{X},\mathfrak{E},\pi,\mu)$ which satisfies two basic assumptions---volume doubling and the ball Poincar\'e inequality---we prove a uniform Poincar\'e inequality for finite $\alpha$-John domains with a fixed $\alpha$. This relies heavily on the definition of a John domain and the use of Whitney coverings. Theorems \ref{th-P1} and~\ref{th-PP1} are the main statements in this section. Section \ref{sec-AW} provides an extension of the results of Section \ref{sec-PQPJD}, namely, Theorems \ref{th-PP1doub} and \ref{th-PP1controlled}. The line of reasoning for these results is adapted from \cite{Jerison,MLSC,LSCAspects} where earlier relevant references can be found (all these references treat PDE type situations). 

Section \ref{sec-Metro} illustrates the results of Section~\ref{sec-AW} in the classical context of the Metropolis-Hastings algorithm.  Specifically, given a finite John domain $U$ in a graph $(\mathfrak X,\mathfrak E)$, we can modify a simple random walk via edges weights in order to target a given probability distribution. Under certain hypotheses on the target distribution, Section \ref{sec-AW} provides useful tools to study the convergence of such chains.  We describe several examples in detail.

Section \ref{sec-Dir} deals with applications to absorbing Markov chains (or, equivalently for our purpose, chains killed at the boundary). We call such a chain Dirichlet-type by reference to the classical concept of Dirichlet boundary condition.  The section has two subsections. The first provides a very general discussion of the Doob transform technique for finite Markov chains. The second applies the results of Section \ref{sec-AW} to Dirichlet-type chains in John domains. The main results are Theorems \ref{th-Doob1}, \ref{th-Doob2}, and \ref{th-JZd}.  

Section \ref{sec-IU} introduces the notion of inner-uniform domain in the context of our underlying discrete space
$(\mathfrak X,\mathfrak E)$. Theorem \ref{theo-Carleson} captures a key property of the Perron-Frobenius eigenfunction $\phi_0$ in a finite inner-uniform domain. This key property is known as a {\em Carleson estimate}  
after Lennart Carleson. There is a vast literature regarding this estimate and its relation to the boundary Harnack principle in the context of potential theory in Euclidean domains (see, e.g., \cite{Ancona,BBB,Aik1,Aik2,Aik3} and the references and pointers given therein).

Corollary \ref{cor-Carleson2} is based on the Carleson estimate of Theorem \ref{theo-Carleson} and on Theorem \ref{th-Doob1}. It provides a sharp ergodicity result for Doob-transform chains in finite inner-uniform domains.  Section \ref{sec-cable} provides a proof of the Carleson estimate via transfer to 
the associated cable-process and Dirichlet form. Because the Carleson estimate is a deep and difficult result, it is nice to be able to obtain it from already known results. We use here a similar (and much more general)
version of the Carleson estimate in the context of local Dirichlet spaces developed in \cite{LierlLSCOsaka,LierlLSCJFA} following \cite{Aik1,Aik4} and \cite{Gyrya}. 
We apply to the eigenfunction $\phi_0$ the technique of passage from the discrete graph to the continuous cable space. This requires an interesting argument.  (See Proposition \ref{pro-*}.) Section~\ref{sec-HPWb} provides more refined results regarding the iterated kernels $K^t_U$ (chain killed at the boundary)
and $K^t_{\phi_0}$ (associated Doob-transform chain) in the form of two-sided bounds valid at all times and all space location in $U$. A key result is Corollary \ref{cor:exit-time-estimate}
which gives, for inner-uniform domains, a sharp two-sided bound on $\mathbf{P}_x(\tau_U>t)$, the probability that the process $(X_t)_{t>0}$ started at $x$ has not yet exited $U$ at time $t$. 

The final section, Section \ref{sec:examples}, describes several explicit examples in detail.

\section{John domains and Whitney coverings} \setcounter{equation}{0} \label{sec-JDW}
This section is concerned with notions of a purely geometric nature.
Our basic underlying structure can be described as a finite or countable set $\mathfrak X$ (vertex set) equipped with an edge set $\mathfrak E$ which, by definition,  is a set of pairs of  distinct vertices $\{x,y\}\subset \mathfrak X$. We write $x\sim y$ whenever $\{x,y\}\in \mathfrak E$ and say that these two points are neighbors.  By definition, a path is a finite sequence of points $\gamma=(x_0,\dots,x_m)$ such that $\{x_i,x_{i+1}\}\in \mathfrak E$, $0\le i <m$. We will always assume that $\mathfrak X$ is connected in the sense that, for any two points in $\mathfrak X$, there exists a finite path between them. 
The graph-distance  function $d$  assigns to any two points $x,y$ in $\mathfrak X$ the minimal length of  a path connecting $x$ to $y$, namely,
$$ d(x,y)=\inf \{m \ : \ \exists \ \gamma=(x_i)_0^m, \ x_0=x, \ x_m=y, \ \{x_i,x_{i+1}\}\in \mathfrak E\}.$$  
We set
$$B(x,r)=\{y:d(x,y)\le r).$$
This is the (closed) metric ball associated with the distance $d$. Note that the radius is a nonnegative real number and $B(x,r)=\{x\}$ for $r\in [0,1)$.
\begin{nota} Given a ball  $E=B(x,r)$ with specified center and radius and $\kappa>0$, let $\kappa E$  denote the ball $\kappa E=B(x,\kappa r)$.
\end{nota}
\begin{rem} We think of $\mathfrak E$ as producing a ``geometric structure'' on $\mathfrak X$. Note that loops are not allowed since the elements of $\mathfrak E$ are pairs, i.e., subsets of $\mathfrak X$ containing two distinct elements. This does not mean that the Markov chains we will consider are forbidden to have positive holding probability at some vertices. The example in the introduction, Figure~\ref{D0}, does have positive holding at some vertices (specifically, at $(N,q)$ for $1\le q\le N$) so the associated Markov chain is allowed to have loops even though the geometric structure does not.
\end{rem}

Let $U$ be a subset of $\mathfrak X$. By definition, the boundary of $U$ is 
$$\partial U=\{y\in \mathfrak X\setminus U: \exists\, x\in U \mbox{ such that } \{x,y\}\in \mathfrak E\}.$$ Note that this is the {\em exterior boundary} of $U$ in the sense that it sits outside of $U$. 
We say that $U$ is connected if, for any two points  $x,y$ in $U$, the exists a finite path $\gamma_{xy}=(x_0,x_1,\dots ,x_m)$ with $x_0 = x$ and $x_m = y$ such that $x_i\in U$ for~$0\le i\le m$. A domain $U$ is a connected subset of $\mathfrak X$. We will be interested here in finite domains. 

\begin{defin} Given a domain $U \subseteq \mathfrak X$, define the intrinsic distance $d_U$  by setting, for any $x,y\in U$,
$$d_U(x,y)=\inf\{ m \ : \ \exists (x_i)_0^m, \ x_0=x, \ x_m=y, \ \{x_i,x_{i+1}\}\in \mathfrak E, \ x_i\in U, \ 0\le i<m\}.$$
In words, $d_U(x,y)$ is the graph distance between $x$ and $y$ in the subgraph $(U, \mathfrak E_U)$ where $\mathfrak E_U=\mathfrak E\cap (U\times U)$. It is also sometimes called the inner distance (in $U$). Let $$B_U(x,r)=\{y\in U: d_U(x,y)\le r\}$$ be the (closed)  ball of radius $r$ around $x$ for the intrinsic distance $d_U$.
\end{defin}

In the example of Figure~\ref{D0}, we set 
$$\mathfrak X=\mathfrak X_N=\{(p,q): 0\le q\le p\le N\}.$$
The edge set $\mathfrak E=\mathfrak E_N$ is inherited from the square grid and $$U=U_N=\{(p,q): 0<q<p\le N\}.$$  It follows that the boundary $\partial U$ of $U$ (in $(\mathfrak X,\mathfrak E)$) is $$\partial U=\partial U_N=
\{(p,p), (p,0): 0\le p\le N\}.$$

\subsection{John domains} \label{sec-JD}

The following definition introduces a key geometric notion which is well known in the areas of harmonic analysis, geometry, and  partial differential equations.
\begin{defin}[John domain]   \label{def-John}
Given $\alpha, R>0$, we say that a  finite domain $U \subseteq \mathfrak X$,  equipped with a point $o\in U$, is in $J(\mathfrak X, \mathfrak E, o,\alpha, R)$
 if the domain $U$ has the property that for any point $x\in U$ there exists a path $\gamma_x=(x_0,\dots, x_m)$ of length $m_x$ contained in  $U$  such that $x_0=x$ and $x_m=o$, with $$ \max_{x\in U}\{m_x\}\le R \;\mbox{ and }  \;d(x_i, \mathfrak X\setminus U) \ge  \alpha (1+i),$$ for $0\le i\le m_x.$ When the context makes it clear what underlying structure $(\mathfrak X,\mathfrak E)$ is considered, we write
 $J(o,\alpha, R)$  for $J(\mathfrak X, \mathfrak E, o,\alpha, R)$.\end{defin}

We can think of a John domain $U$ as being one where any point $x$ is connected to the central point $o$ by a carrot-shaped region, which is entirely contained within $U$. The $x$ is the pointy end of the carrot and the point $o$ is the center of the round, fat end of the carrot. See Figure~\ref{Banana} for an illustration.

Within the lattice $\mathbb Z^d$, there are many examples of John domains: the lattice balls, the lattice cubes, and the intersection of Euclidean balls and Euclidean equilateral triangles with the lattice. See also Examples \ref{exa-T}, \ref{exa-mb}, and \ref{ex:convex} and Figure \ref{D3} below. Domains having large parts connected through narrow parts are not John. These examples, however, are much too simple to convey the subtlety and flexibility afforded by this definition.

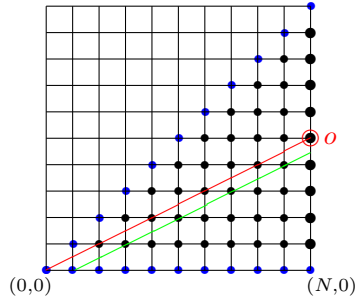
\begin{figure}[h]
\begin{center} 

\begin{picture}(200,140)(-50,-20)

{\color{blue} \multiput(7.6,4)(10,10){11}{\circle*{3}}}

{\color{blue} \multiput(4,4)(10,0){11}{\circle*{3} }}

\multiput(4,4)(10,0){11}{\line(0,1){100}}
\multiput(4,4)(0,10){11}{\line(1,0){100}}
\put(-4,-4){\makebox(4,4){$\scriptstyle (0,0)$}}  \put(110,-4){\makebox(4,4){$\scriptstyle (N,0)$}}

\multiput(24,14)(10,0){9}{\circle*{3}} \multiput(34,24)(10,0){8}{\circle*{3}} \multiput(44,34)(10,0){7}{\circle*{3}}
\multiput(54,44)(10,0){6}{\circle*{3}} \multiput(64,54)(10,0){5}{\circle*{3}} 
\multiput(74,64)(10,0){4}{\circle*{3}}
\multiput(84,74)(10,0){3}{\circle*{3}} \multiput(94,84)(10,0){2}{\circle*{3}} \put(104,94){\circle*{3}}
\multiput(104,14)(0,10){9}{\circle*{4}}

{\color{red}\put(104,54){\circle{6}} 
\put(4,4){\line(2,1){100}}
\put(109,52){\makebox{$o$}}
}

{\color{green}
	\put(11.6,4){\line(2,1){89}}
}

\end{picture}
\caption{The forty-five degree finite cone in $\mathbb Z^2$ with the ``center'' marked as a red $o$.} \label{D1}
\end{center}\end{figure}

\begin{defin}[$\alpha$-John domains] Given $(\mathfrak X,\mathfrak E)$, let
$J(\alpha)=J(\mathfrak X,\mathfrak E,\alpha)$ be the set of all domains $U\subset \mathfrak X$ which belong to $J(\mathfrak X,\mathfrak E, o,\alpha,R)$ for some fixed $o\in U$ and $R>0$.  A finite domain in $J(\alpha)$ is called an $\alpha$-John domain. \end{defin}

\begin{defin}[John center and John radius]For any domain $U\in J(\alpha)$, there is at least one pair $(o,R)$, with $o\in U$ and $R>0$, such that $U\in J(o,\alpha,R)$. Given such a John center $o$, let $R(U,o,\alpha)$ be the smallest $R$ such that $U\in J(o,\alpha, R)$.   
Assuming $\alpha$ is fixed, we call $R(U,o,\alpha)$ the John-radius of $U$ with respect to $o$. 
\end{defin}

\begin{rem} If we apply the second condition of Definition \ref{def-John} 
to any point in $U$ at distance $1$ from the boundary, we see that $\alpha\in (0,1]$.
\end{rem} 
\begin{rem} \label{rem:johnradius} Given $U\in J(\mathfrak X,\mathfrak E, \alpha)$, define the internal radius of $U$, viewed from~$o$, as
$$\rho_o(U)=\max\{d_U(o,x) :  x\in U\}.$$
Then, the John-radius $R(U,o,\alpha)$ is always greater than or equal to $\rho_o(U$), i.e., $R(U,o,\alpha) \geq \rho_o(U)$. Furthermore, we always have
$$  \min\{ d_U(o,z):z\in \mathfrak X\setminus U\}= d(o, \mathfrak X \setminus U)\ge \alpha (1+  R ( U,o, \alpha)),$$
which implies that $$\alpha (1+ R(U,o,\alpha))\le  1+\rho_o(U)\le 1+ R(U,o,\alpha).$$
In words, when $U\in J(\alpha)$ is not a singleton, the John-radius  of $U$ and  $\rho_o(U)$
are comparable, namely, $$ \frac{\alpha}{2} R(U,o,\alpha)\le \rho_o(U)\le R(U,o,\alpha).$$
We can also compare $\rho_o(U)$ to the diameter of the finite metric space $(U,d_U)$. Namely, we have
$$ \rho_o(U)\le \mbox{diam}(U,d_U)\le 2\rho_o(U).$$
\end{rem}
\begin{rem} \label{rem-MS}
 Let us compare this definition of a discrete John domain to the continuous version introduced in the classical reference \cite{MS}.  In \cite{MS}, a Euclidean domain $D$ is an $(\alpha,\beta)$-John domain (denoted $D\in \mathbf J(\alpha,\beta)$) if there exists a point $o\in D$ such that every $x\in D$ can be joined to $o$ by a rectifiable path $\gamma_x: [0,T_x]$ (paramatrized by arc-length) with $\gamma_x(0)=x$, $\gamma_x(T_x)=o$, $T_x\le \beta$ and $d_2(\gamma_x(t),\partial D)\ge \alpha (t/T_x)$ for $t\in [0,T_x]$.  (Here $d_2$ is the Euclidean distance.)  If one ignores the small modifications made in our  definition to account for the discrete graph structure,  the class $J(o, \alpha,R)$ is the analogue of the class $\mathbf J(\alpha R,  R)$ with an explicit center $o$.  The smallest $R$ such that $D$ belong to $\mathbf J(\alpha R,R)$ with a given center $o$ would be the analogue of our John-radius with respect to $o$.  
\end{rem} 

\begin{exa} \label{exa-T}Consider the example depicted in Figure~\ref{D1}.  
From the definition of John domain, one can see that it is best to choose $o$  far from the boundary. We pick $o=(N, \lfloor N/2\rfloor)$, depicted in red in Figure~\ref{D1}. For each point $x=(p,q)\in U$ we will define a (graph) geodesic path $\gamma_x$ joining $x$ to $o$ in $U$ that satisfies the conditions of a John domain. First, draw two straight lines $L$ and $L'$. The first line $L$, shown in red in Figure~\ref{D1}, joins $(0,0) $ to $(N,N/2)$. This is the line with equation $p-2q=0$ and the integer points on this line are at equal graph-distance  from the ``boundary lines'' $\{(p,q):  q=0, p = 0,1,\dots, N\}$  and $\{(p,p): p = 0,1,\dots, N\}$ as shown in blue in Figure~\ref{D1}. The line $L'$, shown in green, has the equation $p-2q = 1$. For any integer point $x=(p,q)$ on the line $L$, there is graph-geodesic path $\gamma_x$ joining $x$ to $o$ obtained by alternatively moving two steps right and one step up.  Similarly, for any integer point $x=(p,q)$ on the line  $L'$, there is a graph-geodesic  path $\gamma_x$ joining $x$ to $o$ by  moving right, then up, to reach a point $x'$ on $L$. From there, following $\gamma_{x'}$ to $o$. For any integer point $x$ in $U$ above $L$, define $\gamma_x$ by moving straight right until reaching an integer point $x'$ on $L$, then follow 
$\gamma_{x'}$ to $o$.  For those $x\in U$ below $L$, move straight up until reaching an integer point $x'$ on $L'$. From there, follow the path $\gamma_{x'}$ to $o$. 

Along any of the paths $\gamma_x=(x_0,\dots, x_m)$, with $x_0=x\in U$ and $x_m=o$,   $d(x_{i},\mathfrak X\setminus U)$ is non-increasing and $d(x_{3i},\mathfrak X\setminus U) \ge 1+ i$. It follows that  $d(x_{j},\mathfrak X\setminus U) \ge \frac{1}{3}(1+ i)$.  This proves that $U$ is a John domain with respect to $o$ with 
parameter $\alpha=1/3$ and John-radius $R(U,o, \frac{1}{3})= \rho_o(U)=N+[N/2] -3$. 
 \end{exa}

\begin{exa}[Metric balls]
\label{exa-mb}
  Any metric ball $U=B(o,R)$ is a 1-John domain, i.e.,  
$$B(o,R)\in J(\mathfrak X,\mathfrak E, o,1, R).$$
This is a straightforward but important example.  For each $x\in B(o,r)$, fix a path of minimal length $\gamma_x=(x_0=x,x_1,\dots, x_{m_x}=o)$, $m_x\le R$, joining $x$ to $o$ in $(\mathfrak X,\mathfrak E)$.  Then, $d(x_i, \mathfrak X\setminus B(o,R)) \ge  1+i$ because, otherwise, there would be 
a point $z \notin B(o,R)$ and at distance at most  $R$ from $o$, contradicting the definition of a ball.
\end{exa} 

\begin{exa}[Convex sets]\label{ex:convex}  In the classical theory of John domains in Euclidean space, convex sets provide basic examples. Round, convex sets have a good John constant ($\alpha$ close to $1$) whereas long, narrow ones have a bad John constant ($\alpha$ close to $0$). We will describe how this theory applies in the case of discrete convex sets, but first, let us review the continuous case. Here is how  the definition of Euclidean John domain given in \cite{MS} applies to Euclidean  convex sets. A Euclidean convex set $C$ belongs to $\mathbf J(\alpha,\beta)$ (see \cite[Definition 2.1]{MS} and Remark \ref{rem-MS} above) if and only if there exists $o\in C$ such that  
	$$B_2(o,\alpha) \subset C\subset B_2(o,\beta).$$
	Here the balls are Euclidean  balls and this is indicated by the subscript $2$, referencing the $d_2$ metric. 
	This condition is obviously necessary for $C\in \mathbf J(\alpha,\beta)$. To see that it is sufficient, observe that along the line-segment $\gamma_{xy}$ between any two points $x,y\in C$, parametrized by arc-length and of length $T$, the function $f(t)=d_2(\gamma_{xy}(t),C^c)$, defined on $[0,T]$, is concave (it is the minimum of the distances to the supporting 
	hyperplanes defining $C$).  Hence,  if we assume that $B(y,\alpha)\subset C$,
	either $d_2(y,C^c) < d_2(x,C^c)$ and then $d_2(\gamma_{xy}(t),C^c)\ge \alpha \ge \alpha \frac{t}{T}$, 
	or $d_2(y,C^c) \ge d_2(x,C^c)$  and 
	$$d_2(\gamma_{xy}(t),C^c) - d_2(x,C^c) \ge \frac{t}{T} (d_2(y,C^c) -d_2(x,C^c)) $$
	which gives
	$$d_2(\gamma_{xy}(t),C^c)  \ge  \alpha \frac{t}{T} + \left(1-\frac{t}{T}\right) d_2(x,C^c)
	\ge \alpha \frac{t}{T}.$$

To transition to discrete John domains,  we first consider the case of finite domains in $\mathbb Z^2$ because it is quite a bit simpler than the general case (compare \cite[Section 6]{DSCNash} and\cite{Virag}). In $\mathbb Z^2$, we can show that
any finite sub-domain $U$ of $\mathbb Z^2$ (this means we assume that $U$ is graph connected) obtained as the trace of a convex set $C$ such that $B_2(o,\alpha R)\subseteq C\subseteq B(o,R)$
for some $\alpha\in (0,1)$ and $R>0$
is a $\alpha'$-John domain with $\alpha'$ depending only on $\alpha$.

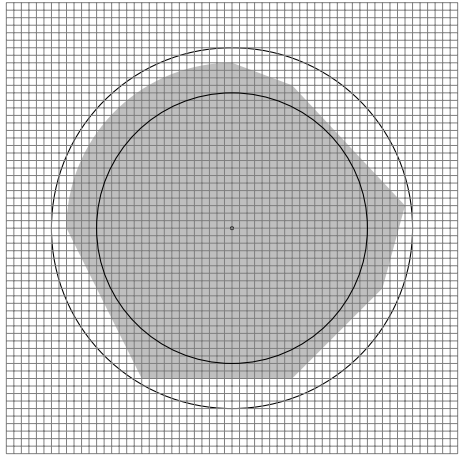
\begin{figure}[h] 
\begin{center}
\begin{tikzpicture}[scale=.1]

\path [fill=lightgray]  (8,30) arc [radius=22, start angle=180, end angle=90] (30,52)  --
 (38,49) -- (53,33) -- (50,22) --   (38,10) -- (18,10) -- (8,30);
 \draw [fill] (30,30)  circle  [radius=.2];
 \draw (30,30)  circle  [radius=24];    
  \draw [help lines] (0,0) grid (60,60);   
  \draw (30,30)  circle  [radius=18];   
    
   \end{tikzpicture}
\caption{A  finite discrete ``convex subset'' of $\mathbb Z^2$}
\label{fig-ConvZ2}
\end{center}
\end{figure}  

To deal with higher dimensional grids ($d>2$), let us adopt here the definition put forward by B\'alint Vir\'ag in \cite{Virag}: a subset $U$ of the square lattice $\mathbb Z^d$ is {\em convex} if and only if there exists a convex set $C\subset \mathbb R^d$ such that $U=\{x\in \mathbb Z^d:   d_\infty(x,C)\le 1/2\}$ where $d_\infty(x,y)=\max\{|x_i-y_i|: 1\le i\le d\}$.  The set $C$ is called a base for $U$.   We will use three distances on $\mathbb R^d$ and $\mathbb Z^d$: the max-distance $d_\infty$, the Euclidean $L^2$-distance $d_2(x,y)=\sqrt{\sum_1^d |x_i-y_i|^2}$ and the $L^1$-distance  $d_1(x,y)=\sum_1^d |x_i -y_i|$ which coincides with the graph distance on $\mathbb Z^d$.

In \cite{Virag}, B. Vir\'ag shows that, given a subset $U$ of $\mathbb Z^d$ that is convex in the sense explained above, for any two points $x,y\in U$, there is a discrete path $\gamma_{xy} =(z_0,\dots,z_m)$ in $U$ such that: (a) $z_0=x,z_m=y$; (b) $\gamma_{xy}$ is a discrete geodesic path in $\mathbb Z^d$; and (c) if $L_{xy}$  is the straight-line passing through $x$ and $y$ then each vertex $z_i$ on $\gamma_{xy}$ satisfies $d_\infty(z_i,L_{xy})<1$.  We will use this fact to prove the following proposition.
\begin{pro} \label{pro-Conv}Let $U\subset \mathbb Z^d$ be convex in the sense explained above, with base $C$.   
Suppose there is a point $o$ in $U$ and positive reals 
$\alpha, R$ such that  
\begin{equation} \label{eq-alphaConv}
B_2(o,\alpha R)\subset C \mbox{ and }C+B_\infty(0,1) \subset B_2(o,R),\end{equation} where $C+B_\infty(0,1) = \{y\in \mathbb R^d: d_\infty(y,C)\le 1\}$.
Then the set $U$ is in $J(o, \alpha', R')$ with 
$\alpha' = \alpha/(6d\sqrt{d}) $  and $\alpha R\le R' \le \sqrt{d}R $, where $d$ is the dimension of the underlying graph $\mathbb{Z}^d$.
\end{pro}  The dimensional constants in this statement are related to the use of three metrics, namely, $d_1,d_2$ and $d_\infty$.

\begin{rem} In practice, this definition is more flexible than it first appears because one can choose the base $C$. Moreover, once a certain finite domain $U$ is proved to be an $\alpha_0$-John-domain in $\mathbb Z^d$, it is easy to see that we are permitted to add and subtract in an arbitrary fashion lattice points that are at a fixed distance $r_0$ from the boundary $\partial U$ of $U$ in $\mathbb Z^d$,  as long as
we preserved connectivity. The cost is to change the John-parameter $\alpha_0$ to $\bar{\alpha}_0$
where $\bar{\alpha}_0$ depends only on $r_0$ and $\alpha_0$.
   \end{rem}

\begin{proof}[Proof of Proposition {\ref{pro-Conv}}]  The convexity of $C$ (together with that of the unit cube $B_\infty(0,1)$) implies the convexity of $C'=C+B_\infty(0,1)$.
Thus, by hypothesis, we know that the straight-line segments $l_x$ joining any point $x\in C'$ to $o$ that witness that $C'\in \mathbf J(\alpha R,R)$.
For $x\in U$,  the construction in \cite{Virag} provides a discrete geodesic path $\gamma_x=(x_0,\dots, x_{m_x})$ (of length $m_x$) in $\mathbb Z^d$ joining $x$ to $o$ within the set $U$ and which stays at most $d_\infty$-distance $1$ from $l_x$.  As usual, we parametrize $l_x$ by arc-length so that $l_x(0)=x$, $l_x(T)=o$, $T=T_x$.  For each point $x_i\in \gamma_x$, we pick a point $z_i$ on $l_x$ such that $d_\infty(x_i,z_i)=d_\infty(x_i,l_x)<1$ and define $t_i\in [0,T]$ by $z_i=l_x(t_i)$.  For each $x \in U$,
$$d_1(x,o) \le \sqrt{d} d_2(x,o) \le \sqrt{d} R .$$
To obtain a lower bound on  $d_1(x_i,\mathbb Z^d \setminus U)$, observe that
$$ d_1(x_i, \mathbb Z^d \setminus U) \ge  d_1(x_i, \mathbb{R}^d\setminus C)  $$
because   $C$ is contained in  $U+ B_{\infty}(0,\frac{1}{2})$. 
By definition of $C'$,  
$d_1(x_i,\mathbb R^d\setminus C)\ge d_1(x_i, \mathbb R^d\setminus C')- d.$ Hence, we have 
$$d_1(x_i, \mathbb Z^d\setminus U)\ge  d_2(x_i,\mathbb R^d\setminus C')-d.$$
Recall that  $z_i=l_x(t_i)$ is on the line-segment from $x$ to $o$ and at $d_\infty$-distance less than $1$ from $x_i$. Further, we know that
$$d_2(z_i,\mathbb R^d\setminus C')\ge \alpha t_i $$
because $C'$ is convex and $B_2(o,\alpha R)\subset C'\subset B_2(o,R)$.   Also, we have
$$t_i=d_2(x,z_i) \ge d_2(x,x_i) -d_2(x_1,z-i) \ge  \frac{d_1(x,x_i)}{\sqrt{d}} - \sqrt{d} = \frac{i-d}{\sqrt{d}}.$$  
Putting these estimates together gives
$$d_1(x_i, \mathbb Z^d\setminus U)\ge  \frac{\alpha}{\sqrt{d}}\left(i - d-\frac{d\sqrt{d}}{\alpha}\right).$$
Since, by construction, $d_1(x_i, \mathbb Z^d\setminus U)\ge 1$ for all $i$, it follows from the previous estimate that,
$$d_1(x_i, \mathbb Z^d\setminus U)\ge  \frac{\alpha}{6d\sqrt{d}}\left(1+ i \right)$$
 for all $0\le i\le m_x$.
\end{proof}
\end{exa}

Convexity is certainly not necessary for a family of  connected subsets of $\mathbb Z^d$ to be  $\alpha$-John domains with a  uniform $\alpha\in (0,1)$. 
Figure~\ref{D3} gives  an example of such a family that is far from convex in any sense. If we denote by $U_N$ the set depicted for a given $N$ and 
let $o_N=(\lfloor 2N/3 \rfloor,\lfloor N/2\rfloor)$ the chosen central point, then there are positive reals $\alpha,c,C$, independent of $N$, such that $U_N$ is a $J(o_N,\alpha, R)$  with $cN\le R\le CN$.   Figure~\ref{D4} gives an example of a family of  sets that is NOT  uniformly in $J(\alpha)$, for any $\alpha>0$.  

\begin{figure}[h]
\begin{center}
\begin{picture}(200,180)(-20,-20)

{\color{red} \put(93,70){\circle{3}}}

{\color{blue}   \multiput(0,0)(10,0){15}{\circle*{4}}
\multiput(0,0)(0,10){15}{\circle*{4}}
\multiput(0,140)(10,0){15}{\circle*{4}}
\multiput(140,0)(0,10){14}{\circle*{4}}
\multiput(10,10)(10,10){6}{\circle*{4}}
\multiput (40,130)(0,-10){4}{\circle*{4}}  
\multiput (90,130)(0,-10){4}{\circle*{4}}  
\put(90,40){\circle*{4}}
}

\multiput(0,0)(10,0){15}{\line(0,1){140}}
\multiput(0,0)(0,10){15}{\line(1,0){140}}

\put(-6,-8){\makebox(4,4){$\scriptstyle (0,0)$}}
\put(-6,145){\makebox(4,4){$ {\scriptstyle (0,N)}$}}
\put(150,-6){\makebox(4,4){$\scriptstyle (N,0)$}}
\put(-27,57){\makebox(4,4){$\scriptstyle (0, \lfloor 3N/7\rfloor)$}}
\put(35,155){\makebox(4,4){$\scriptstyle (\lfloor N/3\rfloor,N)$}}
\put(-25,97){\makebox(4,4){$\scriptstyle (0,\lfloor N/3\rfloor)$}}
\put(85,155){\makebox(4,4){$\scriptstyle (\lfloor 2N/3\rfloor,N)$}}
\put(170,37){\makebox(4,4){$\scriptstyle (N, \lfloor N/3\rfloor)$}}
\put(170,67){\makebox(4,4){$\scriptstyle (N, \lfloor N/2\rfloor)$}}
\put(150,40){\vector(-1,0){6}}
\put(150,70){\vector(-1,0){6}}
\put(40,150){\vector(0,-1){6}}
\put(90,150){\vector(0,-1){6}}

\end{picture}
\caption{A non-convex example of John domain, with the boundary points indicated in blue, and center $o$ indicated in red.}\label{D3} 
\end{center}\end{figure}
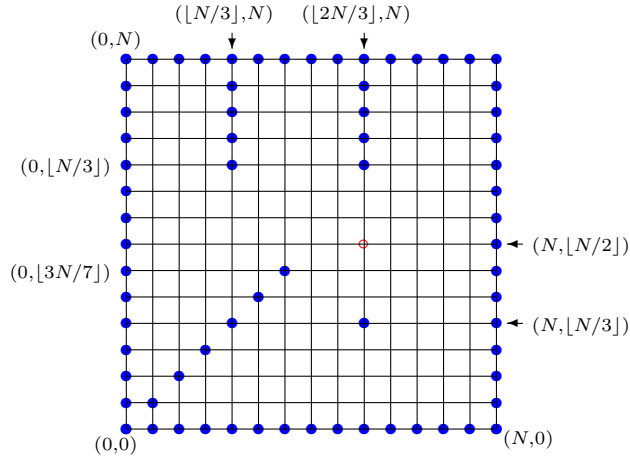

\begin{figure}[h]
\begin{center}
 \begin{picture}(200,200)(-20,-20)

{\color{blue}   \multiput(0,0)(10,0){15}{\circle*{4}}
\multiput(0,0)(0,10){15}{\circle*{4}}
\multiput(0,140)(10,0){15}{\circle*{4}}
\multiput(140,0)(0,10){14}{\circle*{4}}
\multiput(10,10)(10,10){11}{\circle*{4}}

}

\multiput(0,0)(10,0){15}{\line(0,1){140}}
\multiput(0,0)(0,10){15}{\line(1,0){140}}

\put(-6,-8){\makebox(4,4){$\scriptstyle (0,0)$}}
\put(-6,145){\makebox(4,4){$ {\scriptstyle (0,N)}$}}
\put(150,-6){\makebox(4,4){$\scriptstyle (N,0)$}}
\put(-30,110){\makebox(4,4){$\scriptstyle (0,N-\lfloor \sqrt{N}\rfloor )$}}

\end{picture}
\caption{A family of subsets that are \emph{not} uniformly John domains, with the boundary points indicated in blue. The passage between the top and bottom triangles is too narrow. } \label{D4}
\end{center}\end{figure}
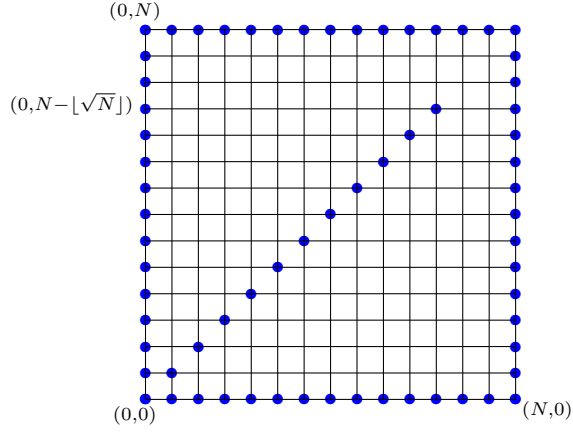

The following lemma shows that any inner-ball  $B_U(x,r)$ in a John domain contains  a ball from the original graph with roughly the same radius.
When the graph is equipped with a doubling measure (see Section \ref{sec-DMPN}),
this shows that the inner balls for the domain $U$ have volume comparable to that of the original balls. 
\begin{lem} Given $U \in J(\mathfrak{X},\mathfrak{E},\alpha)$, recall that $\rho_o(U)=\max\{d_U(o,x): x\in U\}$.
For any $x\in U$ and $r\in [0,2\rho_o(U)]$, there exists $x_r \in U$ such that $B(x_r, \alpha r/8) \subset B_U(x,r)$. For $r\ge 2\rho_o(U)$, we have $B_U(x,r)=U$. 
\end{lem}

\begin{proof} The statement concerning the case $r\ge 2\rho_o(U)$ is obvious. We consider three cases. First, consider the case when $o\in B_U(x,r/4)$ and $\rho_o(U)\le  r <2\rho_o(U)$. Then $B(o,\alpha\rho_o(U)/4) \subseteq B_U(x, r)$ and we can set $x_r=o$. Second, assume that  $o\in B_U(x,r/4)$ and $ r< \rho_o(U)$. Recall from Remark~\ref{rem:johnradius} that $d(o,\mathfrak X\setminus U)\ge \alpha(1+\rho_o(U))$.
It follows that  $B(o,\alpha r/8) \subset U$ and $B(o,\alpha r/8)\subset B_U(x,r)$. We can again set $x_r=o$.  Finally, assume that $o\not\in B_U(x,r/4)$. If $r<8$, we can take $x_r=x$. When $r\ge 8$,
let $\gamma_x=(z_0=x,z_1,\dots,z_m=o)$ be the John-path from $x$ to $o$ and let $x_r=z_i$, where $z_i$ is the first point on $\gamma_x$ such that
$z_{i+1}\not\in B_U(x, \lfloor r/4\rfloor )$. 
 By construction, we  have $d_U(x,x_r)\le \lfloor r/4\rfloor +1 \le r/2$, $i(x,r)\ge \lfloor r/4\rfloor $ and 
$$\delta(x_r)\ge \alpha(1+\lfloor r/4\rfloor)\ge \alpha r/4.$$  Therefore $B(x_r, \alpha r/8) \subset U$ and 
$B_U(x_r, \alpha r/8)=B(x_r,\alpha r/8)\subset B_U(x,r).$~\end{proof}

\subsection{Whitney coverings} \label{sec-WC}
Let $U$ be a finite domain in the underlying  graph $(\mathfrak X,\mathfrak E)$ (this graph may be finite or countable).
Fix a small parameter $\eta\in (0,1)$. For each point $x\in U$, let 
$$B^\eta_x =\{y \in U: d(x,y)\le \eta \delta(x)/4\}$$ be the ball centered at $x$ of radius $ r(x)=\eta \delta (x)/4 $ where 
$$\delta(x)= d(x,\mathfrak X\setminus U)$$
is the distance from $x$ to $\mathfrak X \setminus U$, the boundary of $U$ in $(\mathfrak X,\mathfrak E)$. The finite family $\mathcal F=\{B^\eta_x: x\in U\}$ forms a covering of $U$. Consider the set of all sub-families $\mathcal V$ of $\mathcal F$ with the property that the balls $B^\eta_x$ in $\mathcal V$ are pairwise disjoint. This is a partially ordered finite set and we pick a maximal element 
$$\mathcal W =\{B^\eta_{x_i}: 1\le i\le M\},$$
which, by definition, is a \emph{Whitney covering} of $U$. Note that the Whitney covering of $U$ is not a covering itself, but it generates a covering, because the triples of the balls in $\mathcal W$ are a covering of $U$. Because the balls in $U$ are disjoint, this is a relatively efficient covering.

The size $M$ of this covering will never appear in our computation and is introduced strictly for convenience. This integer $M$ depends on $U,s,\eta$ and on the particular choice made among all maximal elements in $\mathcal V$.

Whitney coverings are useful because they allow us to do manipulations on balls that form a covering---such as doubling their size---without leaving the domain $U$. Moreover, for any $k<4/\eta$, the closed ball
$\{y: d(x,y)\le k r(x)\}$ is entirely contained in $U$. 

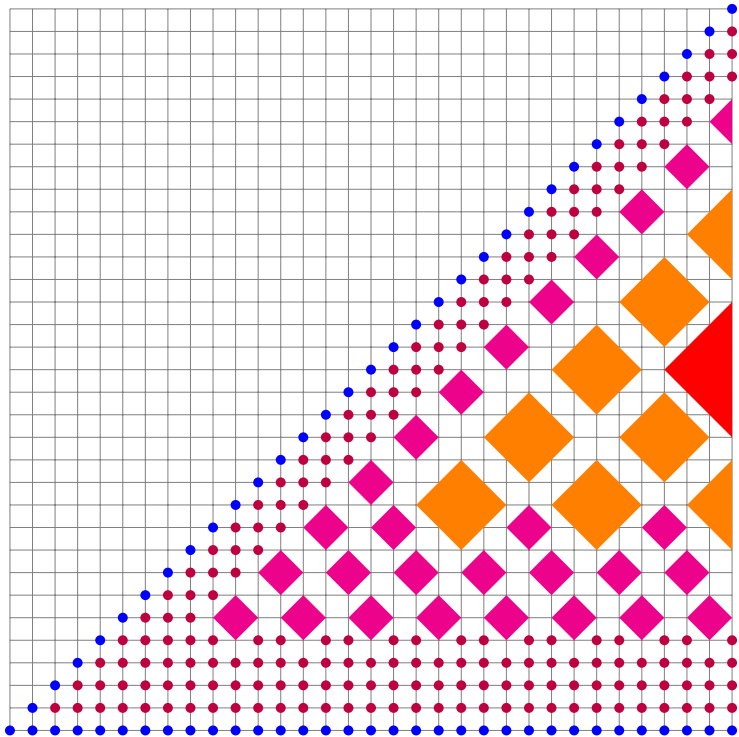
\begin{figure}[h]
	\begin{center}
		\begin{tikzpicture}[scale=0.3]
		
		\draw [help lines] (0,0) grid (32,32);
		
		\foreach \x in {0,1,...,32}
		\node at (\x,0) [circle,fill=blue] [scale=0.4] {};
		\foreach \x in {1,2,...,32}
		\node at (\x,\x) [circle,fill=blue] [scale=0.4] {};
		
		\fill[fill=red]
        (32,19)
     -- (32,13)
     -- (29,16);
		
		\foreach \x/\y in {32/10,32/22}
		{
		\fill[fill=orange]
        (\x,\y+2)
     -- (\x-2,\y)
     -- (\x,\y-2);
 		} 
		
		\foreach \x/\y in {32/27}
		{
 		\fill[fill=magenta]
        (\x,\y+1)
     -- (\x,\y-1)
     -- (\x-1,\y);
        }
        
		\foreach\x/\y in {23/13, 26/16, 26/10, 29/19, 29/13, 20/10}
		{
     	\fill[fill=orange]
        (\x-2,\y)
     -- (\x,\y-2)
     -- (\x+2,\y)
     -- (\x,\y+2);
 		}
		
		\foreach \x/\y in {29/9, 30/7, 27/7, 25/5, 24/7, 23/9, 28/5, 21/7, 22/5, 18/7, 15/7, 12/7, 19/5, 16/5, 13/5, 10/5, 14/9, 17/9, 16/11, 18/13, 20/15, 22/17, 24/19, 26/21, 28/23, 30/25, 31/5}
		{
	   	\fill[fill=magenta]
        (\x-1,\y)
     -- (\x,\y+1)
     -- (\x+1,\y)
     -- (\x,\y-1);
 		}
 		
 		\foreach \x in {2,3,...,32}
		\node at (\x,1) [circle,fill=purple] [scale=0.4] {};
		\foreach \x in {3,4,...,32}
		\node at (\x,2) [circle,fill=purple] [scale=0.4] {};
		\foreach \x in {4,5,...,32}
		\node at (\x,3) [circle,fill=purple] [scale=0.4] {};
		\foreach \x in {5,6,7,8,9,11,12,14,15,17,18,20,21,23,24,26,27,29,30,32}
		\node at (\x,4) [circle,fill=purple] [scale=0.4] {};
		\foreach \x in {6,7,...,32}
		\node at (\x,\x-1) [circle,fill=purple] [scale=0.4] {};
		\foreach \x in {7,8,...,32}
		\node at (\x,\x-2) [circle,fill=purple] [scale=0.4] {};	
		\foreach \x in {8,9,...,32}
		\node at (\x,\x-3) [circle,fill=purple] [scale=0.4] {};	
		
        \end{tikzpicture}
    \caption{A Whitney covering of the forty-five degree cone with $\eta = \frac{4}{5}$, where the boundary of the cone is indicated in blue. The color of each ball in the Whitney covering indicates its radius.}\label{fig:whitney-covering}
    \end{center}
\end{figure}

In the above (standard, discrete) version of the notion of Whitney covering, the largest balls are of size comparable to  
$\eta \max\{d(x,\mathfrak X\setminus U: x\in U\}$.
In the following $s$-version, $s\ge 1$, where $s$ is a (scale)  parameter, the size of the largest balls are at most $s$. 
Fix  $s\ge 1$ and a small parameter $\eta\in (0,1)$ as before. For each point $x\in U$, let 
$$B^{s,\eta}_x =\{y: d(x,y)\le \min\{s,\eta \delta(x)/4\}\}$$ be the ball centered at $x$ of radius $ r(x)=\min\{s,\eta \delta (x)/4\} $. Note as before that, for any $k<4/\eta$, the closed ball
$\{y: d(x,y)\le k \delta(x)/4\}$ is entirely contained in $U$.   
The finite family  $\mathcal F_s=\{B^{s,\eta}_x: x\in U\}$ form a covering of $U$. Consider the set of all sub-families $\mathcal V_s$ of $\mathcal F_s$ with the property that the balls $B^{s,\eta}_x$ in $\mathcal V_s$ are pairwise disjoint. These subfamilies form a partially ordered finite set and, just as we did with $\mathcal{F}$, we pick a maximal element 
$$\mathcal W_s =\{B^{s,\eta}_{x_i}: 1\le i\le M\},$$
which is the \emph{s-Whitney covering}. See Figure~\ref{fig:whitney-corner} for an example.

\begin{figure}
    \begin{center}
    \begin{subfigure}{0.45\textwidth}
    \includegraphics[scale=0.8]{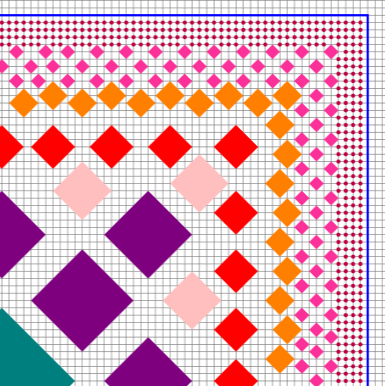}
    \end{subfigure}
	\hspace{0.6cm}
    \begin{subfigure}{0.45\textwidth}
    \includegraphics[scale=0.8]{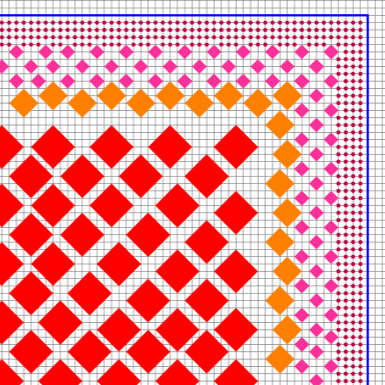}
    \end{subfigure}
    \end{center}
    \caption{The left figure shows a standard Whitney covering of the upper right corner of a square (boundary indicated in blue) with $\eta = \frac{4}{5}$. The right figure shows the same thing, but with the additional assumption that $s=3$, i.e., the maximum radius of a Whitney ball is 3.}\label{fig:whitney-corner} 
\end{figure}

As before, the size $M$ of this covering will never appear in our computations. It will be useful to split the family $\mathcal W_s$ into its two natural components, $\mathcal W_s=\mathcal W_{=s}\cup \mathcal W_{<s}$ where $\mathcal W_{=s}$ is the subset of $\mathcal W_s$ of those balls $B(x_i,r(x_i))$ such that  $r(x_i)=s$. 
\begin{rem}\label{rem-Whit} When the domain $U$ is finite (in a more general context, bounded) any Whitney covering ${\mathcal W}_s$ with parameter $s$ large enough, namely $$s\ge \eta (\rho_o(U)+1)/4,$$ is simply a Whitney covering $\mathcal W$ because $\min\{s,\eta\delta(x)/4\}=\eta \delta(x)/4$ for all $x\in U$. It follows that properties that hold for all $\mathcal W_s$, $s>0$, also hold for any standard Whitney coverings $\mathcal W$. \end{rem}

\begin{lem}[Properties of $\mathcal W_s$, $s\ge 1$] \label{lem-W} For any $s>0$,
the family $\mathcal W_s$ has the following properties.
\begin{enumerate} 
\item The balls $B^{s,\eta}_{x_i}=B(x_i,r(x_i))$, $1\le i\le M$,  are pairwise disjoint and  $$U=\bigcup_1^M B(x_i, 3r(x_i)).$$ In other words, the tripled balls cover $U$.
\item  For any $\rho \le 4/\eta$ and any $z\in B(x_i,  \rho r(x_i))$,  $$\delta(x_i) (1- \rho\eta /4)  \le  \delta (z) \le  (1+\rho\eta/4) \delta (x_i) $$
and  $$(1- \rho\eta /4) r(x_i) \le  r (z) \le  (1+\rho\eta/4) r(x_i).$$
 \item For any $\rho \le 2/\eta$, if the balls $B(x_i, \rho r(x_i))$ and $B(x_j, \rho r(x_j))$ intersect then  
$$ \frac{1}{3}\le \frac{1-\rho\eta/4}{1+\rho\eta/4}\le \frac{\delta(x_i)}{\delta(x_j)}\le \frac{ 1+\rho \eta/4}{1-\rho\eta/4}\le 3.$$
  \end{enumerate}
\end{lem}

\begin{proof}  We prove the first assertion. Consider a point $z\in U$. 
Since $\mathcal W_s$ is maximal, the ball $B^{s,\eta}_z$
intersects $\cup_1^M B(x_i, r(x_i))$. So there is an $i \in \{1,2,\dots,M\}$ and a $y\in B^{s,\eta}_{x_i}$ such that $y\in B^{s,\eta}_z$.  By the triangle inequality,  
$$\delta (x_i)\ge \delta(z) -  r (x_i)- r(z) \ge \delta(z)- \eta\delta(x_i)/4-\eta\delta(z)/4,$$
which yields,
$$ (1+\eta/4) \delta (x_i)\ge (1-\eta/4)\delta(z),$$
and hence,
$$(1+\eta/4) r (x_i)\ge (1-\eta/4) r (z).$$
It follows that  
$$d(x_i, z) \le  r(x_i)+r(z)\le  r (x_i) \left(1+\frac{ 1+\eta/4}{1-\eta/4}\right) \le \left(1+\frac{5}{3}\right)r(x_i) .$$
This contradicts the assumption that $z\not\in \cup_1^M B(x_i, 3r(x_i))$.  

The proofs of (2)-(3) follow the same line of reasoning.
\end{proof}

 \section{Doubling and moderate growth; Poincar\'e and Nash inequalities}    \setcounter{equation}{0} 
 \label{sec-DMPN}

 In this section, we fix a background graph structure $(\mathfrak X,\mathfrak E)$ and use $x \sim y$ to indicate that $\{x,y\} \in \mathfrak{E}$. As before, let $d(x,y)$ denote the graph distance between $x$ and $y$, and let
 $$B(x,r)=\{y: d(x,y)\le r\}$$ be the ball of radius $r$ around $x$. (Note that balls are not uniquely defined by their radius and center, i.e., it's possible that $B(x,r) = B(\tilde{x},\tilde{r})$ for $x \neq \tilde{x}$ and $r \neq \tilde{r}$.) In addition we will assume that $\mathfrak X$  is equipped with a measure $\pi$
 and, later,  that $\mathfrak E$ is equipped with an edge weight $\mu=(\mu_{xy})$ defining a Dirichlet form.
  
 \subsection{Doubling and moderate growth}
 Assume that  $\mathfrak X$ is equipped with a positive measure $\pi$, where $\pi(A)=\sum _{x\in A}\pi(x)$ for any finite subset $A$ of 
 $\mathfrak X$. (The total mass $\pi(\mathfrak X)$ may be finite or infinite.) Denote the volume of a ball with respect to $\pi$ as
 $$V(x,r)=\pi(B(x,r)).$$
 For any function $f$ and any ball $B$ we set
 $$f_B=\frac{1}{\pi(B)}\sum_Bf\pi.$$
 
 If $U$ is a finite subset of $\mathfrak X$, then let $\pi|_U$ be the restriction of $\pi$ to $U$, i.e., $\pi|_U(x) = \pi(x) \mathbf{1}_U(x)$. We often still call this measure $\pi$. Let $\pi_U$ be the probability measure on $U$ that is proportional to $\pi|_U$, i.e., $\pi_U(x) = \frac{\pi|_U(x)}{Z}$ where $Z = \sum_{y \in U} \pi|_U(y)$ is the normalizing constant.
 
 \begin{defin}[Doubling]\label{defn-doubling} We say that $\pi$ is doubling (with respect to $(\mathfrak X,\mathfrak E)$) if there exists a constant $D$ (the doubling constant) such that, for all $x\in \mathfrak X$ and $r>0$,
 $$V(x,2r)\le D V(x,r).$$
 \end{defin}
 This property has many implications. The proofs are left to the reader. 
 \begin{enumerate}
\item For any $x\sim y$,    $\pi(x)\le D\pi(y)$.
\item For any $x\in \mathfrak X$, $\#\{y: \{x,y\}\in \mathfrak E\}\le D^2$.
\item For any $x\in \mathfrak X, r\ge s>0$ and $y\in B(x,r)$,
$$\frac{V(x,r)}{V(y,s)}\le  D^2 \left(\frac{\max\{1,r\}}{\max\{1,s\}}\right)^{\log_2D}.$$
\end{enumerate} 
 
 We will need the following classic result for the case $p=2$. (For example, for the proofs of Theorems~\ref{th-P1} and~\ref{th-PP1}.) The complete proof is given here for the convenience of the reader. 
 \begin{pro}  \label{prop-sumballs}
 Let $(\mathfrak X,\mathfrak E, \pi)$ be doubling.  For any $p\in [1,\infty)$, any real number $t\ge 1$, any sequence of balls $B_i$, and any sequence of non-negative reals $a_i$, we have
 $$\left\| \sum_i a_i \mathbf 1_{tB_i}\right\|_p \le C \left\|\sum_i  a_i \mathbf 1_{B_i}\right\|_p,$$ where $C= 2(D^2p)^{1-1/p}D^{1+\log_2t}$ and $\|f\|_p=\left(\sum_{\mathfrak X} |f|^p \pi\right)^{1/p}$.
 \end{pro}
 \begin{rem} For $p=1$, the result is trivial since $\pi(tB)\le D^{\log_2(t)}\pi(B)$ for any ball $B$. 
 \end{rem}
 \begin{proof}  For any function $f$, consider the maximal function
$$Mf(x)=\sup_{B\ni x}\left\{ \frac{1}{\pi(B)}\sum_{ y\in B}|f(y)| \pi(y) \right\}.$$
By Lemma \ref{lem-Max} below, $\|Mf\|_q\le  C_q \|f\|_q$ for all $1<q\le +\infty$.
Also, for any ball $B$, $x\in B$ and function $h\ge 0$,  we have 
$$ \frac{1}{\pi(tB)}\sum_{y \in tB} h(y)\pi(y) \le(Mh)(x)$$
and thus
$$ \frac{1}{\pi(tB)}\sum_{y \in tB} h(y)\pi(y) \le \frac{1}{\pi(B)}\sum_B (Mh)(y)\pi(y).$$ 
Set $$f(y)= \sum_i a_i \mathbf 1_{tB_i}(y) \;\; \text{ and } \;\; g(y)= \sum_i  a_i \mathbf 1_{B_i}(y).$$
It suffices to prove that, for all functions $h\ge 0$, 
$|\sum f h\pi |\le C\|g\|_p\|h\|_q$,  where $1/p+1/q=1$.
Note that \begin{eqnarray*}
\sum_{y\in \mathfrak X} f(y) h(y)\pi(y) &=&\sum_i  a_i \sum_{y \in tB_i} h(y)\pi(y) \\
&\le & \sum_i  a_i \frac{\pi(tB_i) }{\pi(B_i)} \sum_{y \in B_i} (Mh)(y)   \pi \\
&\le &  D^{1+\log_2 t} \sum_i  a_i  \sum_{y \in B_i} (Mh)(y)   \pi(y) \\
&=&D^{1+log_2 t} \sum _{y \in \mathfrak X}\sum_i  a_i \mathbf 1_{B_i} (Mh)(y) \pi(y)  \\
&\le & D^{1+\log_2 t} \|g\|_p\|Mh\|_q \\
&\le &C_qD^{1+\log_2(t)} \|g\|_p\|h\|_q.
\end{eqnarray*}
Applying this fact with $h=f^{p/q}$ proves the desired result.
\end{proof}

\begin{lem}\label{lem-Max} 
For any $q\in (1,+\infty]$ and any $f$, the maximal function $M$ satisfies $\|Mf\|_q\le C_q\|f\|_q$ with $C_q=2 (D^{2}p)^{1-1/p}$ where $1/p+1/q=1$.
\end{lem}
\begin{proof}  Consider the set $V^f_\lambda=\{x: Mf(x)>\lambda\}$. By definition, for each $x\in V^f_{\lambda}$ there is a ball $B_x$ such that
$\frac{1}{\pi(B_x)}\sum_{B_x} |f|\pi> \lambda$.  Form 
$$\mathcal B=\{B_x: x\in V^f_\lambda\}$$ and extract from it a set of disjoint balls $B_1,\dots, B_q$ so that 
$B_1$ has the largest possible radius  among all balls in $\mathcal B$, $B_2$ has the largest possible radius among all balls in $\mathcal B$ which are disjoint from $B_1$.
At stage $i$, the ball $B_i$ is chosen to have  the largest possible radius among the balls $B_x$ which are disjoint from $B_1,\dots, B_{i-1}$. We stop when no such balls exist. 

We claim that the balls $3B_i$ cover $V^f_\lambda$, where $1\le i\le q$ and $q$ is the size of $\mathcal{B}$. For any $x\in V^f_\lambda$,  
we have $B_x=B(z,r)$, for some $z$ and $r$, and $ B(z,r) \cap \left(\cup_1^q B_i\right) \neq \emptyset$. 
By construction if $j$ is the first subscript such that there exists $y\in B(z,r)\cap  B_j$, $r$ must be no larger than the radius of $B_j$. This implies $z\in 2B_j$
and $x\in 3B_j$.   

It follows from the fact that $3B_i$ cover $V^f_{\lambda}$ that 
$$\pi(V^f_\lambda)\le D^{2} \sum_1^q\pi(B_i) \le D^2\lambda^{-1} \sum_1^q \sum_{B_i}|f|\pi \le D^2 \lambda ^{-1}\sum_{\mathfrak{X}} |f|\pi.$$
Next observe that  $ Mf \le M(f\mathbf 1_{\{|f|>\lambda/2\}})+ \lambda/2$ and thus
$$\{x: Mf(x)> \lambda\}\subset \{x: M(f\mathbf 1_{\{|f|>\lambda/2\}})(x)>\lambda/2\}.$$ 
Therefore
$\pi (Mf >\lambda)\le  2D^2\lambda^{-1} \sum_{\{f>\lambda/2\}} |f|\pi $.
Finally, recall that
$$\|h\|^q_q=  q\int_0^\infty \pi(h>\lambda) \lambda^{q-1}d\lambda.$$
This gives 
$$\|Mf\|_q^q \le 2qD^2 \sum \int_0^{2|f|}\lambda^{q-2}d\lambda  |f| \pi = \frac{qD^2 2^{q}}{q-1} \sum |f|^q\pi.$$ 
This gives $C_q= 2D^{2/q} \left(\frac{1}{1-1/q}\right)^{1/q}$. If $1/p+1/q=1$ then $C_q= 2(D^2p)^{1-1/p}$.
\end{proof}

The following notion of moderate growth is key to our approach. It was introduced in~\cite{DSCmod} for groups and in~\cite{DSCNash} for more general finite Markov chains. The reader will find many examples there.
It is used below repeatedly, in particular, in Lemma \ref{lem-Q} and Theorems \ref{th-KNU}-\ref{th-Ktilde1}-\ref{th-Ktilde2}, and in Theorems \ref{th-Doob1}-\ref{th-Doob2}-\ref{th-JZd}.

\begin{defin} Assume that $\mathfrak X$ is finite.  We say that $(\mathfrak X,\mathfrak E, \pi)$ has $a,\nu$-moderate volume growth if the volume of balls
satisfies 
$$\forall\, r\in (0,\mbox{diam}],\;\;\frac{V(x,r)}{\pi(\mathfrak X)}\ge a \left(\frac{1+r}{\mbox{diam}}\right)^\nu,$$
where $\mbox{diam} = \sup\{|\gamma_{xy}| : x,y \in \mathfrak{X}\}$ is the maximum of path lengths $|\gamma_{xy}|$ with $\gamma_{xy}$ the shortest path between $x,y \in \mathfrak{X}.$
\end{defin}

\begin{rem} When $\mathfrak X$ is finite and $\pi$ is $D$-doubling then $(\mathfrak X,\mathfrak E, \pi)$ has $((D)^{-2},\log_2D)$-moderate growth because
$$\frac{V(x,s)}{\pi(\mathfrak X)}=\frac{V(x,s)}{V(x,\mbox{diam})}\ge  D^{-1}\left( \frac{\max\{1,s\}}{\mbox{diam}}\right) ^{\log_2D} \ge
 D^{-2}\left( \frac{1+s}{\mbox{diam}}\right) ^{\log_2D}.$$
\end{rem}
Because of this remark, moderate growth can be seen as a generalization of the doubling condition. It implies
that the size of $\mathfrak X$ (as measured by $\pi(\mathfrak X)$) is bounded by a power of the diameter (this can be viewed as a ``finite dimension'' condition and a rough upper bound on volume growth). It also implies that the measure of small balls grows fast enough: $V(x,s)\ge a\pi(X)(\mbox{diam})^{-\log_2D}(1+s)^\nu.$

 \subsection{Edge-weight, associated Markov chains and Dirichlet forms} \label{sec-pimuK}
 This section introduces symmetric edge-weights $\mu_{xy}=\mu_{yx}\ge 0$ and the associated quadratic form
$$\mathcal E_\mu(f,g)= \frac{1}{2}\sum_{x,y} (f(x)-f(y))(g(x)-g(y))\mu_{xy}.$$
\begin{defin}\label{defn-adapted-elliptic-subor}
\begin{enumerate}
    \item 
We say the edge-weight $\mu =(\mu_{xy})_{ x\neq y\in \mathfrak X}$,  is adapted to~$\mathfrak E$ if 
$$\mu_{xy}>0  \mbox{ if and only } \{x,y\}\in \mathfrak E.$$ 
\item We say that the edge-weight $\mu=(\mu_{xy})_{ x\neq y\in \mathfrak X}$ is elliptic with constant $P_e\in (0,\infty)$ with respect to $(\mathfrak X,\mathfrak E,\pi)$ 
if 
$$\forall \ \{x,y\}\in \mathfrak E,\;\; P_e\mu_{xy}\ge \pi(x).$$ 
\item We say  that the edge-weight $\mu =(\mu_{xy})_{ x,y\in \mathfrak X}$ is subordinated to $\pi$ on $\mathfrak X$ if 
$$\forall\,x\in \mathfrak X,\;\; \sum_{y\in \mathfrak X}\mu_{xy}\le \pi(x).$$ 
\end{enumerate}
\end{defin}
\begin{rem} An adapted edge-weight $\mu$ is always such that $\mu_{xy}=0$ if $\{x,y\}\not\in \mathfrak E$, so the definition of adapted edge-weight means that $\mu$ is carried by the edge set $\mathfrak{E}$ in a qualitative sense. Ellipticity makes this quantitative in the sense that
$\mu_{xy}\ge P_e^{-1}\pi(x)$. Note that, with this definition, the smaller the ellipticity constant, the better. \end{rem}
\begin{rem}\label{rem:ellip-equiv} Since $\mu_{xy}=\mu_{yx}$, the ellipticity condition is equivalent to  $$P_e\mu_{xy}\ge  \pi(y)$$
 and also to $P_e \mu_{xy}\ge \max\{\pi(x),\pi(y)\}$.
\end{rem}
The condition   $\sum_y\mu_{xy} <+\infty$   implies immediately that the quadratic form  $\mathcal E_\mu$  defined on finitely supported functions is closable with dense domain in $L^2(\pi)$.  In that case, the data $(\mathfrak X,\pi,\mu)$ defines a continuous time Markov process on the state space $\mathfrak X$, reversible with respect to the measure $\pi$. This Markov process is the process associated to the Dirichlet form 
obtained by closing $\mathcal E_\mu$ in $L^2(\pi)$ and to the associated self-adjoint semigroup $H_t$.  See, e.g., \cite[Example 1.2.4]{FOT}.

\begin{defin} Assume the the edge-weight $\mu$ is subordinated to $\pi$, i.e.,
$$\forall\,x\in \mathfrak X,\;\; \sum_y\mu_{xy}\le \pi(x).$$
Set
\begin{equation} \label{def-Ksub}
K_\mu(x,y)=\left\{\begin{array}{ccl}  \mu_{xy}/\pi(x)  & \mbox{ for  } x\neq y,\\
1-(\sum_{y}\mu_{xy}/\pi(x)) & \mbox{ for } x=y. 
\end{array}\right. \end{equation}
\end{defin}
Note that the condition that $\mu$ is subordinated to $\pi$ is  necessary and sufficient for the semigroup $H_t$ to be of the form $H_t=e^{-t(I-K)} $ where $K$ is a Markov kernel on $\mathfrak X$. Indeed, we then have $K=K_\mu$. 
This Markov kernel is always reversible with respect to $\pi$.  
Of course, if we replace the condition $\sum_y\mu_{xy}\le \pi(x)$ by the weaker condition  $\sum_y\mu_{xy}\le A\pi(x)$ for some finite $A$,
then $H_t=e^{-At(I-K_{A^{-1}\mu})}$ where $A^{-1}\mu$ is the weight $(A^{-1}\mu_{xy})_{x,y\in \mathfrak X}$.
    
\subsection{Poincar\'e inequalities}

\begin{defin}[Ball Poincar\'e Inequality]\label{defn-ball-poincare} We say that $(\mathfrak X,\mathfrak E,\pi,\mu)$ satisfies 
the ball Poincar\'e inequality with parameter $\theta$  
 if there exists a constant $P$ (the  Poincar\'e constant) such that, for all $x\in \mathfrak X$ and $r>0$,
 $$\sum_{z\in B(x,r)} |f(z)-f_B|^2\pi(z) \le P r^\theta \sum_{z,y\in B(x,r), z\sim y}|f(z)-f(y)|^2 \mu_{zy}.$$\end{defin}
 
 \begin{rem} Under the doubling property, ellipticity is somewhat related to the Poincar\'e inequality on balls of small radius.  Whenever the ball of radius $1$ around a point $x$ is a star (i.e., there are no neighboring relations between the neighbors of $x$ as, for instance, in a square grid) the ball Poincar\'e inequality  with constant $P$ implies easily  that, at such point $x$ and for any $y\sim x$,
$$\pi(y) \le  P D^2 \mu_{xy}.$$ 
To see this, fix $y\in B(x,1)$ and apply the Poincar\'e inequality on $B(x,1)$ to the test function
defined on $B(x,1)$ by 
$$f(x) = \begin{cases} -c & \text{ if } x \neq y \\ 1 & \text{ if } x=y, \end{cases}$$
where $c=\pi(y)/(\pi(B(x,1))-\pi(y))$ so that the mean of $f$ over $B(x,1)$ is $0$. Recall that $B(x,1)$ is assumed to be a star and note that $0\le c\le \pi(y)/\pi(x)\le D$ where $D\ge 1$ is the doubling constant. 
This yields 
$$ \pi(y)\le
P(1-c)^2\mu_{xy}
\le P D^2 \mu_{xy}. $$
Hence, when all balls of radius $1$ are stars then the ball Poincar\'e inequality with constant $P$ implies ellipticity with constant $P_e=D^2P$. (See Remark~\ref{rem:ellip-equiv}.) However, when it is not the case that all balls of radius $1$ are stars then the ball Poincar\'e inequality does not necessarily imply ellipticity.
\end{rem}

\begin{figure}[t] 
\begin{center} 
\setlength{\unitlength}{0.036cm}
\begin{picture}(250,300)

\put(180,0){
\put(60,0){
\multiput(0,0)(10,10){3}{
\multiput(0,0)(5,5){3}{\circle*{2}}
\put(0,0){\line(1,1){10}}
\put(0,10){\circle*{2}}\put(0,10){\line(1,-1){10}}
\put(10,0){\circle*{2}}}

\put(0,20){
\multiput(0,0)(5,5){3}{\circle*{2}}
\put(0,0){\line(1,1){10}}
\put(0,10){\circle*{2}}\put(0,10){\line(1,-1){10}}
\put(10,0){\circle*{2}}}
\put(20,0){
\multiput(0,0)(5,5){3}{\circle*{2}}
\put(0,0){\line(1,1){10}}
\put(0,10){\circle*{2}}\put(0,10){\line(1,-1){10}}
\put(10,0){\circle*{2}}}}

\put(0,60){\multiput(0,0)(10,10){3}{
\multiput(0,0)(5,5){3}{\circle*{2}}
\put(0,0){\line(1,1){10}}
\put(0,10){\circle*{2}}\put(0,10){\line(1,-1){10}}
\put(10,0){\circle*{2}}}

\put(0,20){
\multiput(0,0)(5,5){3}{\circle*{2}}
\put(0,0){\line(1,1){10}}
\put(0,10){\circle*{2}}\put(0,10){\line(1,-1){10}}
\put(10,0){\circle*{2}}}
\put(20,0){
\multiput(0,0)(5,5){3}{\circle*{2}}
\put(0,0){\line(1,1){10}}
\put(0,10){\circle*{2}}\put(0,10){\line(1,-1){10}}
\put(10,0){\circle*{2}}}}

\multiput(0,0)(30,30){3}{
\multiput(0,0)(10,10){3}{
\multiput(0,0)(5,5){3}{\circle*{2}}
\put(0,0){\line(1,1){10}}
\put(0,10){\circle*{2}}\put(0,10){\line(1,-1){10}}
\put(10,0){\circle*{2}}}

\put(0,20){
\multiput(0,0)(5,5){3}{\circle*{2}}
\put(0,0){\line(1,1){10}}
\put(0,10){\circle*{2}}\put(0,10){\line(1,-1){10}}
\put(10,0){\circle*{2}}}
\put(20,0){
\multiput(0,0)(5,5){3}{\circle*{2}}
\put(0,0){\line(1,1){10}}
\put(0,10){\circle*{2}}\put(0,10){\line(1,-1){10}}
\put(10,0){\circle*{2}}}}}
\put(0,180){
\put(60,0){
\multiput(0,0)(10,10){3}{
\multiput(0,0)(5,5){3}{\circle*{2}}
\put(0,0){\line(1,1){10}}
\put(0,10){\circle*{2}}\put(0,10){\line(1,-1){10}}
\put(10,0){\circle*{2}}}

\put(0,20){
\multiput(0,0)(5,5){3}{\circle*{2}}
\put(0,0){\line(1,1){10}}
\put(0,10){\circle*{2}}\put(0,10){\line(1,-1){10}}
\put(10,0){\circle*{2}}}
\put(20,0){
\multiput(0,0)(5,5){3}{\circle*{2}}
\put(0,0){\line(1,1){10}}
\put(0,10){\circle*{2}}\put(0,10){\line(1,-1){10}}
\put(10,0){\circle*{2}}}}

\put(0,60){\multiput(0,0)(10,10){3}{
\multiput(0,0)(5,5){3}{\circle*{2}}
\put(0,0){\line(1,1){10}}
\put(0,10){\circle*{2}}\put(0,10){\line(1,-1){10}}
\put(10,0){\circle*{2}}}

\put(0,20){
\multiput(0,0)(5,5){3}{\circle*{2}}
\put(0,0){\line(1,1){10}}
\put(0,10){\circle*{2}}\put(0,10){\line(1,-1){10}}
\put(10,0){\circle*{2}}}
\put(20,0){
\multiput(0,0)(5,5){3}{\circle*{2}}
\put(0,0){\line(1,1){10}}
\put(0,10){\circle*{2}}\put(0,10){\line(1,-1){10}}
\put(10,0){\circle*{2}}}}

\multiput(0,0)(30,30){3}{
\multiput(0,0)(10,10){3}{
\multiput(0,0)(5,5){3}{\circle*{2}}
\put(0,0){\line(1,1){10}}
\put(0,10){\circle*{2}}\put(0,10){\line(1,-1){10}}
\put(10,0){\circle*{2}}}

\put(0,20){
\multiput(0,0)(5,5){3}{\circle*{2}}
\put(0,0){\line(1,1){10}}
\put(0,10){\circle*{2}}\put(0,10){\line(1,-1){10}}
\put(10,0){\circle*{2}}}
\put(20,0){
\multiput(0,0)(5,5){3}{\circle*{2}}
\put(0,0){\line(1,1){10}}
\put(0,10){\circle*{2}}\put(0,10){\line(1,-1){10}}
\put(10,0){\circle*{2}}}}}
\multiput(0,0)(90,90){3}{
\put(60,0){
\multiput(0,0)(10,10){3}{
\multiput(0,0)(5,5){3}{\circle*{2}}
\put(0,0){\line(1,1){10}}
\put(0,10){\circle*{2}}\put(0,10){\line(1,-1){10}}
\put(10,0){\circle*{2}}}

\put(0,20){
\multiput(0,0)(5,5){3}{\circle*{2}}
\put(0,0){\line(1,1){10}}
\put(0,10){\circle*{2}}\put(0,10){\line(1,-1){10}}
\put(10,0){\circle*{2}}}
\put(20,0){
\multiput(0,0)(5,5){3}{\circle*{2}}
\put(0,0){\line(1,1){10}}
\put(0,10){\circle*{2}}\put(0,10){\line(1,-1){10}}
\put(10,0){\circle*{2}}}}

\put(0,60){\multiput(0,0)(10,10){3}{
\multiput(0,0)(5,5){3}{\circle*{2}}
\put(0,0){\line(1,1){10}}
\put(0,10){\circle*{2}}\put(0,10){\line(1,-1){10}}
\put(10,0){\circle*{2}}}

\put(0,20){
\multiput(0,0)(5,5){3}{\circle*{2}}
\put(0,0){\line(1,1){10}}
\put(0,10){\circle*{2}}\put(0,10){\line(1,-1){10}}
\put(10,0){\circle*{2}}}
\put(20,0){
\multiput(0,0)(5,5){3}{\circle*{2}}
\put(0,0){\line(1,1){10}}
\put(0,10){\circle*{2}}\put(0,10){\line(1,-1){10}}
\put(10,0){\circle*{2}}}}

\multiput(0,0)(30,30){3}{
\multiput(0,0)(10,10){3}{
\multiput(0,0)(5,5){3}{\circle*{2}}
\put(0,0){\line(1,1){10}}
\put(0,10){\circle*{2}}\put(0,10){\line(1,-1){10}}
\put(10,0){\circle*{2}}}

\put(0,20){
\multiput(0,0)(5,5){3}{\circle*{2}}
\put(0,0){\line(1,1){10}}
\put(0,10){\circle*{2}}\put(0,10){\line(1,-1){10}}
\put(10,0){\circle*{2}}}
\put(20,0){
\multiput(0,0)(5,5){3}{\circle*{2}}
\put(0,0){\line(1,1){10}}
\put(0,10){\circle*{2}}\put(0,10){\line(1,-1){10}}
\put(10,0){\circle*{2}}}}}
\end{picture}
\caption{A finite piece of the Vicsek graph, an  infinite graph which is both a tree and a fractal graph, has volume growth of type $r^d$ with $d=\log 5/\log 3$ and satisfies the Poincar\'e inequality on balls with parameter $\theta=1+d=1+\log 5/\log 3$.}   \label{fig-V2}
\end{center}
\end{figure}
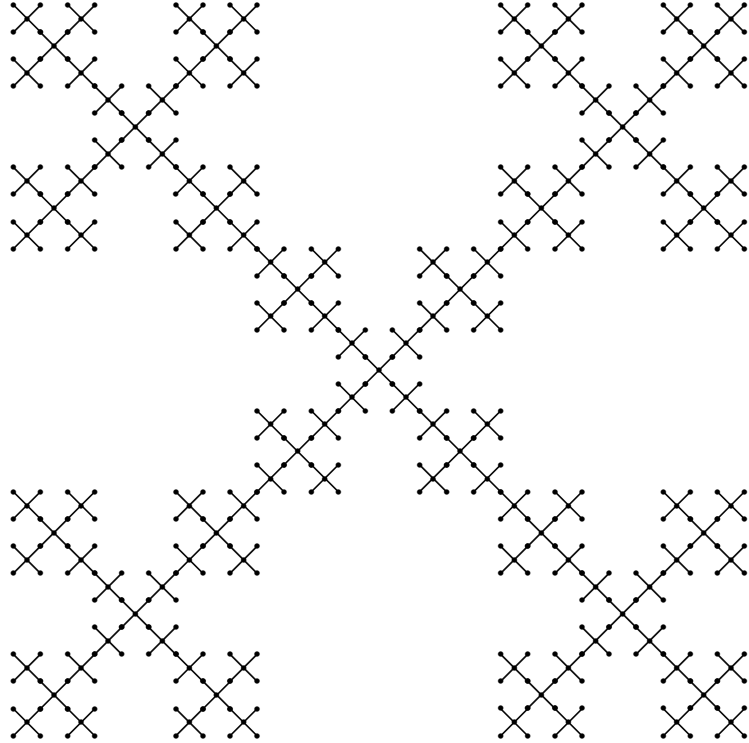

 \begin{defin}[Classical Poincar\'e inequality] A finite subset $U$ of $\mathfrak X$, equipped with the restrictions of $\pi$ and $\mu$ to $U$ and $\mathfrak E\cap (U\times U)$ satisfies the (Neumann-type) Poincar\'e inequality  with constant $P(U)$  if and only if, for any function $f$ defined on $U$,
 $$\sum_{U}|f(x)-f_U|^2\pi(x) \le P(U) \mathcal E_{\mu,U}(f,f)$$
where 
$$\mathcal E_{\mu,U}(f,g)= \frac{1}{2} \sum_{x,y\in U} (f(x)-f(y))(g(x)-g(y)) \mu_{xy}$$ 
and $f_U=\pi(U)^{-1}\sum_U f\pi$.  
 \end{defin} 
 
 \begin{exa}  Assume that $\mathfrak X$ is finite and that $(\mathfrak X,\mathfrak E,\pi,\mu)$ satisfies the ball Poincar\'e inequality
  with parameter $\theta$. Then, taking $r=\mbox{diam}(\mathfrak X)$ implies that $\mathfrak X$ satisfies the Poincar\'e inequality with constant
 $P(\mathfrak X)= 2P \mbox{diam}(\mathfrak X)^\theta$.
 \end{exa}
 
  \begin{defin}[$\mathcal Q$-Poincar\'e Inequality]  
Let $\mathcal Q=\{Q(x,r): x \in \mathfrak X, r>0\}$ be a given collection of finite subsets of $\mathfrak X$.
 We say that $(\mathfrak X,\mathfrak E,\pi,\mu)$ satisfies 
the $\mathcal Q$-Poincar\'e inequality with parameter $\theta$ if there exists a constant $P$ such that for any function $f$ with finite support and $r>0$, 
 $$\sum_x |f(x)-Q_rf(x)|^2\pi(x)\le P r^\theta \mathcal E_\mu(f,f)$$
 where $$\mathcal{E}_{\mu}(f,g) = \frac{1}{2}\sum_{x,y \in \mathfrak{X}} (f(x) - f(y))(g(x) - g(y))\mu_{xy}$$
 and $Q_rf(x) =\pi(Q(x,r))^{-1}\sum_{y\in Q(x,r)}f(y)\pi(y)$.
 \end{defin}
 
 The notion of $\mathcal Q$-Poincar\'e inequality is tailored to make it a useful tool to prove the Nash inequalities discussed in the next subsection. We can think of $Q_rf$ as a regularized version of $f$ at scale $r$. The $\mathcal Q$-Poincar\'e inequality provides control (in $L^2$-norm) of the difference $f-Q_rf$. If $\mathfrak X$ is finite and there is an $R>0$ such that $Q(x,R)=\mathfrak X$ for all $x$ then
 $Q_Rf(x)$ is the $\pi$ average of $f$ over $\mathfrak X$ and the $\mathcal Q$-Poincar\'e inequality at level $R$ becomes a classical Poincar\'e inequality as defined above.
 
 \begin{exa} The typical example of a collection $\mathcal Q$ is the collection of all balls $B(x,r)$. In that case, $Q_rf(x)=f_r(x)$ is simply the average of $f$ over 
 $B(x,r)$. In this case, the $\mathcal Q$-Poincar\'e inequality is often called a pseudo-Poincar\'e inequality.  Furthermore, if $(\mathfrak X,\mathfrak E,\pi,\mu)$ 
 satisfies the doubling property and the ball Poincar\'e inequality then it automatically satisfies the pseudo-Poincar\'e inequality. 
 \end{exa}
 
 \subsection{Nash inequality}
 
 Nash inequalities (in $\mathbb R^n$) were introduced in a famous 1958 paper of John Nash as a tool to capture the basic decay of the heat kernel over time. Later, they where used by many authors for a similar purpose in the contexts of Markov semigroups and Markov chains on countable graphs.
 Nash inequalities where first used in the context of finite Markov chains
in \cite{DSCNash}, a paper to which we refer for a more detailed introduction. 

 Assume that $(\mathfrak X,\mathfrak E)$ is equipped with a measure $\pi$ and an edge-weight $\mu$. 
The following is a variant of \cite[Theorem 5.2]{DSCNash}. The proof is the same. 
\begin{pro} \label{pro-Nash1}
Assume that there is a family of operators defined on finitely supported functions on $\mathfrak X$, $Q_s$ (with $0\le s \le T$) such that
$$\|Q_sf\|_\infty \le  M (1+s)^{- \nu} \|f\|_1$$
for some $\nu \geq 0$ and   that the edge weight $\mu=(\mu_{x,y})$ is such that  
$$ \|f-Q_sf\|_2^2\le P s^\theta \mathcal E_{\mu}(f,f) .$$
then the Nash inequality
$$\|f\|_2^{2(1+\theta/\nu)}\le C\left[\mathcal E_{\mu}(f,f)+\frac{1}{PT^\theta }\|f\|_2^2\right]\|f\|_1^{2\theta/\nu}$$
holds with $C= (1+\frac{\theta}{2\nu})^2(1+\frac{2\nu}{\theta})^{\theta/\nu}M^{\theta/\nu} P$.
\end{pro}
\begin{rem} When $$Q_rf(x)=\pi(Q(x,r))^{-1}\sum_{y\in Q(x,r)}f(y)\pi(y)$$ as in the definition of the $\mathcal Q$-Poincar\'e inequality,
the first assumption, $$\|Q_sf\|_\infty \le  M (1+s)^{- \nu} \|f\|_1,$$ amounts to a lower bound on the volume of the set $Q(x,r)$. In that case, the second assumption
is just the requirement that the $\mathcal Q$-Poincar\'e inequality is satisfied.
\end{rem}
For the next statement, we assume that $\mu$ is subordinated to $\pi$, i.e., for all $x$,  $\sum_y\mu_{xy}\le \pi(x)$. We consider the Markov kernel $K$ defined 
at (\ref{def-Ksub}) for which $\pi$ is a reversible measure and whose associated Dirichlet form on $L^2(\pi)$ is $\mathcal E_\mu(f,f)=\langle (I-K)f,f\rangle_\pi$. 

\begin{pro}[{\cite[Corollary 3.1]{DSCNash}}]  \label{pro-Nash2}
Assume that $\mu$ is subordinated to $\pi$ and that
$$\forall \,f\in L^2(\pi),\;\;\|f\|_2^{2(1+\theta/\nu)}\le C\left[\mathcal E_{\mu}(f,f)+\frac{1}{N}\|f\|_2^2\right]\|f\|_1^{2\theta/\nu}.$$
Then, for all $0\le n\le 2N$,
$$\sup_{x,y}\left\{K^{2n}(x,y)/\pi(y)\right\} =\sup_{x}\left\{K^{2n}(x,x)/\pi(x)\right\}\le 2\left(\frac{8C(1+\nu/\theta)}{n+1}\right)^{\nu/\theta}.$$
\end{pro}
 This proposition demonstrates how the Nash inequality provides some control on the decay of the iterated kernel of the Markov chain driven by $K$ over time.

\section{Poincar\'e and $\mathcal Q$-Poincar\'e inequalities for John domains}  \setcounter{equation}{0} 
\label{sec-PQPJD}
 
This is a key section of this article as well as one of the most technical. Assuming that $(\mathfrak X,\mathfrak E,\pi,\mu)$  is adapted, elliptic, and 
satisfies the doubling property and the ball Poincar\'e inequality with parameter $\theta$, we derive both a Poincar\'e inequality (Theorem \ref{th-P1}) and a $\mathcal Q$-Poincar\'e inequality (Theorem \ref{th-PP1}) on finite John domains. 
The statement of the Poincar\'e inequality can be described informally as follows: for a finite domain $U$ in $J(\alpha)$  we have,
for all functions $f$ defined on $U$,
$$\sum_{U}|f(x)-f_U|^2\pi(x) \le CR^\theta \mathcal E_{\mu,U}(f,f)$$
where $R$ is the John radius for $U$ and
$C$ depends only on $\alpha$ and the constants, coming from doubling, the Poincar\'e inequality on balls, and ellipticity, which describe the basic properties of $(\mathfrak X,\mathfrak E,\pi,\mu)$. (Instead of $R$, one can use the intrinsic diameter of $U$ because they are comparable up to a multiplicative constant depending only on $\alpha$, see Remark \ref{rem:johnradius}.) We give an explicit description of the constant $C$ without trying to optimize what can be obtained through the general argument. For many explicit examples running a similar argument while taking advantage of the feature of the example will lead to (much) improved estimates for $C$ in terms of the basic parameters. 

These results will be amplified in Section \ref{sec-AW} by showing that the same technique  works as well for a large class of weights which can be viewed as modifications of the pair $(\pi,\mu)$.
 
Throughout this section, we fix a finite domain $U$ in  $\mathfrak X$ with (exterior) boundary $\partial U$ such that $U\in J( o,\alpha,R)$ for some $o\in U$. 
We also fix  a witness family of John-paths $\gamma_x$ for each $x\in U$, joining $x$ to $o$ and fulfilling the $\alpha$-John domain condition.

\subsection{Poincar\'e inequality for John domains}

Fix a Whitney covering of $U\in J(o,\alpha,R)$,
$$\mathcal W=\{B_i= B^\eta_{x_i}=B(x_i, r_i): 1\le i\le Q\},$$
with $r_i=\eta \delta(x_i)/4$ and parameter $\eta<1/4$.
By construction, the collection of balls  $B'_i=3B_i=B(x_i,3r_i)$ covers $U$, and it is useful to set
$$\mathcal W'=\{3B_i:\; 1\le i\le Q\}.$$
Please note that we always think of the elements of $\mathcal W,\mathcal W'$ as balls, each with a specified center and radius, not just subsets. 
\begin{lem} \label{lem-Rad}
Any ball $E$ in $\mathcal W$ (i.e., $E = B_i$ for some $i$) has radius $r$ bounded  above by   $\eta (2R+1)/4$. 
\end{lem}
\begin{proof} By hypothesis, $U\in J(o,\alpha,R)$. Let $R_o=\delta(o)$. Any other point $x\in U$ is at distance at most $R$ from $o$.
It follows that $\delta (x)\le  R+ R_o\le 2R+1$. 
\end{proof}

Fix a ball $E_o$ in $\mathcal W$ such that $3E_o$ contains the point $o$.  For any $E=B(z,r) \in \mathcal W$, let $\gamma^E=\gamma_z$ be the John-path from $z$ to $o$ and
select a finite sequence 
\begin{equation}
\label{eq:whitney-chain}
    \mathcal W'(E)=(F^E_0,\dots,F^E_{q(E)})=(F_0,\dots,F_{q(E)})
\end{equation} of distinct balls 
 $F^E_i=F_i\in \mathcal W'$, for $0\le i\le q(E)$
such that   $F^E_0=3E$, $F^E_{q(E)}=3E_o$, $F^E_i$ intersects $\gamma^E$ and  $d(F^E_{i+1}, F^E_i)\le 1$, $0\le i\le q(E)-1$.  
This is possible since the balls in $\mathcal W'$ cover $U$. When the ball $E$ is fixed, we drop the superscript $E$ from the notation $F_i^E$. For each $E\in \mathcal W$, the sequence of balls
$3F_i$ (for $1\le i\le q(E)$) provides a chain of adjacent balls joining
$z$ to $o$ along the John-path $\gamma^E$. The union of the balls $6F^E_i$ form a carrot-shaped region joining $z$ to $o$ (thin at $z$ and wide at $o$). These families of balls are a key ingredient in the following arguments. See Figure~\ref{fig:chain} for an example.

\begin{figure}
\begin{center}
    \includegraphics[scale=0.8]{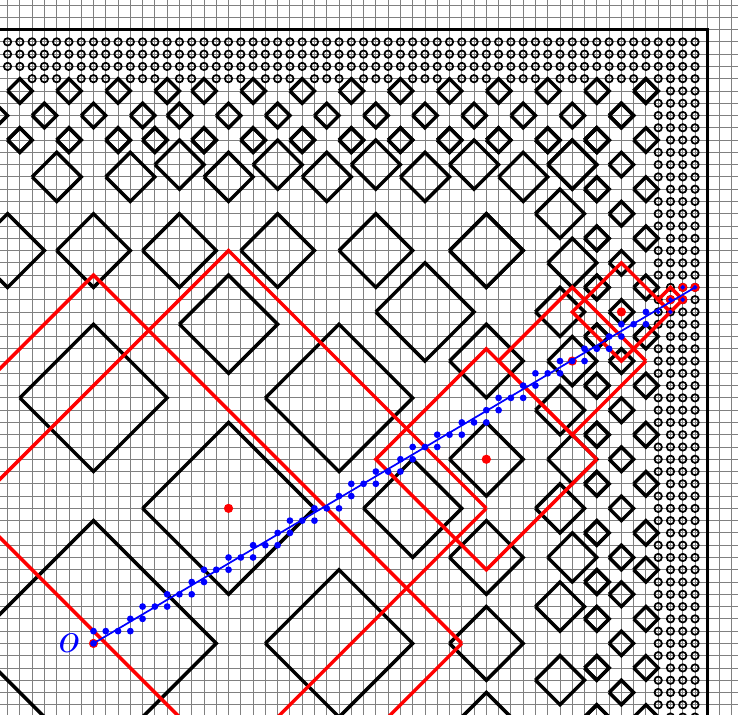}
\caption{A  chain of $9$ balls  $\mathcal W'(E) =\{3F_0, 3F_1, \dots,3F_{q(E)}\}$, $q(E)=8$, covering the path $\gamma_z$ from $o$ to $z$ (blue points staying close to the straight line from $o$ to $z$)
with $E=B(z,r)\in \mathcal W$, where $W$ is a Whitney covering of the corner of a square. The ball centers are in red. The Whitney parameter $\eta=4/5$. The initial Whitney ball $E$ has radius $1/5$ so $3F_{8}=E=\{z\}$. The ball $3F_7,3F_6$ are also singleton but $3F_5$ has radius $9/5$.  The ball $3F_0$ is centered at $o$ and has radius $30$.}
\label{fig:chain}
\end{center}
\end{figure}  

\begin{lem} \label{lem-WN}
Fix $\eta <1/4$ and $\rho\le 2/\eta$. 
The doubling property implies that any point $z\in U$ is contained in at most  $D^{1+\log_2(4\rho+3)}$  distinct balls of the form $\rho E$ with $E\in \mathcal W$, where $D$ is the volume doubling constant.
\end{lem}
\begin{rem}Note that this property  does not necessarily hold if $\rho$ is much larger than $2/\eta$. This lemma implies that $$\sum_{E \in \mathcal{W}} \chi_{\rho E} \leq D^{1+\log_2(4\rho + 3)}.$$
\end{rem}
\begin{proof} Suppose $z\in U$ is contained in $N$ balls $\rho E$ with $E\in \mathcal W$, and call them  $E_i=B(x_i,r_i)$, $1\le i\le N$.  By Lemma \ref{lem-W}(3),
the radii  $r_i$  satisfy     $r_i/r_j\le 3$ (this uses the inequality $\rho\le 2/\eta$) and it follows that $$\bigcup_1^N B(x_i,r_i) \subset B(x_j, (4\rho+ 3)r_j) .$$ Because the balls $E_i$ are disjoint, applying this inclusion with $j$ chosen so that $\pi(E_j)=\min\{\pi(E_i):1\le i\le N\}$ yields  $$N \pi(E_j)\le \pi((4\rho+3)E_j)\le D^{1+\log_2(4\rho+3)} \pi(E_j),$$
which, dividing by $\pi(E_j)$ proves the lemma.
\end{proof}

 \begin{lem} \label{lem-WN2}
Fix $\eta <1/4$ and $\rho\le 2/\eta$.  For any ball $E=B(x,r(x))\in \mathcal W$ and any ball $F=B(y,3r(y))\in \mathcal W'(E)$, where $\mathcal{W}'(E)$ is defined in~\eqref{eq:whitney-chain}, we have 
$E\subset  \kappa F$ with $\kappa = 7\alpha^{-1}\eta^{-1}$.
\end{lem}
\begin{proof} By construction, there is a point $z$ in $F$ on the John-path $\gamma^E$ from $x$ to $o$ and $\delta(z)\ge \alpha (1+d(z,x))$.
This implies 
$$4r(y)/\eta=\delta(y)\ge \delta(z)-3r(y)\ge  \alpha (1+d(z,x)) -3r(y),$$ that is, $((4/\eta)+3)r(y)\ge \alpha (1+d(x,z))$.  It follows that
$$ x\in B(y ,(3+\alpha^{-1} \eta^{-1}(4+3\eta)) r(y)).$$  Observe that
$$\delta(x)\le \delta(y) +d(x,y)\le 4\eta^{-1}r(y) + (3+ \alpha^{-1}\eta^{-1}(4+3\eta))r(y)$$ which gives
$$r(x)=\eta \delta(x)/4\le r(y)(1+ \alpha^{-1}(3\alpha \eta+ 4+3\eta)/4).$$  Then,
$$B(x,r(x))\subset B(y,d(x,y)+r(x)),$$
which gives 
$$B(x,r(x))\subset B(y, (4+\alpha^{-1}\eta^{-1}(4+3\eta+(3\alpha \eta^2
+4\eta+3\eta^2)/4)r(y)).$$ Because $\alpha \le 1$ and we assumed  $\eta<1/4$, we have
$$ 4+\alpha^{-1}\eta^{-1}(4+3\eta+(3\alpha \eta^2
+4\eta+3\eta^2)/4 \le 4+6\alpha^{-1}\eta^{-1}\le 7 \alpha^{-1}\eta^{-1},$$
and hence $B(x,r(x)) \subseteq \kappa B(y,2r(y))$ with $\kappa = 7\alpha^{-1}\eta^{-1}$.
\end{proof}
 \begin{lem}  \label{lem-WN3}
 Fix  $\eta\le 1/4$.  For each $E\in \mathcal W$, the sequence 
 $$\mathcal W'(E)=(F^E_0,\dots, F^E_{q(E)})$$ has the following properties. Recall that for each 
 $i\in \{0,\dots, q(E)\}$,  $F^E_i=B(z^E_i,\rho^E_i)$ with $\rho^E_i=3r^E_i= (3\eta/4) \delta(z^E_i)$  and that $F_0^E=3E$, $F^E_{q(E)}=3E_o$. (We drop the reference to $E$ when $E$ is clearly fixed.)
 \begin{enumerate}
 \item  For each $E$, when $\rho_i <1$ we have $B(z_i,\rho_i)=\{z_i\}$ and
 $$1+d(z_0,z_i) \le  4/(3\alpha \eta).$$
 
\item For each $E$ and $i\in \{1,\dots, q(E)-1\}$ such that  $\max\{\rho_i,\rho_{i+1}\} < 1$, we have
$$ |f_{F_i}-f_{F_{i+1}}|^2 =|f(z_{i})-f(z_{i+1})|^2 \le \frac{P_e}{\pi(z_i)} \sum_{z\sim z_i, z\in U} |f(z)-f(z_i)|^2\mu_{zz_i}.$$

\item  For  each $E$ and $i\in \{1,\dots, q(E)-1\}$ such that  $\max\{\rho_i,\rho_{i+1}\}\ge 1$, we have 
 $$|f_{F_i}-f_{F_{i+1}}|^2 \le 2D^6 P (8\rho_i)^\theta \frac{1}{\pi(F_i)} \sum_{x,y\in 8F_{i}, x\sim y}|f(x)-f(y)|^2\mu_{xy},$$
 for any function $f$ on $U$. \end{enumerate}
 \end{lem}

 \begin{proof} In the first statement we have $\rho_i=\rho(z_i)= 3\eta \delta(z_i)/4 <1$. Because  $U\in J(\alpha)$, $E_0=B(z_0,r_0)=E$ and  $z_i$ must be on $\gamma^E=\gamma_{z_0}$, 
 $$\delta(z_i)\ge \alpha (1+ d(z_i,z_0)).$$   It follows that $ 1+d(z_0,z_i)\le  4/(3\alpha \eta)$. 
 
 The second statement is clear. 
 
 For the third statement, we need some preparation. First we obtain the lower bound 
 $$\min\{\delta(z_i),\delta(z_{i+1})\}\ge \frac{5}{6\eta},$$
 based on the assumption that $\max \{\rho_i,\rho_{i+1}\}\ge 1$.
If both $\rho_i,\rho_{i+1}$ are
at least $1$, there nothing to prove. If 
is one of them is less than $1$, say $\rho_i<1$, then $F_i=B(z_i,\rho_i)=\{z_i\}$
and $d(z_i,F_{i+1})\le 1$.  It follows that
$$\frac{4}{3\eta}\le  \frac{4}{3\eta}\rho_{i+1}= \delta(z_{i+1}) \le 1+ \rho_{i+1}+\delta(z_i).$$ But $\rho_{i+1}= (3/4\eta) \delta(z_{i+1})$, so $$\left(1-\frac{3\eta}{4}\right) \delta(z_{i+1})\le 1+\delta(z_i)$$
and (using the fact that $\eta\le 1/4$)
$$ \frac{5}{6\eta}\le \frac{4}{3\eta}- 2\le \delta(z_i).$$ 
This shows that $\min\{\rho_i,\rho_{i+1}\} \geq \frac{5}{8}$ because $$\min\{\rho_i,\rho_{i+1}\} = \frac{3\eta}{4}\min\{\delta(z_i),\delta(z_{i+1})\} \ge \frac{3\eta}{4}\frac{5}{6\eta}=5/8.$$
 
Next, we show that  
$$F_i\cup F_{i+1}\subset 8F_i \cap 8F_{i+1}\subset U .$$  
By assumption, the balls $B(z_{j+1}, 6r_{j+1})$ and $B(z_j,6r_j)$ intersect. Applying Lemma \ref{lem-W}(3) with $\rho=6$ and $\eta \leq 1/4$ gives that
 $5/11\le r_{j+1}/r_j\le 11/5$  and it follows that
 $$   \max \left\{ \frac{\rho_{i+1}}{\rho_i}, \frac{\rho_{i}}{\rho_{i+1}}\right\}\le 11/5.$$
Moreover, because $d(F_i,F_{i+1})\le 1$, we have
 $$ \max\{d(z_i,z):z\in F_{i+1})\le \rho_i+ 2\rho_{i+1} +1\le  8\rho_i$$ and similarly, $$ \max\{d(z_{i+1},z):z\in F_{i} \} \le \rho_{i+1}+2\rho_i+1\le 8\rho_{i+1}.$$ 
 It follows that $F_i\cup F_{i+1}\subset 8F_i\cap 8F_{i+1}\subset U$. 
 Now, we are ready to prove the inequality stated in the lemma. Write
 \begin{eqnarray*} |f_{F_i}-f_{F_{i+1}}|^2 &=& \left|\frac{1}{\pi(F_i)\pi(F_{i+1})} \sum_{\xi\in F_i,\zeta\in F_{i+1}} [f(\xi)-f(\zeta)] \pi(\xi)\pi(\zeta) \right|^2\\
 &\le & \frac{1}{\pi(F_i)\pi(F_{i+1})} \sum_{\xi, \zeta\in 8F_i} |f(\xi)-f(\zeta)| ^2\pi(\xi)\pi(\zeta) \\
 &= & \frac{2\pi(8F_i)}{\pi(F_i)\pi(F_{i+1})} \sum_{\xi \in 8F_i} |f(\xi)-f_{8F_i}|^2 \pi(\xi)\\
 &\le & \frac{2 P \pi(8F_i)(8\rho_i)^\theta}{\pi(F_i)\pi(F_{i+1})} \sum_{x,y\in 8F_i} |f(x)-f(y)|^2\mu_{xy}  \\ 
 &\le & \frac{2 D^6 P (8\rho_i)^\theta}{\pi(F_i)} \sum_{x,y\in 8F_i} |f(x)-f(y)|^2\mu_{xy}    \end{eqnarray*} 
 \end{proof}

\begin{theo}  \label{th-P1}
Fix $\alpha, \theta, D,P,  >0$.  Assume that $(\mathfrak X,\mathfrak E,\pi,\mu)$ 
is adapted, elliptic, and satisfies the doubling property with constant $D$ and the ball Poincar\'e inequality with parameter $\theta$ and constant $P$. Assume that the finite  domain $U$ and the point $o\in U$ are such that $U\in J(o,\alpha,R)$, $R>0$. Then there exist a constant  $C$ depending only on $\alpha, \theta, D,P$ and such that
$$\sum_{U}|f(x)-f_U|^2\pi(x) \le P(U) \mathcal E_{\mu,U}(f,f)$$
with 
$$P(U) \le C R^\theta  \mbox{ with }   C= 4^{-\theta}2PD^5+16 D^{14+ 2\log(2\kappa)}  \max\{ R^{-\theta} P_e, 2^{\theta}2PD^6\} $$
where $\kappa = 84/\alpha$.  In particular, 
$$C\le 17 D^{30+2\log_2(1/\alpha)} \max\{R^{-\theta}P_e,2^\theta P\}.$$
\end{theo} 
\begin{proof}  We pick a Whitney covering with $\eta=1/12$.
Recall from Lemma \ref{lem-Rad} that all balls in $\mathcal W$ have radius at most $R/16$.
It suffices to bound $\sum_U|f-f_{3E_o}|^2\pi$ because 
 $$\sum_U|f-f_{U}|^2\pi=\min_{c}\left\{ \sum_U|f-c|^2\pi\right\}.$$ 
 The balls in $\mathcal W'$ cover $U$ hence $$\sum_U|f-f_{3E_o}|^2\pi \le \sum _{E\in W} \sum_{ 3E} |f-f_{3E_o}|^2\pi.$$  
Next, using the fact that $(a+b)^2\le 2(a^2+b^2)$, write
$$\sum _{E\in W} \sum_{3E}|f-f_{3E_o}|^2\pi \le 2\left(\sum _{E\in W} \sum_{3E}|f-f_{2E}|^2\pi\right)+ 2\sum _{E\in W} \pi(3E)|f_{3E}-f_{3E_o}|^2.$$
 
 We can bound and collect the first part of the right-hand side very easily because, using the Poincar\'e inequality in balls of radius at most $3R/16\le R/4$ and then Lemma~\ref{lem-WN}, we have
 \begin{eqnarray}
  \sum _{E\in W} \sum_{ 3E} |f-f_{3E}|^2\pi &\le & P (R/4)^\theta \sum_{E\in W} \sum_{\substack{x,y\in 3E \\ x\sim y}}|f(x)-f(y)|^2\mu_{xy}  \nonumber \\
& \le & PD^5 (R/4)^\theta  \mathcal E_{\mu,U}(f,f).  \label{crux1} \end{eqnarray}
 This reduces  the proof to bounding 
$$\sum_{E\in W} \pi(3E)|f_{3E}-f_{3E_o}|^2.$$
For this, we will use the  chain of balls  $\mathcal W'(E) =(F^E_0,\dots, F^E_{q(E)})$ to write
$$|f_{3E}-f_{3E_o}|\le \sum_{0}^{q(E)-1} |f_{F^E_i}-f_{F^E_{i+1}}|.$$
\begin{nota} For any function $f$ on $U$ and any ball $ F=B(x,\rho)\in \mathcal W'$ set
 $$G(F,f)=  \left(\frac{1}{\pi(F)} \sum_{x\in 8F\cap U}\sum_ {y\sim x, y\in U}|f(x)-f(y)|^2\mu_{xy}\right)^{1/2}.$$
 \end{nota}
With this notation, Lemma \ref{lem-WN3}(2)-(3)  yields 
$$ |f_{F^E_i}-f_{F^E_{i+1}}| \le Q R^{\theta/2}  G(F^E_i,f),$$
where $Q^2=\max\{ R^{-\theta} P_e, 2^{\theta}2PD^6\}$. With $\kappa $ as in Lemma \ref{lem-WN2}, this becomes
$$ |f_{3E}-f_{3E_o}|  \mathbf 1_E \le  QR^{\theta/2} \sum_{0}^{q(E)-1} G(F^E_i,f)  \mathbf 1_E \mathbf 1_{\kappa F^E_i}.$$
Write
\begin{eqnarray*}
\lefteqn{\sum_{E\in W} \pi(3E)|f_{3E}-f_{3E_o}|^2\le  D^2\sum_{E\in W}\sum_U |f_{3E}-f_{3E_o}|^2 \mathbf 1_{E}(x) \pi(x)}&& \\
&\le &Q^2D^2 R^\theta\sum_{E\in W}\sum_U \left|     \sum_{0}^{q(E)-1}    G(F^E_i,f) \mathbf 1_{\kappa F^E_i} (x)   \right|^2 \mathbf 1_{E}(x) \pi(x) \\
&\le &  Q^2D^2 R^\theta\sum_{E\in W}\sum_U \left|     \sum_{F\in \mathcal W'}   G(F,f) \mathbf 1_{\kappa F}  (x)  \right|^2 \mathbf 1_{E}(x) \pi(x) \\
&\le &  Q^2D^2 R^\theta \sum_{\mathfrak X} \left|     \sum_{F\in \mathcal W'}   G(F,f) \mathbf 1_{\kappa F}  (x)  \right|^2  \pi(x) \end{eqnarray*}
 where the last step follows from the observation that $\sum_{E\in \mathcal W}\mathbf 1_E\le 1$ because the balls in $\mathcal W$ are pairwise disjoint.

By Proposition \ref{prop-sumballs}  and the fact that the balls in $\mathcal W$ are disjoint, we have
\begin{eqnarray*} \lefteqn{
\sum_{\mathfrak X} \left|\sum_{E\in  \mathcal W} G(3E,f)  \mathbf 1_{3\kappa E} (x)  \right|^2  \pi(x) } &&\\
&\le &   8 D^{4+2\log_2(2\kappa)} \sum_{\mathfrak X} \left|\sum_{E\in  \mathcal W} G(3E,f)  \mathbf 1_{ E}  (x)  \right|^2  \pi(x) \\
&=& 8 D^{4+2\log_2(2\kappa)}  \sum_{E\in  \mathcal W} G(3E,f)^2 \pi(E) \\
&=&  8 D^{4+2\log_2(2\kappa)} \sum_{E\in  \mathcal W} \sum_{x,y\in 24 E\cap U}|f(x)-f(y)|^2\mu_{xy}
\end{eqnarray*}
 By Lemma \ref{lem-WN} (note that $2/\eta=24$),  for each $x\in \mathfrak X$, there are at most  $D^8$ balls $E$ in $\mathcal W$ such that  $24E$ contains $x$.
This yields
 $$\sum_{\mathfrak X} \left|\sum_{F\in  \mathcal W} G(2F,f)  \mathbf 1_{2\kappa F} (x)  \right|^2  \pi(x)   \le
 8 D^{12+ 2\log(2\kappa)} \mathcal E_{U,\mu} (f,f).$$
  Collecting all terms gives Theorem \ref{th-P1} as desired.
 \end{proof}

 \subsection{$\mathcal Q$-Poincar\'e inequality for John domains}
 
For any $s\ge 1$, fix  a scale-$s$ Whitney covering  $\mathcal W_s$ with Whitney parameter $\eta<1/4$. For our purpose, we can restrict ourselves to integer parameters $s$  no greater than  $2R+1$ which results in making only finitely many choice of coverings. Recall that $\mathcal W_s$ is the disjoint union of $\mathcal W_{=s}$ (balls of radius exactly $s$) and $\mathcal W_{<s}$ (balls of radius strictly less than $s$).  As before, we denote by $\mathcal W_{s}',
\mathcal W_{=s}'$ and $\mathcal W_{<s}'$, the sets of balls obtained by tripling the radius of the balls in $\mathcal W_{s},
\mathcal W_{=s}$ and $\mathcal W_{<s}$.

Fix a ball $E^s_o$ in $\mathcal W_s$ such that $3E^s_o$ contains the point $o$.  For any ${E=B(z,r) \in \mathcal W_s}$, select a finite sequence $$\mathcal W_s'(E)=(F^{s,E}_0,\dots,F^{s,E}_{q_s(E)})=(F_0,\dots,F_{q(E)})$$ of distinct balls 
 $F^{s,E}_i=F_i\in \mathcal W'_s$ (for $0\le i\le q_s(E)$)
such that   $F^{s,E}_0=3E$, $F^E_{q(E)}=3E^s_o$, $F^{s,E}_i$ intersects $\gamma^E$ and  $d(F^{s,E}_{i+1}, F^{s,E}_i)\le 1$ ($0\le i\le q_s(E)-1$).  
This is obviously possible since the balls in $\mathcal W'_s$ cover $U$. When the parameter $s$ and  the ball $E$ are fixed, we drop the supscripts $s,E$ from the notation $F_i^{s,E}$.  We only need a portion of this sequence, namely, 
\begin{equation}\label{chain}
W'_{<s}(E)=(F^{s,E}_0,\dots, F^{s,E}_{q^*_s(E)})\end{equation} where $q^*_s(E)$ is the smallest index $j$ such that  $r_j=s$. If no such $j$ exists, set $q^*_s(E)=q(E)$.  
For  future reference, we call these sequences of balls {\em local s-chains}. Namely, the sequence $W'_{<s}(E)$ is the local s-chain for $E$ at scale $s$.

We set
$$F(s,E)=F^{s,E}_{q^*_s(E)},$$
to be the last ball in the local s-chain of $E$. For each $x$, choose a ball $E(s,x)\in \mathcal W_s$ with maximal radius among those $E\in \mathcal W_s$ such that $3E$ contains $x$ and set 
$$F(s,x)=\left\{\begin{array}{cl} 3E(s,x) & \mbox{ when } x\in \bigcup_{E\in \mathcal W_{=s}}3E,\\
F(s,E(s,x)) & \text{otherwise.} \end{array}\right.$$
The ball $F(s,x)$ is, roughly speaking, chosen among those balls  of radius $3s$ in the Whitney covering that are not too far from $x$ and away from the boundary of $U$ --- for points $x$ near the boundary, where the Whitney balls have radius less than $s$, $F(s,x)$ is the last ball in the local s-chain of $E \in \mathcal{W}_s$, where $3E$ covers $x$.

\begin{defin}  \label{def-Q}  For $s\in [0,1]$, set $Q_s=I$ (i.e., $Q_sf=f$).
For any $s>1$, define the averaging operator 
$$Q_sf(x)=\sum_y Q_s(x,y)f(y)\pi(y)$$ by setting 
$$Q_s(x,y)= 
\frac{1}{\pi(F(s,x))}\mathbf 1_{F(s,x)}(y).$$
  \end{defin}
  
Next we collect the $s$-version of the  statements analogous  to Lemmas  \ref{lem-WN} and~\ref{lem-WN2}. The proofs are the same.
\begin{lem} \label{lem-WPP}Fix $\eta <1/4$ and $\rho\le 2/\eta$.
For any $s>0$, the following properties hold.
\begin{enumerate}
\item  
Any point $z\in U$ is contained in at most  $D^{1+\log_2(4\rho+3)}$  distinct balls  $\rho E$ with $E\in \mathcal W_s$.
\item 
For any ball $E=B(x,r(x))\in \mathcal W_{<s}$ and any ball ${F=B(y,3r(y))\in \mathcal W'_{<s}(E)}$ we have 
$E\subset  \kappa F$ with $\kappa = 7\alpha^{-1}\eta^{-1}$.
\end{enumerate}
\end{lem}

The $s$-version of Lemma \ref{lem-WN3} is as follows. The proof is the same.
 \begin{lem}  \label{lem-WPP3}
 Fix  $\eta\le 1/4$.  For each  $s\ge 1$, and $E\in \mathcal W_{<s}$, the sequence 
 $$\mathcal W'_{<s}(E)=(F^{s,E}_0,\dots, F^{s,E}_{q^*_s(E)})$$ has the following properties. Set
 $i\in \{0,\dots, q^*_s(E)\}$,  $F^{s,E}_i=B(z^{s,E}_i,\rho^{s,E}_i)$ with $\rho^{s,E}_i=3r^{s,E}_i= 3\min\{s,\eta \delta(z^E_i)/4\}$  and that $F_0^{s,E}=3E$, $F^{s,E}_{q^*_s(E)}=F(s,E)$. We drop the reference to $s$ and $E$ when they are clearly fixed.
 \begin{enumerate}
 \item  For each $E\in W_{<s}$, when $\rho_i <1$ we have $B(z_i,\rho_i)=\{z_i\}$ and
 $$1+d(z_0,z_i) \le  4/(3\alpha \eta).$$
 
\item For each $E\in W_{<s}$ and $i\in \{1,\dots, q^*_s(E)-1\}$ such that  $\max\{\rho_i,\rho_{i+1}\} < 1$, we have
$$ |f_{F_i}-f_{F_{i+1}}|^2 =|f(z_{i})-f(z_{i+1})|^2 \le \frac{P_e}{\pi(z_i)} \sum_{z\sim z_i, z\in U} |f(z)-f(z_i)|^2\mu_{zz_i}.$$

\item  For  each $E\in W_{<s}$ and $i\in \{1,\dots, q^*_s(E)-1\}$ such that  $\max\{\rho_i,\rho_{i+1}\}\ge 1$ we have 
  $\min\{\delta(z_i),\delta(z_{i+1})\}\ge 4/(9\eta) $, $\min\{\rho_i,\rho_{i+1}\}\ge 1/3$  and  $$F_i\cup F_{i+1}\subset 8F_i\subset U.$$  
 Furthermore, for  any function $f$ on $U$, 
 $$|f_{F_i}-f_{F_{i+1}}|^2 \le 2D^6 P (8\rho_i)^\theta \frac{1}{\pi(F_i)} \sum_{x,y\in 8F_{i}, x\sim y}|f(x)-f(y)|^2\mu_{xy}.$$  \end{enumerate}
 \end{lem}

 \begin{theo}  \label{th-PP1}
Fix $\alpha, \theta, D, P_e,P  >0$.  Assume that $(\mathfrak X,\mathfrak E,\pi,\mu)$ 
is adapted, elliptic and, satisfies the doubling property with constant $D$ and the ball Poincar\'e inequality with parameter $\theta$ and constant $P$. Assume that the finite  domain $U$ and the point $o\in U$ are such that $U\in J(o,\alpha,R)$, $R>0$. Then there exists a constant  $C$ depending only on $\alpha, \theta, D,P$ and such that
$$\forall\, s>0,\;\;\sum_{U}|f(x)-Q_sf(x)|^2\pi(x) \le C s^\theta \mathcal E_{\mu,U}(f,f)$$
with 
$$ C= 3^\theta 7 PD^5 +16 D^{14+ 2\log(2\kappa)}  \max\{P_e, 8^{\theta}2PD^6\}$$
where $\kappa =84/\alpha $.
\end{theo}  

\begin{proof}  The conclusion trivially holds when $s\in [0,1]$ because $Q_sf=f$ in this case. For $s>1$,
as in the proof of Theorem \ref{th-P1}, we pick a Whitney covering with $\eta<1/12$. 
We need to bound
\begin{eqnarray*}\sum_x|f(x)-Q_sf(x)|^2\pi(x)&=&\sum_{E\in \mathcal W_s}\sum_{\substack{x\in 3E \\ E=E(s,x)}}|f(x) -f_{F(s,E)}|^2\pi(x)\\
&=& \sum_{E\in \mathcal W_{=s}}\sum_{\substack{x\in 3E \\ E=E(s,x)}}|f(x) -f_{3E}|^2\pi(x) \\
&& +\sum_{E\in \mathcal W_{<s}}\sum_{\substack{x\in 3E \\ E=E(s,x)}}|f(x) -f_{F(s,E)}|^2\pi(x) \\
&\le& \sum_{E\in \mathcal W_{=s}}\sum_{x\in 3E}|f(x) -f_{3E}|^2\pi(x) \\
&& +\sum_{E\in \mathcal W_{<s}}\sum_{x\in 3E}|f(x) -f_{F(s,E)}|^2\pi(x).\end{eqnarray*} 
 Note that, in the first two lines, we are only summing over the $x$ such that $E = E(s,x)$ i.e., $E \in \mathcal{W}_s$ is the selected ball of radius $s$ which covers $x$. That way, $x \in U$ appears once in the sum. In the third line, we expand the sum and each $x$ may appear multiple times.
 
 We can bound and collect the first part of the right-hand side of the last inequality using the Poincar\'e inequality on balls of radius  $3s$ and Lemma~\ref{lem-WPP}(1),
\begin{eqnarray}
  \sum _{E\in W_{=s}} \sum_{ 3E} |f-f_{3E}|^2\pi &\le & P (3s)^\theta \sum_{E\in W_{=s}} \sum_{\substack{x,y\in 3E \\ x\sim y}}|f(x)-f(y)|^2\mu_{xy}  \nonumber \\
& \le &3^\theta P D^5 s^\theta  \mathcal E_{\mu,U}(f,f).  \label{cruxPP1} \end{eqnarray}
 This reduces  the proof to bounding 
\begin{eqnarray*}\lefteqn{\sum_{E\in W_{<s}} \sum_{x\in 3E}|f_{f(x)-f_{F(s,x)}}|^2\pi(x)} &&\\
& \le& 2\sum_{E\in \mathcal W_{<s}}\left(\sum_{x\in 3E}|f(x)-f_{3E}|^2\pi(x)+ \pi(3E)|f_{3E}-f_{F(s,x)}|^2\right).\end{eqnarray*}
The first part of  the right-hand side is, again, easily bounded by
\begin{eqnarray*}2\sum_{E\in \mathcal W_{<s}}\sum_{x\in 3E}|f(x)-f_{3E}|^2\pi(x) &\le & 3^{1+\theta} P s^\theta \sum_{E\in \mathcal W_{<s}}
\sum_ {x,y\in 3E, x\sim y}|f(x)-f(y)|^2\mu_{xy} \\
&\le & 3^{1+\theta} P D^5 s^\theta  \mathcal E_{\mu,U}(f,f).\end{eqnarray*}
The second part is
$$2\sum_{E\in \mathcal W_{<s}}\pi(3E)|f_{3E}-f_{F(s,x)}|^2$$
for which we use the  chain of balls  $\mathcal W'_{s}(E) =(F^{s,E}_0,\dots, F^{s,E}_{q^*_s(E)})$ to write
$$|f_{3E}-f_{F(s,E)}|\le \sum_{0}^{q^*_s(E)-1} |f_{F^{s,E}_i}-f_{F^{s,E}_{i+1}}|.$$
Lemma \ref{lem-WPP3}(2)-(3) and the notation $G(F,f)$ introduced for the proof of Theorem \ref{th-P1}  yields 
$$ |f_{F^{s,E}_i}-f_{F^{s,E}_{i+1}}| \le Q s^{\theta/2}  G(F^{s,E}_i,f),\;\;Q^2=\max\{ s^{-\theta} P_e, 8^{\theta}2PD^6\} $$
and,  with $\kappa $ as in Lemma \ref{lem-WPP}(2), 
$$ |f_{3E}-f_{F(s,E)}|  \mathbf 1_E \le  Qs^{\theta/2} \sum_{0}^{q^*_s(E)-1} G(F^{s,E}_i,f)  \mathbf 1_E \mathbf 1_{\kappa F^{s,E}_i}.$$
Using this estimate, the same argument used at the end of the proof of Theorem \ref{th-P1} (and based on Proposition \ref{prop-sumballs})  gives
$$2\sum_{E\in \mathcal W_{<s}}\pi(3E)|f_{3E}-f_{F(s,x)}|^2
\le  16 Q^2 D^{13+2\log_2 2\kappa} s^\theta  \mathcal E_{\mu,U}(f,f).$$
 \end{proof}

\section{Adding weights and comparison argument}  \setcounter{equation}{0} 
\label{sec-AW}

Comparison arguments are very useful in the study of ergodic finite Markov chains (see ~\cite{DSCcomprev} and~\cite{DSCcompg}). This section uses these ideas in the present context. The results here are used in Section~\ref{sec-Metro} to study the rates of convergence for Metropolis type chains and in Sections~\ref{sec-Dir} and~\ref{sec-IU} for studying Markov chains which are killed on the boundary.

By their very nature, the (almost identical)  proofs of Theorems \ref{th-P1} and~\ref{th-PP1} allow for a number of important variants. In this subsection, we discuss transforming
the pair $(\pi,\mu)$ into a pair $(\widetilde{\pi},\widetilde{\mu})$ so that the proofs of the preceding section yield Poincar\'e type inequalities (including $Q$-type) for this new pair.  

\begin{defin}  Let $U$ be a finite domain in $(\mathfrak X,\mathfrak E,\pi,\mu)$. Let $(\widetilde{\pi},\widetilde{\mu})$  be given on $(U,\mathfrak E_U)$. 
We say that the pair  $(\widetilde{\pi},\widetilde{\mu})$ 
$(\eta,A)$-dominates the pair $(\pi,\mu)$ in $U$ if, for any ball $E=B(z,r)\subset U$ with
$r\le 6 \eta\delta(z)$, we have
$$\sup\left\{\frac{\widetilde{\pi}(x)}{\pi(x)} \ : \ x \in B\right\}\le A \inf\left\{\frac{\widetilde{\mu}_{xy}}{\mu_{xy}}:x\in B,\{x,y\}\in E\right\} ,$$
\end{defin}
\begin{rem}  If $\eta\ge 1/6$, this property is very strong and not very useful. We will use it with $\eta\le 1/12$ so that each of the  balls considered is
far from the boundary relative to the size of its radius. The size of balls for which this property is required, namely, balls such that $r\le 6\eta\delta(z)$ is dictated by the fact the we will have to use this property for the balls   $24 E$ where $E=B(z,r(z))$ is a ball that belong to an $\eta$-Whitney covering of $U$. See Lemma \ref{lem-WPP3}(3). By construction, such a ball $E$ will satisfy
$r(z)=\eta \delta(z)/4$ and $r=24r(z)$ satisfies $r = 6\eta \delta(x)$.   
\end{rem}

The following obvious lemma justifies the above definition.
\begin{lem} \label{lem-domP}
Assume that  $(\widetilde{\pi},\widetilde{\mu})$ 
$(\eta,A)$-dominates the pair $(\pi,\mu)$ in $U$.
\begin{enumerate}
\item
If $(\pi,\mu)$ is $P_e$-elliptic then $(\widetilde{\pi},\widetilde{\mu})$ is 
$AP_e$-elliptic on $U$.
\item If $B=B(z,r)$ is a ball such that $r\le 6\eta\delta(z)$ and the Poincar\'e inequality
$$\sum_{x \in B}|f(x)-f_B|^2\pi\le P(B)\sum_{x,y\in B}|f(x)-f(y)|^2\mu_{xy}$$
holds on $B$ 
then
$$\sum_B|f-\widetilde{f}_{B}|^2\widetilde{\pi}\le
\sum_{x\in B}|f(x)-f_B|^2\widetilde{\pi}\le AP(B)\sum_{x,y\in B}|f(x)-f(y)|^2\widetilde{\mu}_{xy}$$
where $\widetilde{f}_{B}$ is the mean of $f$ over $B$ with respect to $\widetilde{\pi}$ and $f_B$ is the mean of $f$ over $B$ with respect to $\pi$.
\end{enumerate}
\end{lem}

\begin{defin} \label{rem-tilde}
Assume that $U$ is a connected subset of $(\mathfrak X,\mathfrak E)$
with internal boundary 
$\delta U=\{x\in U:\exists \,y\in \mathfrak X\setminus U,\;\{x,y\}\in \mathfrak E\}.$ For each $x\in \delta U$, introduce an auxiliary symbol, $x^c$ and set  
$$\mathfrak U= U\cup \{x^c:x\in \delta U\},\;\;\mathfrak E_\mathfrak U=\mathfrak E_U\cup \{\{x,x^c\}: x\in \delta U\},$$
so that $\mathfrak U$ has an additional copy of $\delta U$ attached to $\delta U$. By inspection, a domain $U$ is in $ J(\mathfrak X,\mathfrak E,\alpha, o,R)$ if and only if
$U\in J(\mathfrak U,\mathfrak E_\mathfrak U,\alpha,o,R)$.   If $\widetilde{\pi}$ is a measure on $U$  then we can extend this measure to a measure on $\mathfrak U$, which we still call $\widetilde{\pi}$, by setting 
 $\widetilde{\pi}(x^c)=\widetilde{\pi}(x)$, $x\in \delta U$.  If $\widetilde{\pi}$ is  $\widetilde{D}$-doubling on $(U,\mathfrak E_U)$ then its  extension is $2\widetilde{D}$-doubling on $(\mathfrak U,\mathfrak E_U)$.
\end{defin}

\subsection{Adding weight under the doubling assumption for the weighted measure}

\begin{theo}\label{th-PP1doub}
Referring to the setting of Theorems {\em \ref{th-P1}-\ref{th-PP1}},
assume further that we are given $\eta\in (0,1/12)$ and a pair $(\widetilde{\pi},\widetilde{\mu})$ on $U$ which dominates $(\pi,\mu)$ with constants $(\eta,A)$ and such that $\widetilde{\pi}$ is $\widetilde{D}$-doubling on $(U,\mathfrak E_U)$.
Then there exists a constant $C$ depending only on $\eta, A,\alpha, \theta, \widetilde{D},P,P_e$
such that
$$\forall\,s>0,\;\;\sum_{x \in U}|f(x)-\widetilde{Q}_sf(x)|^2\widetilde{\pi}(x)\le C s^\theta 
\mathcal E_{\widetilde{\mu},U}(f,f).$$
We can take 
$$ C= 7(3^\theta PA(2\widetilde{D})^5)+16 A(2\widetilde{D})^{14+ 2\log(2\kappa)}  \max\{P_e, 8^{\theta}2P (2\widetilde{D})^6\}$$
where $\kappa =7/(\alpha\eta) $.

Here $\widetilde{Q}_s$ is as in {\em Definition \ref{def-Q}} with $\widetilde{\pi}$ instead of $\pi$.
In particular,
$$\sum_U|f(x)-\widetilde{f}_U|^2\widetilde{\pi}(x)\le C R^\theta 
\mathcal E_{\widetilde{\mu},U}(f,f).$$\end{theo}
\begin{proof} Follow the proofs of Theorems \ref{th-P1}-\ref{th-PP1}, using a $\eta$-Whitney covering with $\eta$  small enough that the Poincar\'e inequalities on  Whitney balls (in fact, on double Whitney balls) holds for the pair $(\widetilde{\pi},\widetilde{\mu})$ by Lemma \ref{lem-domP}. 
To make the argument go as smoothly as possible, use the construction of $(\mathfrak U,\mathfrak E_U)$ in Definition \ref{rem-tilde}.  The proof proceeds as before with $(\widetilde{\pi},\widetilde{\mu})$ instead of $(\pi,\mu)$. The full strength of the assumption that  $\widetilde{\pi}$ is doubling is key in applying Proposition 
\ref{prop-sumballs} in this context.
\end{proof}

\subsection{Adding weight without the doubling assumption for the weighted measure} 

\begin{defin}Let $\psi:U\ra (0,\infty)$ be a positive function on $U$ (we call it a weight). 
We say that $\psi$ is  $A$-doubling on $U$ if the measure $\psi\pi$ is doubling on $(U,\mathfrak E_U)$ with constant $A$.
\end{defin}

\begin{defin}\label{def-regular} Let $\psi:U\ra (0,\infty)$ be a positive function on $U$.
We say that $\psi$ is $(\eta,A)$-regular on $U$ if
 $$\psi(x)\le A \psi(y) \mbox{ for all }\{x,y\}\in \mathfrak E_U,$$
 and, for any ball $E=B(z,r)\subset U$ with
$r\le 6\eta\delta(z)$, we have
$$\max_{E}\{\psi\}\le A \min_{E}\{\psi\}.$$
\end{defin}

\begin{rem} \label{rem-psimu}
Assume that $\psi$  is $(\eta,A)$-regular and consider any pair
$(\widetilde{\pi},\widetilde{\mu}) $ on $(U,\mathfrak E_U)$ such that
$$ \widetilde{\pi}\le \psi \pi,\;\;\mu_{xy} \psi(x)\le A' \widetilde{\mu}_{xy}.$$
Then the pair $(\widetilde{\pi},\widetilde{\mu})$  $(\eta,AA')$-dominates $(\pi,\mu)$. For instance we can set $\widetilde{\pi}=\psi\pi$ and take $\widetilde{\mu}$ to be given by one of the following choices:
$$\mu_{xy}\sqrt{\psi(x)\psi(y)},\;\;\widetilde{\mu}_{xy}=\mu_{xy}\min\{\psi(x),\psi(y)\}
\mbox{ or } \widetilde{\mu}_{xy}=\mu_{xy}\max\{\psi(x),\psi(y)\}.$$
In these three cases $A'=\sqrt{A}$, $A'= A$ and $A'=1$, respectively.
\end{rem}

\begin{defin} \label{def-controlled}
Fix $\eta \in (0,1/8)$. Let $\psi$ be a weight  on a finite domain $U$ such that  $\psi$ is $(\eta,A)$-regular on $U$. Assume $U$ is a John domain, $U\in  J(\alpha,o,R)$, equipped with John paths $\gamma_x$ joining $x$ to $o$, $x\in U$, and a family of $\eta$-Whitney coverings $\mathcal W_s$, $s\ge 1$. We say that $\psi$ is $(\omega,A_1)$-controlled if, for any local s-chain $\mathcal W'_{<s}(E)
=(F^{s,E}_0,\dots,F^{s,E}_{q^*_s(E)})$ with $F^{s,E}_i=B(x_i,3r(x_i))$, $0\le i\le q^*_s(E)$, we have
$$\forall\,s\ge 1,\;\;\forall \, i\in \{0,\dots, q^*_s(E)\},\;\;\psi(x_0)\le  A_1 s^\omega \psi(x_{i}).$$
When we say that an $(\eta,A)$-regular weight $\psi$ on $U\in J(\alpha, o,R)$ is $(\omega,A_1)$-controlled, we assume implicitly that a family of $\eta$-Whitney coverings $\mathcal W_s$, $s\ge 1$ has been chosen. \end{defin}

\begin{rem} 
When $\omega=0$, the weight $\psi$ is essentially increasing along the John path joining Whitney balls to $o$. 
\end{rem}

\begin{theo} \label{th-PP1controlled} Given the setting of Theorems \ref{th-P1} and~\ref{th-PP1},
assume further that we are given $\eta\in (0,1/12)$ and a weight $\psi$ on $U$ such that
$\psi$ is $(\eta,A)$-regular and $(\omega,A_1)$-controlled. Set $\widetilde{\pi}=\psi\pi$ and let $\widetilde{\mu}$ be a weight defined on $\mathfrak E_U$ such that
\begin{equation} \label{psimu}
\forall x,y\in U,\;\;
\psi(x)\mu_{xy}\le A_2 \widetilde{\mu}_{xy}.
\end{equation}
Then there exist a constant  $C$ depending only on $\eta, \alpha, \theta,  A, A_1,A_2 D,P$ and such that
$$\forall\,s>0,\;\;\sum_U|f(x)-\widetilde{Q}_sf(x)|^2\widetilde{\pi}(x)\le C s^{\theta+\omega} 
\mathcal E_{\widetilde{\mu},U}(f,f).$$
Here $\widetilde{Q}_s$ is as in {\em Definition \ref{def-Q}} with $\widetilde{\pi}$ instead of $\pi$. The constant $C$ can be taken to be
$$C= C= 7AA_2(3^\theta PD^5)+16 D^{14+ 2\log(2\kappa)} A^3A_1A_2 \max\{P_e, 8^{\theta}2PA^2D^6\}$$
where $\kappa =7/(\alpha\eta) $. In particular,
$$\sum_U|f(x)-\widetilde{f}_U|^2\widetilde{\pi}(x)\le C R^{\theta +\omega} 
\mathcal E_{\widetilde{\mu},U}(f,f).$$\end{theo} 
\begin{proof} (The case $s\in [0,1]$ is trivial and we can assume $s>1$). This result is a bit more subtle than the previous result because the measure $\widetilde{\pi}$ may not be  doubling. 
However, because $\psi$ is $(\eta,A)$-regular and  $\widetilde{\mu}$ satisfies (\ref{psimu}), it follows from Remark \ref{rem-psimu}  that $(\widetilde{\pi},\widetilde{\mu})$ $(\eta,AA_2)$-dominates $(\pi,\mu)$.  By Lemma \ref{lem-domP} this implies that $(\widetilde{\pi},\widetilde{\mu})$ is $AA_2P_e$-elliptic and 
the $\theta$-Poincar\'e inequality on balls $B(z,r)$ such that ${r\le \eta \delta(z)}$, $z\in U$, with constant  $PAA_2$. 
Using the notation $\widetilde{f}_B$ for the mean of $f$ over $B$ with respect to $\widetilde{\pi}$, we also have, for any ball $E$ in $\mathcal W_{<s}$ and 
its local s-chain $\mathcal W'_{<s}(E)=(F^{s,E}_i)_0^{q^*_s(E)}$ with $F^{s,E}_i=B(x_i,3r(x_i))$, $F^{s,E}_0=3E$,
$$|\widetilde{f}_{F^{s,E}_i}-\widetilde{f}_{F^{s,E}_{i+1}}|  \le  Q s^{\theta/2} \widetilde{G}(F^{s,E}_i,f)$$ 
where $\widetilde{G}$ is defined just as $G$ but with respect to the pair  $(\widetilde{\pi},\widetilde{\mu})$.  Here we can take 
$$Q^2=AA_2\max\{P_e,8^\theta 2 P A^2D^6\}.$$
In this computation  (see the proof of Lemma \ref{lem-WN3}), we  have had to estimate 
$\widetilde{\pi}(8F_j)/\widetilde{\pi}(F_j)$ by $ AD^3$ using the doubling property of $\pi$ and the fact that $\psi$ is $(\eta,A)$-regular (in words, what is used here is the fact that, because $\psi$ is $(\eta,A)$-regular, $\widetilde{\pi}$ is doubling on balls that are far away from the boundary even so it is not necessarily globally doubling on $(U,\mathfrak E_U)$).

Next, set 
$$G^*(F,f) = \left(\frac{1}{\pi(F)}\sum _{x\in 8F \cap U}\sum_{y\sim x, y\in U}|f(x)-f(y)|\widetilde{\mu}_{xy}\right)^{1/2}.$$
This differs from $\widetilde{G}(F,f)$ only  by the use of $\pi$ instead of $\widetilde{\pi}$ in the fraction appearing in front of the summations (but note that this quantity involves the edge weight $\widetilde{\mu}$).  Now,
we have
$$|\widetilde{f}_{F^{s,E}_i}-\widetilde{f}_{F^{s,E}_{i+1}}| \left(\frac{\widetilde{\pi}(3E)}{\pi(3E)} \right)^{1/2} \le  A\sqrt{A_1}Qs^{(\theta+\omega)/2} G^*(F^{s,E}_i,f)$$ 
because 
$$\frac{\widetilde{\pi}(3E)}{\pi(3E)} \le  A\psi(x_0)\le A A_1s^{\omega} \psi(x_i)\le   A ^2A_1 s^{\omega}  \frac{\widetilde{\pi}(F^{s,E}_i)}{\pi(F^{s,E}_i)}.$$
This gives
\begin{eqnarray*}
\lefteqn{|\widetilde{f}_{3E}-\widetilde{f}_{F(s,E)}| \left(\frac{\widetilde{\pi}(3E)}{\pi(3E)} \right)^{1/2} \mathbf 1_E}&&\\
&\le& A^2A_1Q s^{(\theta+\omega)/2} \sum_0^{q^*_{s}(E)-1} 
G^*(F^{s,E}_i,f) \mathbf 1_E \mathbf 1_{ \kappa F^{s,E}_i}.\end{eqnarray*}
To finish the proof, we square both sides, multiply by $\pi(3E)$, and proceed as at the end of the proof of Theorem \ref{th-PP1}, using the doubling property of $\pi$. 
 \end{proof}
 
\subsection{Regular weights are always controlled}

The following  lemma is a version of a well-known fact  concerning chains of Whitney balls in John domains.

\begin{lem}  \label{lem-RC}
Assume that $(\mathfrak X,\mathfrak E,\pi)$ is doubling with constant $D$.
Fix $\eta \in (0,1/8)$. Let $\psi$ be a weight  on a finite domain $U$ such that  $\psi$ is $(\eta,A)$-regular on $U$ and $U$ is a John domain, $U\in  J(\alpha,o,R)$, equipped with John paths $\gamma_x$ joining $x$ to $o$, $x\in U$, and a family of $\eta$-Whitney coverings $\mathcal W_s$, $s>0$.
Then there exist $\omega\ge 0$ and $A_1\ge 1$ such that  $\psi$ is $(\omega,A_1)$-controlled on $U$.    Here $A_1=A^{2+4\kappa}$ and $\omega=2\kappa \log _2A$ with $\kappa= D^{4+\log_2(1+1/(\alpha\eta))} .$\end{lem}
\begin{proof}  Using the notation of Definition \ref{def-controlled}, we need to compare the values taken by the weight $\psi$ at any pair of points $x_0,x_i,$
such that $x_0$ is the center of a Whitney ball $E$  and $x_i$ is the center of a ball belonging to the local s-chain $\mathcal W'_{<s}(E)$.  This local s-chain is made of balls in $\mathcal W'_s$, each of which has radius at most $3s$  and  intersects the John path $\gamma_E=\gamma_{x_0}$ joining $x_0$ to $o$. 

Assume that we  can prove that
\begin{equation}\label{logcount}
\#\{ K\in \mathcal W_{<s}:  2K \cap \gamma_E\neq \emptyset\}\le  \kappa \log_2 (4s).\end{equation} Of course, under this assumption,
$$1+q_s^*(E)=\#\mathcal W'_{<s}(E)\le 1+\kappa  \log_2 (4s) .$$
Further,  by definition of $\mathcal W'_{<s}(E)=(F^{s,E}_0,\dots,F^{s,E}_{q^*_s(E)})$, the balls  $2F^{s,E}_i,2F^{s,E}_{i+1}$ have a non-empty intersection or are singletons $\{x_i\}$, $\{x_{i+1}\}$ with  $\{x_i,x_{i+1}\}\in \mathfrak E$.  Since $\psi$ is  $(\eta,A)$-regular and the ball $2F^{s,E}_i$ has radius
$6r(x_i)\le 3\eta \delta(x_i)/2$,  we have 
$$\psi(x_i)\le A^2\psi(x_{i+1}), \;\;i=0,\dots,q^*_s(E).$$  
This implies 
\begin{equation} \label{RtoC1}
\psi(x_0)\le  A^{2(1+\kappa \log_2 (4s))} \psi(x_i)= A^{2+4\kappa} s^{2\kappa \log_2 A} \psi(x_i)
,\;\; \;\;i=0,\dots,q^*_s(E).\end{equation}

To prove (\ref{logcount}),  for each $\rho\ge 1$,  let the John path $\gamma_{x_0}$ be
$$\gamma_{x_0}=(\xi_0=x_0, \dots,\xi_m=o).$$
Consider
$$\#\{ K=B(x,r)\in \mathcal W_{<s}:  3K \cap \gamma_E\neq \emptyset,  r\in [\rho,2\rho) \},\;\;\rho\ge 1.$$
Let $K=B(x,r), K'=B(x',r')$ be any two balls from that set and let  
$\xi_i\in 3K$ and $\xi'_{i}\in 3K'$ be two points on the John path $\gamma_{x_0}$ that are witness to the fact that these balls intersect $\gamma_{x_0}$. 
 Now,  by construction,   
 $$ d(x,\xi_i)\le 3r,\;r=\eta \delta(x)/4\mbox{ and } \delta (\xi_i) \ge \alpha (1+i)$$
It follows that   $\delta(x)\ge \delta(\xi_i)- (3\eta/4)\delta(x)$ and thus, using a similar argument for~{$x',\xi'_{i},r'$,}
$$\delta(x)\ge (\alpha/2) (1+i) \mbox{ and }   \delta(x')\ge (\alpha/2) (1+i') .$$ 
This implies that $1+\max\{i,j\}\le (16/\alpha \eta) \rho$ and 
$$d(x,x')\le  8\left(1+ \frac{2}{\alpha\eta} \right)\rho,\;\; B(x',\rho)\subset B\left(x, \left(9+ \frac{16}{\alpha\eta} \right) \rho\right).$$
By construction, the balls $K'\in \mathcal W_{<s}$ are disjoint and the doubling property of $\pi$ thus implies that 
$$\#\{ K=B(x,r)\in \mathcal W_{<s}:  3K \cap \gamma_E\neq \emptyset,  r\in [\rho,2\rho) \}   \le    D^{ 4+\log_2 (1+1/(\alpha\eta))}.
$$
The same argument shows that
$$\#\{ K=B(x,r)\in \mathcal W_{<s}:  3K \cap \gamma_E\neq \emptyset,  r\in (0,1) \}   \le D^{ 4+\log_2 (1+1/(\alpha\eta))}.
$$
For $s\in (2^k,2^{k+1}]$, this implies
\begin{eqnarray*}
\#\{ K\in \mathcal W_{<s}:  3K \cap \gamma_E\neq \emptyset\}
&\le & 
D^{ 4+\log_2 (1+1/(\alpha\eta))} (k+2)\\
& \le & D^{ 4+\log_2 (1+1/(\alpha\eta))} \log_2(4s).\end{eqnarray*}
This, together with~\eqref{RtoC1},  yields
$$
 \psi(x_0)\le A^{2+4\kappa} s^{2\kappa \log_2 A} \psi(x_i)$$ for $i=0,\dots,q^*_s(E),  \kappa= D^{4+\log_2(1+1/(\alpha\eta))}$.
\end{proof}

\section{Application to Metropolis-type chains} \setcounter{equation}{0}  \label{sec-Metro} \label{sec-Met}

\subsection{Metropolis-type chains}
We are ready to apply the technical results developed so far (primarily within Section~\ref{sec-AW}) to Metropolis-type chains
on John domains. The reader may find motivation in the explicit examples of Section~\ref{sec-metropolis-ex}. First we explain what we mean by Metropolis-type chains. Classically, The Metropolis and Metropolis Hastings algorithms give a way of changing the output of one Markov chain to have a desired stationary distribution. See~\cite{liu} or~\cite{DSCMetro} for background and examples.

Assume we are given the background structure $(\mathfrak X,\mathfrak E,\mu, \pi)$
with $\mathfrak X$ finite or countable. Assume that $\mu$ is adapted and subordinated to $\pi$. Let $U$ be a finite domain in $\mathfrak X$. This data determines  an irreducible Markov kernel $K_{N,U}$ on $U$ with reversible probability measure $\pi_U$, proportional to $\pi |_U$,
given by (this is similar to~\eqref{def-Ksub}) 
\begin{equation} \label{def-KNU}
K_{N,U}(x,y)=\left\{\begin{array}{ccl}  \mu_{xy}/\pi(x)  & \mbox{ for  } x\neq y,\,x,y\in U\\
1-(\sum_{z\in U:z\sim x}\mu_{xz}/\pi(x)) & \mbox{ for } x=y\in U. 
\end{array}\right. \end{equation}
The notation $K_{N,U}$ captures the idea that this kernel corresponds to imposing the Neumann boundary condition in $U$ (i.e., some sort of reflexion of the process at the boundary). 

Suppose now that we are given a vertex weight $\psi$ and  a symmetric edge weight $h_{xy}$ on the domain $U$.  Set
$$\widetilde{\pi}=\psi \pi,\;\; \widetilde{\mu}_{xy}=\mu_{xy}h_{xy},$$
and  assume  that   
$$\sum_{y\in U} \widetilde{\mu}_{xy}\le \widetilde{\pi}$$
so that $\widetilde{\mu}$ is subordinated to $\widetilde{\pi}$ in $U$.
This yields a new Markov kernel  $\widetilde{K}$ defined on $U$ by
\begin{equation} \label{def-Ktilde}
\widetilde{K}(x,y)=\left\{\begin{array}{ccl}  \widetilde{\mu}_{xy}/\widetilde{\pi}(x)  & \mbox{ for  } x\neq y,\,x,y\in U\\
1-(\sum_{z\in U: z\sim x}\widetilde{\mu}_{xz}/\widetilde{\pi}(x)) & \mbox{ for } x=y\in U. 
\end{array}\right. \end{equation}
This kernel is irreducible and reversible with reversible probability measure proportional to $\widetilde{\pi}$. 

\begin{exa}\label{ex-metropolis} The choice $h_{xy}=\min\{\psi(x),\psi(y)\}$ satisfies this property and yields the well-known Metropolis chain with proposal chain $(K_{N,U},\pi_U)$ and target probability measure $\widetilde{\pi}_U$, proportional to
 $\widetilde{\pi}=\psi\pi |_U$.   Other choice of $h$ would lead to similar chains including the variants of the Metropolis algorithm considered by Hastings and Baker.  See the discussion in \cite[Remark 3.1]{BD}.
\end{exa}
 
\subsection{Results for Metropolis type chains} 
In order to simplify notation, we fix the background structure  $(\mathfrak X,\mathfrak E, \pi, \mu)$. We assume that $\pi$ is $D$-doubling, $\mu$ is adapted 
and that  the pair 
$(\pi,\mu)$ is elliptic and satisfies the $\theta$-Poincar\'e inequality on balls with constant $P$. We also assume that $\mu$ is subordinated to $\pi$. We also fix $\alpha\in (0,1)$.  In the statements below, we will use  $c,C$ to denote  quantities whose exact values change from place to place and depend only on $\theta, D,P_e,P$ and $\alpha$. Explicit descriptions of these quantitates in terms of the data can be obtained from the proofs. They are of the form 
$$\max\{c^{-1}, C\}\le  A_1^\theta D^{A_2(1+\log 1/\alpha)}\max\{P_e,P\} $$
where $ A_1,A_2 $ are  universal constants. 

Within this fixed background, we consider the collection of all finite domains  $U\subset \mathfrak X$ which are John domains of type $J(\alpha,o,R)$ for some point $o\in U$
and  $R\le 2 R(U,o,\alpha)$.  The parameter $R$ is allowed to vary freely and all estimates are expressed in terms of $R$.  
Recall that $\rho_o(U)=\max\{d_U(o,x): x\in U\}$ satisfies 
$$2R(U,o,\alpha)\ge 2\rho_o(U)\ge  \delta(o) \ge \alpha R(U,o,\alpha).$$
We always assume implicitly that $U$ is not reduced to a singleton so that $R(U,o,\alpha)\ge 1$.
Since $\alpha$ is fixed,  it follows that  $R\asymp \rho_o(U)$, namely,
$$\frac{\alpha}{2}R \le  \rho_o(U)   \le 4R. $$  

We need the following simple technical lemma. 
\begin{lem}  \label{lem-Q}
Assume that $U\in J(o,\alpha,R)$, with $R\le 2R(U,o,\alpha)$, is not a singleton and $0<\eta<1/4$. Referring to the construction of the ball $F(s,x)$, $s>0$, $x\in U$, used in Definition \ref{def-Q},  
any $\eta$-Whitney covering $\mathcal W_s$ of $U$ satisfies
\begin{itemize}
\item
 $W_{=s}=\emptyset$  whenever  $s\ge 3 R(U,o,\alpha)$. In that case, $F(s,x)=3E^s_o$ for all $x\in U$ and the ball $E^s_o$ has radius 
 $$r(o)=\eta\delta(o)/4  \mbox{ with }     \alpha R(U,o,\alpha)\le  \delta(o) \le 2R(U,o,\alpha). $$
\item  When $ s\le \alpha \eta R(U,o,\alpha)/4$,  all balls $F(s,x)$ have radius $3s$. 
\item   When  $s\in (\alpha \eta R(U,o,\alpha)/4, 3R(U,o,\alpha))$, each ball $F(s,x)$, $x\in U$,  has radius contained in the interval
$$ [\alpha \eta  R(U,o,\alpha)/2, 9R(u,o,\alpha)].$$ 
\end{itemize}
In particular, for all   $s\in (0, \alpha \eta R/8)$  
\begin{equation} \label{MGF}
\frac{\pi(F(s,x))}{\pi(U)}\ge \frac{\pi(B(z(x),s))}{\pi(B(z(x),8R)}\ge \frac{1}{D^2} \left(\frac{1+s}{8R}\right)^{\log_2 D},\end{equation}
and, for all $s\in (0,\alpha\eta R/8)$,  the averaging operator  $Q_s$  (Definition \ref{def-Q}) satisfies
$$\|Q_s f\|_\infty\le  
M(1+s)^{-\log_2 D} \|f\|_1 $$ 
 where  $\|f\|_1=\sum|f|\pi$
and $M= D^2 (8R)^{\log_2 D} /\pi(U) $.\end{lem}
\begin{rem} \label{rem-MGF}
The bound~\eqref{MGF} is a version of moderate growth for the metric measure space $(U,d_U,\pi)$ with the additional twist that, for
    each $s,x$, we consider the ball $F(s,x)$ instead of the ball $B_U(x,s)$. The reason for this is that it is the balls $F(s,x)$ that appear in the definition of the operator $Q_s$ because of the crucial use we make of the Whitney coverings $\mathcal W_s$, $s>0$.
\end{rem}

Our first result concerns the Markov chain driven by $K_{N,U}$ defined in Example~\ref{ex-metropolis}. This is a reversible chain
 with reversible probability measure $\pi_U$.  We let $\beta=\beta_{N,U}$ be the second largest eigenvalue of $K_{N,U}$ and $\beta_{-}=\beta_{N,U,-}$ be the smallest eigenvalue of $K_{N,U}$.  From the definition, it is possible that  $U = \mathfrak X$ and $\beta_-=-1$.
  \begin{theo}\label{th-KNU} There exist constants $c,C$ such that for  any $R>0$ and any finite domain $U\in J(\alpha,R)$, we have
  $$1-\beta_{N,U}\ge cR^{-\theta}.$$
  Assume further that $1+\beta_{N,U,-}\ge cR^{-\theta}$. Under this assumption, for all $t\ge R^\theta$,
  $$\max_{x,y\in U}\left\{ \left|\frac{K^t_{N,U}(x,y)}{\pi_U(y)} -1\right|\right\}\le C e^{-2ct/R^\theta}.$$
\end{theo}
 \begin{proof}  This result is a consequence of Theorems \ref{th-P1} and~\ref{th-PP1}. 
 We use a Whitney covering family $W_s$, $s>0$, with $\eta=1/4$.
 For later purpose when we will need to use a given $\eta$, we keep $\eta$ as a parameter in the proof. Theorem \ref{th-P1} gives the estimates for $1-\beta_{N,U}$.  (Theorem \ref{th-PP1} also gives that eigenvalue estimate if we pick $s\asymp R$ large enough that the Whitney covering $\mathcal W_s$ is such that $W_{=s}$ is empty.) 
By Lemma \ref{lem-Q}, for $s\in (0,\alpha \eta R/8]$
$$\|Q_sf\|_\infty\le M (1+s)^{-\log_2 D}\|f\|_1$$ where $M=D^2 (8R)^{\log_2 D}/  \pi(U) .$ Now, we appeal to Theorem \ref{th-PP1} and Propositions \ref{pro-Nash1} and~\ref{pro-Nash2} to obtain
$$\sup_{x,y\in U} \{K^t(x,y)/\pi(y)\}\le  C \pi(U)^{-1} (R^\theta/(n+1))^{\log_2 D/\theta} $$
for all $t\le  (\alpha \eta R/4)^\theta $.  This is the same as
\begin{equation} \label{NashK}
\sup_{x,y\in U} \{K^t(x,y)/\pi_U(y)\}\le  C (R^\theta/(n+1))^{\log_2 D/\theta} \end{equation}
for all $t\le  (\alpha \eta R/4)^\theta$, because $\pi_U=\pi(U)^{-1}\pi|_U$.
The constant $C$ is of the type described above and incorporates various factors depending only on $D,\theta, \alpha, P, P_e$
which are made explicit in Theorem \ref{th-PP1}, Lemma \ref{lem-Q} and Propositions \ref{pro-Nash1} and~\ref{pro-Nash2}.  

 The next step is (essentially) \cite{DSCNash}[Lemma 1.1].  Using operator notation for ease, write
 $$\sup_{x,y\in U}\left\{\left|\frac{K^t(x,y)}{\pi_U(y)}-1\right|\right\}= \|(K-\pi_U)^t\|_{L^1(\pi_U)\ra L^\infty}$$
 and observe that, for any $t_1,t_2$ such that $t=t_1+2t_2$, $\|(K-\pi_U)^t\|_{L^1(\pi_U)\ra L^\infty}$ is bounded above by the product of 
$$ \|(K-\pi_U)^{t_2}\|_{L^1(\pi_U)\ra L^2(\pi_U)}, \;\;\|(K-\pi_U)^{t_1}\|_{L^2(\pi_U)\ra L^2(\pi_U)}$$ and 
 $$\|(K-\pi_U)^{t_2}\|_{L^2(\pi_U)\ra L^\infty}.$$ The first and last factors are equal (reversibility and duality) and also equal to
 $$ \sqrt{\sup_{x,y\in U} \{K^{2t_2}(x,y)/\pi_U(y)\} }.$$ 
 The second factor is
 $$\|(K-\pi_U)^{t_2}\|_{L^2(\pi_U)\ra L^2(\pi_U)} = \max \{\beta_{N,U},|\beta_{N,U,-}|\}^{t_1}.$$ 
 We pick $t_2$ to be the largest integer less than or equal to  $(\alpha\eta R)^\theta /8$ and apply (\ref{NashK}) to obtain
 $$\sup_{x,y\in U}\left\{\left|\frac{K^t(x,y)}{\pi_U(y)}-1\right|\right\} \le  2^{\log_2 D/\theta} C (\alpha  \eta )^{\log_2 D/\theta} \max \{\beta_{N,U},|\beta_{N,U,-}|\}^{t_1}.$$
 This gives the desired result.
   \end{proof}
 
 The following very general example illustrates the previous theorem.
 \begin{exa}[Graph metric balls]  Fix  constants $P_e,P,\theta$ and $D$. Assume that $(\mathfrak X,\mathfrak E,\pi,\mu)$ is such that the volume doubling property holds with constant $D$ together with $P_e$-ellipticity and  the $\theta$-Poincar\'e inequality with constant $P$. We also assume (for simplicity)
 that $$ \sum_{y\sim x} \mu_{xy}\le \pi(x)/2.$$
Under this assumption, for any finite domain $U$, the kernel $K_{N,U}$ has the property that $K_{N,U}(x,x)\ge 1/2$ (this is often called ``laziness'')  and it implies that  $\beta_{N,U,-}\ge 0$.  

Let $U=B(o,R)$ be any graph metric ball  
in $(\mathfrak X,\mathfrak E)$. From Example~\ref{exa-mb}, such a ball is a John domain with $\alpha=1$, namely, $U\in J(\mathfrak X,\mathfrak E, 1, o,R)$ and $R=R(U,\alpha,o)$.
Since $\beta_{N,U,-}\ge 0$, Theorem \ref{th-KNU} applies and show that $K^t_{N,U}$ converges to $\pi_U$ in times of order $R^\theta$. This applies for instance to the metric balls of the Vicsek graph of Figure \ref{fig-V2}.
 \end{exa}

 Next we consider a weight $\psi$ which is $(\eta,A)$-regular to $U$ and 
 $A$-doubling. This means that the measure $\widetilde{\pi}=\psi\pi$ is $A$ 
 doubling on $(U,\mathfrak E_U)$ (and also it extension to $(\mathfrak U,\mathfrak E_\mathfrak U)$ is $2A$ doubling).  
 For simplicity we pick $\widetilde{\mu}$ to be given by the Metropolis choice
 $$\widetilde{\mu}_{xy}=\mu_{x,y}\min\{\pi(x),\pi(y)\}.$$
 This implies  that $\widetilde{\mu}$ is subordinated to $\widetilde{\pi}$ and we  let
 $$K_{U,\psi}=\widetilde{K}$$ be defined by (\ref{def-Ktilde}). This reversible Markov kernel
 has reversible probability measure $\widetilde{\pi}_U$ proportional to $\psi\pi$ on $U$.
 Also, the hypothesis that $\psi$ is $(\eta,A)$-regular to $U$ implies that the pair
$(\widetilde{\psi},\widetilde{\mu})$ $(\eta, A^2)$-dominates $(\pi,\mu)$ on $U$. See Remark \ref{rem-psimu}. This shows that we can use Theorem \ref{th-PP1doub}
to prove the following result using the same line of reasoning as for Theorem \ref{th-KNU}.  We will denote by $\beta_{U,\psi}$ the second largest eigenvalue of 
$\widetilde{K}=K_{U,\psi}$ and by $\beta_{U,\psi,-}$ its lowest eigenvalue.

\begin{theo}\label{th-Ktilde1} For fixed $\eta\in (0,1/8),A\ge 1$, there exist constants $c,C$ such that for any $R>0$, any finite domain $U\in J(\alpha,R)$ and any weight $\psi$ which is $(\eta,A)$-regular (see Definition~\ref{def-regular})and $A$-doubling on $U$, we have
  $$1-\beta_{U,\psi}\ge cR^{-\theta}.$$
  Assume further  that $1+\beta_{U,\psi,-}\ge cR^{-\theta}$. Under this assumption, for all $t\ge R^\theta$,
  $$\max_{x,y\in U}\left\{ \left|\frac{\widetilde{K}^t(x,y)}{\widetilde{\pi}_U(y)} -1\right|\right\}\le C e^{-2ct/R^\theta}.$$
 There are universal constants $A_1,A_2$ such that
  $$\max\{c^{-1}, C\}\le  A_1^\theta (AD)^{A_2(1+\log 1/\alpha\eta)}\max\{P_e,P\} .$$
\end{theo}

Replacing the hypothesis that $\psi$ is $(\eta,A)$-regular and $A$-doubling by the hypothesis that $\psi$ is $(\eta,A)$-regular and $(\omega,A)$-controlled leads to the following similar statement. 
\begin{theo}\label{th-Ktilde2} For fixed $\eta\in (0,1/8),A\ge 1$ and $\omega\ge 0$, there exist constants $c,C$ such that for any $R>0$, any finite domain $U\in J(\alpha,R)$ and any weight $\psi$ which is $(\eta,A)$-regular and $(\omega,A)$-controlled (see Definition~\ref{def-controlled}) on $U$, we have
  $$1-\beta_{U,\psi}\ge cR^{-(\theta+\omega)}.$$
  Assume further  that $1+\beta_{U,\psi,-}\ge cR^{-(\theta+\omega)}$. Under this assumption, for all $t\ge R^\theta$,
  $$\max_{x,y\in U}\left\{ \left|\frac{\widetilde{K}^t(x,y)}{\widetilde{\pi}_U(y)} -1\right|\right\}\le C e^{-2ct/R^{(\theta+\omega)}}.$$
 There are universal constants $A_1,A_2$ such that
  $$\max\{c^{-1}, C\}\le  A_1^{\theta +\omega}(AD)^{A_2(1+\log 1/\alpha\eta)}\max\{P_e,P\} .$$
\end{theo}
 
 \subsection{Explicit examples of Metropolis type chains}\label{sec-metropolis-ex}
 
We give four simple and instructive explicit examples regarding Metropolis chains. 
There are based on  a cube $U=[-N,N]^d$ in some fixed dimension $d$. The key parameter which is allowed to vary is $N$. This cube is equipped with it natural edge structure induced by the square grid.
The underlying edge weight is $\mu_{x,y}=(2d)^{-1}$ and $\pi$ is the counting measure.

To obtain each of our examples, we will define a ``boundary" for $U$ and a weight $\psi$
 that is $(1/8,A)$-regular and $A$ doubling.
 
 \begin{exa} Our first example uses the natural boundary of $U=[-N,N]^d$ in the square grid $\mathbb Z^d$. The weight $\psi=\psi_\nu$, $\nu\ge 0$, is given by 
 $$\psi(x)=\delta(x)^\nu.$$
Recall that $\delta(x)$ is the distance to the boundary. Thus, this power weight is largest at the center of the cube. It is $(1/8,A)$-regular and $A$-doubling with $A$ depending of $d$ and $\nu$ which we assume are fixed.
Theorem \ref{th-Ktilde1} applies (with $\theta=2$). The necessary estimates on the 
lowest eigenvalue $\beta_{U,\psi,-}$ holds true because there is sufficient holding probability provided by the Metropolis rule at each vertex (this holding is of order at least $1/N$ and, in addition, there is  also enough holding at the boundary). Here $R\asymp N$ and convergence occurs in order $N^2$ steps.
\end{exa}

\begin{exa} Our second example is obtained by adding two points to the box from the first example, which will serve as the boundary. Let $\mathfrak{X} = [-N,N]^d \cup \{u_{-},u_{+}\}$, where $u_-$ is attached by one edge to $(-N,\dots, -N)$
and $u_+$ attached by one edge to $(N,\dots,N)$. Within $\mathfrak{X}$, let $U=[-N,N]^d$, so the boundary is $\{u_-,u_+\}$. Again, we consider the power weight
$$\psi(x)=\psi_\nu(x)=\delta(x)^\nu,\;\;\nu>0$$
but this time $\delta$ is the distance to the boundary $\{u_-,u_+\}$. This power weight 
is constant  along the hyperplanes $\sum_1^d x_i=k$ and maximum on $\sum_1^d x_i=0$. 

\begin{figure}[h] 
\begin{center}
\newcommand{\Depth}{2}
\newcommand{\Height}{2}
\newcommand{\Width}{2}
\begin{tikzpicture}
\coordinate (O) at (0,0,0);
\coordinate (A) at (0,\Width,0);
\coordinate (B) at (0,\Width,\Height);
\coordinate (C) at (0,0,\Height);
\coordinate (D) at (\Depth,0,0);
\coordinate (E) at (\Depth,\Width,0);
\coordinate (F) at (\Depth,\Width,\Height);
\coordinate (G) at (\Depth,0,\Height);

\draw[fill]  (0,0,0) circle [radius=.04];
\draw (O) -- (C) -- (G) -- (D) -- cycle;
\draw (O) -- (A) -- (E) -- (D) -- cycle;
\draw  (O) -- (A) -- (B) -- (C) -- cycle;
\draw[blue, thick, fill=blue!20, opacity=.5]   (0,2,1) -- (1,2,0) -- (2,1,0) -- (2,0,1) -- (1,0,2) -- (0,1,2) -- (0,2,1) ;\draw (D) -- (E) -- (F) -- (G) -- cycle;
\draw  (C) -- (B) -- (F) -- (G) -- cycle;
\draw (A) -- (B) -- (F) -- (E) -- cycle;

\draw[fill]  (2,2,2) circle [radius=.05];
\draw (1,1,1) circle [radius=.04];
\draw [->]  (1,1,1) -- (1,1,1.8); \draw [->] (1,1,1) -- (1,1.5,1); \draw [->] (1,1,1) -- (1.5,1,1);

\end{tikzpicture}
\caption{The box $U=[-N,N]^3$ with two boundary points  $u_1,u_+$ attached at  corners $(-N,-N,-N)$ and $(N,N,N)$ (these to corners are marked with black dots). The blue plane is the set of points in $U$ at maximal distance from the boundary points $\{u_-,u_+\}$. The center of the box is shown with the axes. The grid is not shown.}
\end{center}
\end{figure}
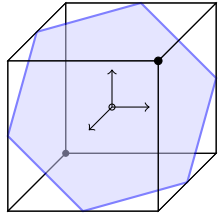  
This weight is $(1/8,A)$-regular and $A$-doubling with $A$ depending of $d$ and $\nu$ which are fixed. Theorem \ref{th-Ktilde1} applies (with $\theta=2$). The necessary estimates on the 
lowest eigenvalue $\beta_{U,\psi,-}$ hold true because there is sufficient holding probability provided by the Metropolis rule (again, at least order $1/N$ at each vertex). We have $R\asymp N$ and convergence occurs in order $N^2$ steps.
\end{exa}

\begin{exa} Our third example is obtained  by adding only one boundary point to the box from the first example. Let $\mathfrak{X} = [-N,N]^d \cup \{u_0\}$ where $u_0$ is attached by one edge to the center $(0,\dots, 0)$. Within $\mathfrak{X}$, let $U = [-N,N]^d$, so the boundary is $\{u_0\}$. Still, we consider the power weight
$$\psi(x)=\psi_\nu(x)=\delta(x)^\nu,\;\nu>0,$$
where $\delta$ is the distance to the boundary $\{u_0\}$. This power weight 
is constant  along the boundary of the graph balls centered at $(0,\dots,0)$. It is largest
at the four corners.  In this case, we obtain a John domain with a fixed $\alpha$
only when $d>1$ (in the case $d=1$, there is no way to avoid passing near the boundary point $u_0$). When $d>1$, we can chose $o$ to be one of the four corners. Again, the weight is $(1/8,A)$-regular and $A$ doubling with $A$ depending on $d$ and $\nu$ which are fixed. Theorem \ref{th-Ktilde1} applies (with $\theta=2$). The necessary estimates on the 
lowest eigenvalue $\beta_{U,\psi,-}$ hold true as n the previous examples.  Again, $R\asymp N$ and convergence occurs in order $N^2$ steps.  We note that there is no problems replacing the single ``pole'' $0$ in this example by an arbitrary finite set $\mathfrak O$ of ``poles'', as long as we fix the number of elements in $\mathfrak O$. 
\end{exa}

\begin{exa} This last example involves  weights which lead to non-doubling measure but are $\omega$-controlled. Take $d=1$ and $U=[-N,\dots,N]$, a symmetric interval around $0$ in $\mathbb Z$. Fix $\nu>1$ and consider the weight
$\psi_{\nu}=\delta(x)^{-\nu}$, where $\delta$ is the distance to the boundary $\{-N-1,N+1\}$. It is easy to check that this weight is not doubling (compare the $\widetilde{\pi}$-volume of $B(0,N/2)$ to that of $B(0,N)$).  Obviously, $\psi_\nu$ is $\omega$-controlled with $\omega=\nu$.
The reference \cite[Theorem 9.6]{LSCSimpleNash} applies to this family and provides the eigenvalue estimate 
$$1-\beta_{U,\psi_\nu}\approx N^{-1-\nu}$$
and the fact that this chain converges to its equilibrium measure in order $N^{1+\nu}$ steps.
This should be compared with the eigenvalue estimate of Theorem \ref{th-Ktilde2} which reads $1-\beta_{U,\psi_\nu}\ge c N^{-2-\nu} $
because $R\approx N$ and $\omega=\nu$.  This estimate is off by a factor of $N$, but it is clear that the parameter $\omega=\nu$ plays a key role in 
estimating $\beta_{U,\psi_\nu}$ in this case.  The following modification of this example shows that the eigenvalue estimate of Theorem \ref{th-Ktilde2} is actually almost optimal.  
Consider $\mathbb [-(N+1),(N+1)]$ equipped with the measure 
$\pi(x)= (N+2-|x|)^{-\alpha}$, $\alpha\in (0,1)$ and the usual graph structure induced by $\mathbb Z$. This space is doubling and satisfies  the Poincar\'e inequality on balls (this is not obvious, but it can be proved). On this space, let $U=[-N,\dots,N]$ and repeat the construction above with $\psi_\nu(x)=\delta(x)^{-\nu}$, $\nu>1-\alpha$. Now, on this new space,
this weight is not doubling but it is $\omega$-controlled with $\omega=\nu$.
The previous argument shows that the eigenvalue 
$\beta_{U,\alpha,\psi_\nu}$ satisfies
$1-\beta_{U,\alpha,\psi_\nu}\approx  N^{-1-\alpha-\nu}$ whereas Theorem \ref{th-Ktilde2} yields $\beta_{U,\alpha,\psi_\nu}\ge c N^{-2-\nu}.$
Since $\alpha$ can be chosen as close to $1$ as desired, Theorem \ref{th-Ktilde2} is indeed almost sharp.
\end{exa}

\section{The Dirichlet-type chain in $U$} \label{sec-Dir}
\setcounter{equation}{0} 

We continue with our general  setup described by the data $(\mathfrak X,\mathfrak E,\pi,\mu)$. We assume that $\mu$ is adapted and that $\mu$ is subordinated to $\pi$. For any finite domain $U$, we consider $K_{D,U}$, the Dirichlet-type kernel in $U$, defined by
\begin{equation} \label{KDU}
K_{D,U}(x,y)= \left\{\begin{array}{ccl}  \mu_{xy}/\pi(x)  & \mbox{ for  } x\neq y \mbox{ with } x,y\in U\\
1-(\sum_{z\in \mathfrak X: z\sim x}\mu_{xz}/\pi(x)) & \mbox{ for } x=y\in U. 
\end{array}\right. \end{equation} 
This is the kernel describing the chain that is killed when it exits $U$. Let us point out  the subtle but essential difference between this definition and that of $K_{N,U}$, the Neumann-type kernel on $U$.  The values of these two kernels  are the same  
when $x\neq y$
or when $x=y$ has no neighbors outside $U$.  But when $x=y$ has a neighbor outside $U$, we have
$$K_{N,U}(x,x)= 1- \left(\sum_{z\in U: z\sim x}\mu_{xz}\right)/\pi(x)  $$
whereas 
$$K_{D,U}(x,x)= 1- \left(\sum_{z\in \mathfrak X: z\sim x}\mu_{xz}\right)/\pi(x) .$$

Because $\mu$ is adapted, at such a point $x$,  
$$\sum_{y\in U}K_{N,U}(x,y)=1 \mbox{ whereas }\sum_{y\in U}K_{D,U}(x,y) <1.$$
In words, the kernel $K_{D,U}$ is not strictly Markovian and the Markov chain corresponding to this kernel includes killing at the boundary.
In terms of the global Markov kernel $K=K_\mu$ defined on $\mathfrak X$ by (\ref{def-Ksub}), we have
$$K_{D,U}= \mathbf 1_U(x)K(x,y)\mathbf 1_U(y).$$
To simplify notation, we  set
$$K_U=K_{D,U}.$$

The goal of this section is to apply the previous results to the study of the iterated kernel $K^t_U(x,y)$. This will be done using the method of Doob's transform explained in more general terms in the next subsection.

\subsection{The general theory of Doob's transform} \label{sec-Doob}

For the purpose of this subsection, we simply assume we are given  a finite or countable state space $\mathfrak X$ equipped with a Markov kernel $K$. We do not assume 
any reversibility.  Fix a finite subset $U$ and consider the restricted kernel
 $$K_U(x,y)=\mathbf 1_U(x)K(x,y)\mathbf 1_U(y). $$
 Throughout this section, we assume that this kernel $K_U$ is irreducible on $U$ in the sense that for any $x,y\in U$ there is an integer $t=t(x,y)$ such that $K_U^t(x,y)>0$. The period $d$ of $K_U$ is the greatest common divisor of $\{t: K_U^t(x,x)>0\}$. Note that $d$ is independent of the choice of $x\in U$.
 When $d=1$ (which is referred to as the aperiodic case), there exists an $N$ such that $K^N_U(x,y)>0$ simultaneously for all $x,y\in U$.  
 We are interested in understanding the behavior of  the chain driven by $K$  on 
 $\mathfrak X$,  started in $U$ and killed at the first exit from $U$.
 If $(X_t)_0^\infty$ denotes the chain driven by $K$ on $\mathfrak X$ 
 and 
 $$\tau=\tau_U=\inf\{t\ge 0: X_t\not \in U\}$$
  is the exit time  from $U$, we would like to have  good approximations for quantities
such as 
$$\mathbf P_x(\tau_U >\ell),\;\;\mathbf P_x( X_t=y \ | \ \tau_U >\ell ),\;\;\mathbf P_x( X_t=y \mbox{ and } \tau>t),$$
for $x,y\in U_N,\;\; 0\le t\le \ell.$ The last of these quantities is, of course,
$$\mathbf P_x( X_t=y \mbox{ and } \tau>t)= K^t_U(x,y).$$
See \cite{CMSM} for a book length discussion of such problems.
The key lemma is the following.

\begin{lem}  \label{lem-Cond}
Assume that $K_U$ is irreducible aperiodic. Let $\beta_0,\phi_0$ denote the Perron-Frobenius eigenvalue  and right eigenfunction of $K_U$.
The limit
 $$\mathbf{P}_x(X_t = y \ | \ \tau_U = \infty) = \lim_{L \rightarrow \infty} \mathbf{P}_x(X_t = y \ | \ \tau_U > L)$$
exists and it is equal to
$$ 
 \mathbf{P}_x (X_t = y \ | \ \tau_U = \infty)= K_{\phi_0}^t(x,y)$$
 where  $K_{\phi_0}$ is the irreducible aperiodic  Markov kernel given by
  \begin{equation} \label{def-Kphi0}
   K_{\phi_0}(x,y)=\beta_0^{-1}\frac{1}{\phi_0(x)}K_U(x,y)\phi_0(y),\;\;x,y\in U.
   \end{equation}
\end{lem}
\begin{rem} When $K_U$ is irreducible but periodic, it still has a unique Perron-Frobenius eigenvalue and right eigenfunction, $\beta_0,\phi_0$, and one can still define the Markov kernel $K_{\phi_0}$ (and use it to study $K_U$), but the limit in the lemma does not typically exist. See Example~\ref{ex-five-vertices} below.
\end{rem}
\begin{rem} \label{rem-Doob}
In general terms, {\em Doob's transform method} studies the Markov kernel $K_{\phi_0}$ in order to study  the iterated kernel 
$K^t_U$.  By definition,
$$K^t_U(x,y)= \beta_0^t \phi_0(x) K^t_{\phi_0}(x,y) \frac{1}{\phi_0(y)}.$$
Let $\phi_0^*$ denote the (positive) left eigenfunction of $K_U$ associated with $\beta_0$. By inspection, the positive function $\phi_0^*\phi_0$, understood as a measure on $U$, is invariant under the action of $K_{\phi_0}$, that is, $$\sum_x\phi_0^*(x)\phi_0(x)K_{\phi_0}(x,y)= \phi_0^*(y)\phi_0(y).$$ This measure can be normalized to provide  the invariant probability measure for the irreducible Markov kernel $K_{\phi_0}$.  We call this invariant probability measure $\pi_{\phi_0}$. It is given by
$$\pi_{\phi_0}=\frac{\phi_0^*\phi_0}{\sum_U\phi_0^*\phi_0}.$$ 
The measure $\pi_{\phi_0}$ is one version of the quasi-stationary distribution (a second version is in Definition~\ref{eq-second-qs} below). The measure $\pi_{\phi_0}$ gives the limiting behavior of the chain, conditioned never to be absorbed. As shown below, it is the key to understanding the absorbing chain as well. The Doob transform is a classical tool in Markov chain theory~\cite[Chapter 8]{kemeny}. For many applications and a literature review see~\cite{Pang}.
\end{rem}
 \begin{proof}[Proof of Lemma \ref{lem-Cond}]
 Fix $T \in \mathbb{N}$ and any $t \leq T$. Temporarily fix $L$, but we will let it tend to infinity.
	\begin{eqnarray}\lefteqn{
	\mathbf{P}_x(X_t = y, \tau_U > t \ | \ \tau_U > L) } && \nonumber \\
	&=& \frac{\mathbf{P}_x(\tau_U > L \ | \ X_t = y, \tau_U > t) \ \mathbf{P}_x(X_t = y, \tau_U > t)}{\mathbf{P}_x(\tau_U > L, \tau_U > t)} \label{eq:bayes-1}
	\end{eqnarray}
	
	We can assume $L>T$, because we will later take the limit as $L$ tends to infinity. So~\eqref{eq:bayes-1}, the identity above, becomes,
	\begin{eqnarray*}\lefteqn{
	\mathbf{P}_x(X_t = y \ | \ \tau_U > L)} &&  \\
	& = & \frac{\mathbf{P}_x(\tau_U > L \ | \ X_t = y, \tau_U > t) \ \mathbf{P}_x(X_t = y, \tau_U > t)}{\mathbf{P}_x(\tau_U > L)}
	\end{eqnarray*}
	or equivalently,
	\begin{equation}
	\label{eq:bayes-2}
	\mathbf{P}_x(X_t = y \ | \ \tau_U > L) = \frac{\mathbf{P}_x(\tau_U > L \ | \ X_t = y, \tau_U > t)}{\mathbf{P}_x(\tau_U > L)}K_U^t(x,y)
	\end{equation}	
	
	Because $(X_t)$ is a Markov chain,
	\begin{align*}
	\frac{\mathbf{P}_x(\tau_U > L \ | \ X_t = y, \ \tau_U > t)}{\mathbf{P}_x(\tau_U > L)} &= \frac{\mathbf{P}_y(\tau_U > L-t)}{\mathbf{P}_x(\tau_U > L)} \\
	&= \frac{\sum_{z \in U}K_U^{L-t}(y,z)}{\sum_{z \in U}K_U^{L}(x,z)} \\
	&= \frac{\sum_{z \in U}\beta_0^{L-t}\phi_0(y)K_{\phi_0}^{L-t}(y,z)\phi_0(z)^{-1}}{\sum_{z \in U}\beta_0^{L}\phi_0(x)K_{\phi_0}^L(x,z)\phi_0(z)^{-1}}.
	\end{align*}
	
	Plugging this into~\eqref{eq:bayes-2}, we have
	\begin{eqnarray}\lefteqn{
	\mathbf{P}_x(X_t = y \ | \ \tau_U > L) = } && \nonumber \\
	&=& \left[\frac{\sum_{z \in U}K_{\phi_0}^{L-t}(y,z)\phi_0(z)^{-1}}{\sum_{z \in U}K_{\phi_0}^L(x,z)\phi_0(z)^{-1}}\right]\beta_0^{-t}\phi_0(x)^{-1}K_U^t(x,y)\phi_0(y)
	\end{eqnarray}
	Now we take the limit as $L$ tends to infinity. To finish the proof of Lemma~\ref{lem-Cond} we need to show that
	\begin{equation}
	\label{eq:frac-limit}
	\lim_{L \rightarrow \infty} \frac{\sum_{z \in U}K_{\phi_0}^{L-t}(y,z)\phi_0(z)^{-1}}{\sum_{z \in U}K_{\phi_0}^L(x,z)\phi_0(z)^{-1}}=1,
	\end{equation}
	which is the content of the following, Lemma~\ref{lem:Kphi0-equals-cond}.
\end{proof}

\begin{lem}
	\label{lem:Kphi0-equals-cond} Assume that $K_U(x,y)$ is irreducible and aperiodic on $U$. Then,
	\begin{equation}\label{eq-equals-cond}
	\lim_{L \rightarrow \infty} \frac{\sum_{z \in U}K_{\phi_0}^{L-t}(y,z)\phi_0(z)^{-1}}{\sum_{z \in U}K_{\phi_0}^L(x,z)\phi_0(z)^{-1}} = 1.
	\end{equation}
\end{lem}
\begin{proof}
	By Remark \ref{rem-Doob}, $K_{\phi_0}$ is an irreducible aperiodic Markov kernel with invariant measure $\pi_{\phi_0}$ proportional to $\phi_0^*\phi_0$.
	By the basic convergence theorem for finite Markov chains (e.g.,~\cite[Thm. 1.8.5]{Norris}),
	$$\lim_{L \rightarrow \infty} K_{\phi_0}^L(x,y) = \pi_{\phi_0}(y).$$
	Applying this to
	$$\frac{\sum_{z \in U}K_{\phi_0}^{L-t}(y,z)\phi_0(z)^{-1}}{\sum_{z \in U}K_{\phi_0}^L(x,z)\phi_0(z)^{-1}},$$
	we can see that both the numerator and denominator approach
	$$\sum_{z \in U} \pi_{\phi_0}(z)\phi_0(z)^{-1}.$$
The stated result follows.\end{proof}

\begin{rem}  If $K_U$ is irreducible and periodic of period $d>1$ then so is $K_{\phi_0}$. The chain driven by $K_{\phi_0}$  has $d$ periodic classes, $C_i$ (with $0\le i\le d-1$) each of which has the same measure, $\pi_{\phi_0}(C_i)=\pi_{\phi_0}(C_0)$, and the limit theorem reads
$$\lim_{ L\rightarrow \infty} K_{\phi_0}^{t+Ld}(x,y) =\left\{ \begin{array}{cl}  \pi_{\phi_0}(y)/d &\mbox{ if }  x\in C_i , y\in C_{i+t}\\
0 & \mbox{ otherwise}. \end{array}\right.$$
Here, $0\le i\le d-1$, and  the index $i+t$ in $C_{i+t}$ is taken modulo $d$.  It follows that, typically, the ratio in Lemma \ref{lem:Kphi0-equals-cond}
has no limit. See below for a concrete example.
\end{rem}

 \begin{figure}[h]
	\begin{center}\label{pic:five-vertices} 
		
\begin{picture}(200,40)
					
			\put(0,20){\line(1,0){200}}
			\put(0,15){\line(0,1){10}}
			\put(50,15){\line(0,1){10}}
			\put(100,15){\line(0,1){10}}
			\put(150,15){\line(0,1){10}}
			\put(200,15){\line(0,1){10}}
			
			\put(-3,10){$\scriptstyle x_0$}
			\put(47,10){$\scriptstyle x_1$}
			\put(97,10){$\scriptstyle x_2$}
			\put(147,10){$\scriptstyle x_3$}
			\put(197,10){$\scriptstyle x_4$}
			
\end{picture}\caption{Simple random walk on five vertices}\end{center}\end{figure}
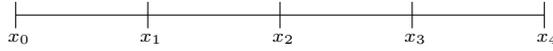

\begin{exa}\label{ex-five-vertices} As a concrete example, consider the simple random walk on five vertices where the boundary vertices have holding probability $\frac{1}{2}$.
$$K(x_i,x_j) = \begin{cases}
\frac{1}{2} & \text{if }|i-j|=1, \ i=j=0, \text{ or } i=j=4 \\
0 & \text{else}
\end{cases}.$$
Let $U = \{x_1,x_2,x_3\}$ be the middle three vertices and define $K_U$ to be sub-Markovian kernel described above. The transition matrix for $K_U$ is given by
$$\begin{bmatrix}
0 & \frac{1}{2} & 0 \\
\frac{1}{2} & 0 & \frac{1}{2} \\
0 & \frac{1}{2} & 0
\end{bmatrix},$$
with largest eigenvalue $\beta_0 = \frac{\sqrt{2}}{2}$ and normalized eigenfunction $$\phi_0 = \begin{bmatrix}
\frac{1}{2} \\ \frac{\sqrt{2}}{2} \\ \frac{1}{2}
\end{bmatrix}.$$
This is a reversible situation (hence, $\phi_0^*=\phi_0$) and the period is $2$ with periodic classes: $C_0 = \{x_2\}$ and $C_1 = \{x_1,x_3\}$.  We have
 $$ \lim_{L\ra \infty}\sum_y K_{\phi_0}^{2L} (x_2,y) \phi_0^{-1}(y)=\sqrt{2}$$
 and
 $$\lim_{L\ra \infty}\sum_y K_{\phi_0}^{2L+1} (x_2,y) \phi_0^{-1}(y)=2,$$
 and hence the ratio in Lemma~\ref{lem:Kphi0-equals-cond} has no limit. 
\end{exa}

Previously, we were considering $\mathbf{P}_x(X_t = y \ | \ \tau_U > L)$, the probability that the process $(X_t)$ equals $y$ at time $t$ and is still inside $U$ at some other time $L$. Now, we consider the case where $t=L$.
\begin{defin}\label{eq-second-qs} Set
	$$\nu_x^t(y) = \mathbf{P}_x(X_t = y \ | \ \tau_U> t), \;\;x,y\in U.$$
This is the second form of quasi-stationary distribution; $\nu_x^t (y)$ describes the chance that the chain is at $y$ at time $t$ (starting from $x$) given that it is still alive.	
\end{defin}

\begin{theo}
	\label{theo:mu-limit}
	Assume that $K_U$ is irreducible and aperiodic. Then
		$$\lim_{t \rightarrow \infty} \nu_x^t(y) = \frac{\phi^*_0(y)}{\sum_{U} \phi^*_0}.$$
\end{theo}
\begin{proof} Write
	\begin{align}
	 \nu^t_x(y) &= \mathbf{P}_x(X_t = y \ | \ \tau_U > t) \nonumber\\
	&= \frac{\mathbf{P}_x(X_t = y, \tau_U > t)}{\mathbf{P}_x(\tau_U > t)} \nonumber\\
	&=  \frac{K^t_U(x,y)}{\sum_{z \in U} K^t_U(x,z)} \nonumber\\
	&= \frac{\beta_0^t\phi_0(x)K_{\phi_0}(x,y)\phi_0(y)^{-1}}{\sum_{z \in U} \beta_0^t \phi_0(x) K_{\phi_0}^t(x,z)\phi_0(z)^{-1}} \nonumber\\
	&=  \frac{K^t_{\phi_0}(x,y)\phi_0(y)^{-1}}{\sum_{z \in U} K_{\phi_0}^t(x,z)\phi_0(z)^{-1}} \label{defn-nu-t}.\end{align}
Taking the limit when $t$ tends to infinity yields
$$\lim_{t\ra \infty}\nu_x^t(y)=
 \frac{\phi_0(y)^{-1}\pi_{\phi_0}(y)}{\sum_{z \in U} \phi_0(z)^{-1}\pi_{\phi_0}(z)}=  \frac{\phi_0^*(y)}{\sum_z \phi_0^*(z)}.$$
 This equality follows from the basic Markov chain convergence theorem~\cite[Theorem 4.9]{LevP}.
	The stated result follows since $\pi_{\phi_0}$ is proportional to $\phi_0^*\phi_0$.
\end{proof}
\begin{theo}\label{thm-nu-control}  Assume that $K_U$ is irreducible and aperiodic. Then the rate of convergence
in
$$\lim_{t \rightarrow \infty} \nu_x^t(\cdot) = \frac{\phi^*_0}{\sum_{U} \phi^*_0}$$
is controlled by that of
$$\lim_{t\ra \infty} K^t_{\phi_0}(x,\cdot)= \pi_{\phi_0}.$$

More precisely, fix $\epsilon >0$.  Assume that $N_{\epsilon}$ is such that, for any 
$t \geq N_{\epsilon}$ and~${y \in U}$,
$$\left|\frac{K^t_{\phi_0}(x,y)}{\pi_{\phi_0}(y)} - 1\right| < \epsilon.$$
Then, for any $t \geq N_{\epsilon}$,
$$\left|\frac{(\sum_U\phi_0^*)\nu_x^t(y) }{\phi_0^*(y)}- 1\right| < \frac{2\epsilon}{1 - \epsilon}.$$
\end{theo}
\begin{proof}
	For a fixed $\epsilon>0$, let $N_{\epsilon}$ be such that, for $t \geq N_{\epsilon}$ and $z \in U$,
	\begin{equation}
	\label{eq:conv-of-one}
	\left|\frac{K^t_{\phi_0}(x,z)}{\pi_{\phi_0}(z)} - 1\right| < \epsilon,
	\end{equation}
	or equivalently,
	$$\left|K^t_{\phi_0}(x,z)\phi_0^{-1}(z) - c\phi_0^*(z)\right| < \epsilon c \phi_0^*(z),$$
	where $c = (\sum_U \phi_0\phi_0^*)^{-1}$ is the normalization constant $\pi_{\phi_0}=c\phi_0\phi_0^*$.
	Summing over all $z \in U$ and applying the triangle inequality,
	\begin{equation}
	\label{eq:conv-of-sum}
	\left|\frac{\sum_{z \in U} K^t_{\phi_0}(x,z)\phi_0^{-1}(z)}{c\sum_{z \in U} \phi^*_0(z)} - 1 \right| < \epsilon.
	\end{equation}
	
	For ease of notation, we abbreviate
	$$a_t = K^t_{\phi_0}(x,y)\phi_0(y)^{-1}, \hspace{.8cm} a = c\phi_0^*(y),$$
	$$b_t  =\sum_{z \in U} K^t_{\phi_0}(x,z)\phi_0(z)^{-1}, \hspace{.8cm} b = c\sum_{z \in U} \phi_0^*(z),$$
	so that~\eqref{eq:conv-of-one} and~\eqref{eq:conv-of-sum} become,
	$$\left|\frac{a_t}{a} - 1\right|<\epsilon \text{ and } \left|\frac{b_t}{b} - 1\right|<\epsilon.$$
	The formula \eqref{defn-nu-t} for $\nu_x^t(y)$
	gives
	$$\frac{(\sum_U\phi_0^*)\nu_x^t(y) }{\phi_0^*(y)}= \frac{(\sum_U\phi_0^*)K_{\phi_0}^t(x,y)\phi_0(y)^{-1}}{\phi_0^*(y)\sum_{z \in U} K_{\phi_0}^t(x,z)\phi_0(z)^{-1}} = 
	\frac{a_t}{b_t}\cdot\frac{b}{a}$$
	and thus
	\begin{align*}
	\left|\frac{(\sum_U\phi_0^*)\nu_x^t(y) }{\phi_0^*(y)}- 1\right|  & = \left|\frac{a_t}{b_t}\cdot\frac{b}{a} - 1\right| 
	= \left|\frac{a_tb - b_ta}{b_ta}\right| \\
	&\leq \frac{b}{b_t}\left(\left|\frac{a_t}{a} - 1\right| + \left|\frac{b_t}{b} - 1\right|\right) \\
	&< \frac{2\epsilon}{1-\epsilon}.
	\end{align*}	
	
\end{proof}	
\subsection{Dirichlet-type chains in John domains}

We return to our main setting of an underlying space $(\mathfrak X,\mathfrak E, \pi, \mu)$ with $\mu$ subordinated to $\pi$ and $K$ defined by this data as in (\ref{def-Ksub}).  For any finite domain $U\subset \mathfrak X$, we consider the kernel $K_U=K_{D,U}$ defined at (\ref{KDU}) and equal to $K_U(x,y)=\mathbf 1_U(x)K(x,y)\mathbf 1_U(y)$. We also let
$\pi_U$ be the probability measure proportional to $\pi|_U$, i.e., $\pi_U(x) = \frac{\pi|_U(x)}{Z}$ where $Z = \sum_{y \in U} \pi|_U(y)$ is the normalizing constant. Let $\phi_0,\phi_0^*$
be the right and left Perron eigenfunctions of the kernel $K_U$  considered in subsection \ref{sec-Doob} above.  By construction, 
$K_U(x,y)/\pi_U(y)$ is symmetric in $x,y$, that is,
$$\pi_U(y)K_U(y,x) =\pi_U(x)K_U(x,y).$$
Multiplying by $\phi_0(y)$ and summing over $y$, we have
$$\sum_y \phi_0(y) \pi_U(y)K_U(y,x) =\pi_U(x)\sum_y K_U(x,y)\phi_0(y)= \beta_0\pi_U(x)\phi_0(x).$$  
This shows that $\phi_0(y)\pi_U(y)$ is proportional to $\phi_0^*(y)$. If we choose to normalize $\phi_0$
by the natural condition $\sum \phi_0^2\pi_U=1$, then  the invariant probability measure 
of the Doob transform kernel $K_{\phi_0}$ at (\ref{def-Kphi0})---which is proportional to $\phi_0^*\phi_0$---is 
$$\pi_{\phi_0}=\phi_0^2\pi_U.$$

Next, observe that, for any $x,y \in \mathfrak X$,
$$\pi(x)K(x,y)=\mu_{xy}$$
and, for any $x,y \in U$,
$$\phi^2(x)\pi|_U(x)K_{\phi_0}(x,y)= \beta_0^{-1}\phi_0(x)\phi_0(y) \pi|_U(x)K(x,y) = \beta_0^{-1}\phi_0(x)\phi_0(y)\mu_{xy}.$$ 
This means that the kernel $K_{\phi_0}$ is obtained as a Markov kernel on the graph 
$(U,\mathfrak E_U)$ using the pair of weights $(\bar{\mu}, \bar{\pi})$  where
$$\begin{cases}
\bar{\mu}_{xy}= \beta_0^{-1}\phi_0(x)\phi_0(y) \mu_{xy} \\ \bar{\pi}= \phi_0^2\pi |_U,
\end{cases}$$ 
i.e., $K_{\phi_0} = \bar{\mu}_{xy}/\bar{\pi}$. Indeed,  for any $x,y \in U$, we have 
$$\bar{\mu}_{xy}= \left(\sum_U\pi\right) \pi_{\phi_0}(x)K_{\phi_0}(x,y) \mbox{ and }  \bar{\pi}(x)= \left(\sum_U\pi\right) \pi_{\phi_0}(x).$$	
Furthermore, $\bar{\mu}$ is subordinated to $\bar{\pi}$  in $U$  because, for any $x\in U$,
$$\sum_{y\in U} \bar{\mu}_{xy}=  \sum_{y\in U} \beta_0^{-1}\phi_0(x)\phi_0(y) \pi|_U(x)K(x,y)
=\phi_0(x)^2\pi |_U(x)=\bar{\pi}(x).$$

All of this means that we are in precisely the situation of Section~\ref{sec-AW}. We now list four assumptions that  will be used to obtain good results concerning the behavior of the chain $(K_{\phi_0},\pi_{\phi_0})$ by applying the techniques described in Section \ref{sec-AW} and Section \ref{sec-Met}. In what follows, we always fix the parameter $\alpha\in (0,1]$ as well as $\theta\ge 2$.

For the reader's convenience we give brief pointers to notation that will be used crucially in what follows: John domains (Section~\ref{sec-JD}), Whitney coverings (Section~\ref{sec-WC}), $D$-doubling (Definition~\ref{defn-doubling}, the ball Poincar\'e inequality (Definition~\ref{defn-ball-poincare}), elliptic (Defintion~\ref{defn-adapted-elliptic-subor}), subordinated weight (Defintion~\ref{defn-adapted-elliptic-subor}), $(\eta,A)$-regular (Defintion~\ref{def-regular}), and $(\eta,A)$-controlled (Defintion~\ref{def-controlled}).

\begin{description}
\item[Assumption A1 (on $(\mathfrak X,\mathfrak E,\pi,\mu)$)]	 The measure $\pi$ is $D$-doubling, $\mu$ is adapted and  the pair  $(\pi,\mu)$ is elliptic and satisfies the $\theta$-Poincar\'e inequality on balls with constant $P$.  In addition,  $\mu$ is subordinated to $\pi$. 	
\item[Assumption A2 (on the finite domain $U$)]  The finite domain  $U\subset \mathfrak X$  belongs to $J(o,\alpha, R)$ for some $o\in U$ with $R(o,\alpha,U)\le R\le 2R(o,\alpha,U)$. 
\item[Assumption A3 (on $U$ in terms of $\phi_0$)] There are $\eta\in (0,1/12]$ and $A\ge 1$ such that  $\phi_0$ is $(\eta,A)$-regular and $A$-doubling on $U$.
\item[Assumption A4 (on $U$ in terms of $\phi_0$)]
There are $\eta\in (0,1/12]$, $\omega\ge 0$, and $A\ge 1$ such that  $\phi_0$ is $(\eta,A)$-regular and $(\omega,A)$-controlled on $U$.
\end{description}

Assumption A1 will be our basic assumption about the underlying weighted graph structure  $(\mathfrak X,\mathfrak E,\pi,\mu)$.  Assumption  A2 is  a strong and relatively sophisticated assumption regarding the geometric properties of the finite domain $U$.  Assumptions 3 and 4 are technical requirements necessary to apply the methods in Sections\ref{sec-PQPJD} and~\ref{sec-AW}. In the classical case when the parameter $\theta$ in the assumed Poincar\'e inequality satisfies $\theta=2$, Assumptions A1-A2  imply that Assumption 
A4 is satisfied. This follows from Lemma \ref{lem-H} below and Lemma \ref{lem-RC}.

\begin{lem} \label{lem-H}
Assume that  {\em A1-A2} are satisfied  and $\theta=2$.  
Then $\phi_0$  is  $(1/8,A)$-regular with $A$  depending only on the  quantities $D,P_e,P$
appearing in Assumption {\em A1}. 
\end{lem}
\begin{proof} The short outline of the proof is that doubling and Poincar\'e (with $\theta=2$) imply  the Harnack inequality
$$\sup_{B}\{ \phi_0\}\le C_H \inf_{B}\{\phi_0\} $$
for any ball $B$ such that $2B\subset U$.  The constant $C_H$ is independent of  $B$ and $U$ and depends only of $D,P_e,P$. This would follow straightforwardly from  Delmotte's elliptic Harnack inequality (see \cite{Delm-EH})   if $\phi_0$ were a positive solution of $$(I-K)u=0$$ in the ball $2B$. However,
$\phi_0$ is a positive solution of $$(I-K) u= (1-\beta_0)u.$$  Heuristically, at scale less than $R$, this is almost the same because Assumption A1 implies that 
$1-\beta_0 \le  C R^{-2}$.  This easy estimate follows by using a tent test function in the ball  $B(o,R/4)\subset U$. To prove the stated Harnack inequality for $\phi_0$, one can either extend Delmotte's argument (adapted from Moser's proof of the elliptic Harnack inequality for uniformly elliptic operators in $\mathbb R^n$), see~\cite{Delm-EH}, or use the more difficult parabolic Harnack inequality of \cite{Delm-PH}.   Indeed, to follow this second approach,
$$v(t,x)=e^{-\frac{1}{2}(1-\beta_0)t}\phi_0(x)\;\;  (\mbox{resp. }\;w(t,x)= (1-\frac{1}{2}(1-\beta_0))^t\phi_0(x))$$
  is a positive  solution of the continuous-time (resp. discrete-time)  parabolic equation
  $$ \left[\partial_t  +\frac{1}{2}(I-K)\right] v=0  \;\;  (\mbox{resp. } \; w(t+1,x)-w(t,x)= - \left[\frac{1}{2}(I-K)w_t\right] (x) )$$ 
in $U$  (in the discrete time case, $w_t=w(t,\cdot)$).  These parabolic equations are associated with the (so-called) lazy version of the Markov kernel $K$, that is,  $\frac{1}{2}(I+K)$ to insure that the results of \cite{Delm-PH} are applicable. The parabolic Harnack inequality in \cite{Delm-PH} necessitates that
the time scale be adapted to the size of the ball on which it is applied, namely, the time scale should be $r^2$ if the ball has radius $r$. Our positive  solution $v(t,x)=e^{-\frac{1}{2}(1-\beta_0)t}\phi_0(x)$ of the heat equation is defined on $\mathbb R\times B$ where $B=B(z,r)\subset U$ is a ball of radius $r$. The parabolic Harnack inequality gives that there is a constant $C_H$ such that, for all $x,y\in B$, 
$$v(r^2,x)\le C_H v(2r^2,y).$$
Because  $1-\beta_0\le C R^{-2}$ and $r\le R$, the exponential factors $$e^{-\frac{1}{2}(1-\beta_0)r^2}, \;\;e^{-(1-\beta_0)r^2}$$ behave like the  constant $1$.  This implies that $\phi_0(x)\approx \phi_0(y)$ for all $x,y\in B$. 
\end{proof}

The following statement is an easy corollary of the last part of the proof of Lemma \ref{lem-H}. See the remarks following the statement.
\begin{lem}  \label{lem-HBB} Fix $\theta\ge 2$. 
Assume that $(\mathfrak X,\mathfrak E,\pi,\mu)$ is such that $\mu$ is adapted,  the pair  $(\pi,\mu)$ is elliptic  and $\mu$ is subordinated to $\pi$. 	
  In addition, assume that the operator $\frac{1}{2}(I+K_\mu) $ satisfies the $\theta$-parabolic inequality {\em PHI($\theta$)} of {\em \cite[(1.9)]{BB}}.  
 If $U$ is a finite domain in $\mathfrak X$ satisfying {\em A2}, the function $\phi_0$ is  $(1/8,A)$-regular with $A$  depending only on the $\theta$, $P_e$, and the constant  $C_H$  from the $\theta$-parabolic Harnack inequality.
\end{lem}
\begin{rem} The $\theta$-parabolic inequality  PHI($\theta$) of {\em \cite{BB}} implies the doubling property and the $\theta$-Poincar\'e inequality (\cite[Theorem 1.5]{BB}).  In addition it implies the so-called cut-off Sobolev inequality $\mbox{CS}(\theta)$ (\cite[Definition 1.4; Theorem 1.5]{BB}). Conversely,  doubling, the $\theta$-Poincar\'e inequality and $\mbox{CS}(\theta)$   imply  PHI($\theta$).   In the case $\theta=2$, the cut-off Sobolev inequality is always trivially satisfied.  When $\theta>2$, the cut-off Sobolev inequality is non-trivial and become essential  to the characterization of the parabolic Harnack inequality PHI($\theta$). See \cite[Theorem 5]{BB} (in \cite{BB}, the parameter $\theta$ is called $\beta$).  \end{rem}
\begin{rem}  To prove Lemma \ref{lem-HBB}, it is essential to have an upper bound $1-\beta_0\le C R^{-\theta}$ on the spectral gap $1-\beta_0$.  This upper bound easily follows from the cut-off Sobolev inequality $\mbox{CS}(\theta)$. 
\end{rem}

We can now state two very general results  concerning the reversible Markov chain $(K_{\phi_0}, \pi_{\phi_0})$ in the finite domain $U$. The first theorem has weaker hypotheses and is, in principle, easier to apply. When the parameter $\omega=0$, the two theorems gives essentially identical conclusions. The proofs are immediate application of the results in Section \ref{sec-AW} and follow the exact same line of reasoning used in Section \ref{sec-Met} to obtain 
 Theorems \ref{th-Ktilde1}-\ref{th-Ktilde2}.  In the following statement, $\beta_-$
 is the least eigenvalue of the pair $(K_{\phi_0},\pi_{\phi_0})$
and $\beta$ is second largest eigenvalue of  $(K_{\phi_0},\pi_{\phi_0})$.
 If $\beta_{U,-}$ denotes the smallest eigenvalue of $K_U$ on $L^2(U,\pi_U)$, then $\beta_-=\beta_{U,-}/\beta_0$.  If $\beta_{U,1}$ denotes the second largest eigenvalue of $K_U$ on $L^2(U,\pi_U)$, then $\beta= \beta_{U,1}/\beta_0$. The eigenfunction $\phi_0$ is normalized by $\pi_U(\phi_0^2)=1$.

 \begin{theo}  \label{th-Doob1}
 Fix $\alpha, \theta, \eta, \omega, P_e,P,D,A$ and assume {\em A1-A2-A4}.  Under these assumptions there  are constants $c,C\in (0,\infty)$ (where $c,C$ depend only on the parameters $\alpha, \theta, \eta, \omega, P_e,P, D$ and $ A$) such that
$$  1- \beta_0\le CR^{-\theta}$$
and
$$  1-\beta\ge c R^{-(\theta+\omega)}.$$
Assume further that $1+\beta_-\ge cR^{-(\theta+\omega)}$. Then,  
 for  all $t\ge R^{\theta+\omega}$, we have the following $L^{\infty}$ rate of convergence,
 $$\max_{x,y\in u}\left|\frac{K^t_{\phi_0}(x,y)}{\pi_{\phi_0}(y)} -1\right| \le C \exp\left(-c\frac{t}{R^{\theta+\omega}}\right).$$
Equivalently, in terms of the kernel $K_U$, this reads 
$$\left|K^t_U(x,y) - \beta_0^t\phi_0(x)\phi_0(y)\pi_U(y)\right|\le C\beta_0^t\phi_0(x)\phi_0(y)\pi_U(y)e^{-ct/R^{\theta+\omega}},$$
for all $x,y\in U$ and $t\ge R^{\theta+\omega}$. \end{theo}

\begin{rem} Part of the proof of this result is to show that there are constants $C,\nu$ such that, for all $t\le R^{\theta+\omega}$ and $x,y\in U$,
$$\frac{K^t_{\phi_0}(x,y)}{\pi_{\phi_0}(y)}\le C (R^{\theta+\omega}/t)^\nu,$$
where $C,\nu$ depends only on the parameters $\alpha, \theta, \eta, \omega, P_e,P, D$ and $ A$. In terms of $K_U^t$, this becomes
for all $t\le R^{\theta+\omega}$ and $x,y\in U$, 
$$\frac{K^t_{U}(x,y)}{\pi_U(y)}\le C (R^{\theta+\omega}/t)^\nu \phi_0(x)\phi_0(y).$$  This type of estimate for $K^t_U$ is called intrinsic ultracontractivity. It first appeared in the context of Euclidean domains in \cite{DavSim,DavUl} (see also \cite{DavisB}) and has been discussed since by many authors. In its classical form, ultracontractivity of the Dirichlet heat semigroup in a bounded Euclidean domain $U$ is the statement that, for each $t>0$, there is a constant $C_t$ such that for all $x,y\in U$,
$$h^D_U(t,x,y)\le C_t\phi_0(x)\phi_0(y)$$
Here $h^D_U(t,x,y)$ is the fundamental solution (e.g., heat kernel) of the heat equation with Dirichlet boundary condition in $U$. Ultracontractivity may or may not hold in a particular bounded domain. It is known that it holds in bounded Euclidean John domains, see \cite{CipUJ}. We note here that running the line of reasoning used here in the case of bounded Euclidean John domains would produce more effective ultracontractivity bounds than the ones reported in \cite{CipUJ}.  
\end{rem}
\begin{rem} As mentioned above, Theorem  \ref{th-Doob1} is relatively easy to apply. Hypothesis A1 is our basic working hypothesis regarding $(\mathfrak X,\mathfrak E,\pi,\mu)$. Hypothesis A2 requires the finite domain $U$ to be a John domain. When $\theta=2$, Hypothesis A4 is automatically satisfied for some $\omega\ge 0$ depending only on the other fixed parameters (Lemma \ref{lem-H}). When $\theta>2$, we would typically appeal to Lemma \ref{lem-HBB} in order to verify A4. This requires an additional assumption on $(\mathfrak X,\mathfrak E,\pi,\mu)$, namely, that $\frac{1}{2}(I+K_\mu)$ satisfies the parabolic Harnack inequality PHI($\theta$) of \cite{BB}.
For instance, Theorem \ref{th-Doob1} applies uniformly to the graph metric balls in $(\mathfrak X,\mathfrak E,\pi,\mu)$ under Hypothesis A1 when $\theta=2$, and under A1 and PHI($\theta$) when $\theta>2$.  Consider the infinite Vicsek fractal graph $(\mathfrak X^V,\mathfrak E^V)$ (a piece of which is pictured in Figure \ref{fig-V2}) equipped the vertex weight $\pi^V(x)=4$, $x\in \mathfrak X^V$ and the edge weight $\mu^V_{xy}=1$, $\{x,y\}\in \mathfrak E^V$. This structure is a good example for the case $\theta>2$.
It has $\theta=d+1$ where $d=\log 5/\log 3$ and also volume growth $\pi(B(x,r))\asymp r^d$.
It satisfies the parabolic Harnack inequality PHI$(\theta)$. See, e.g., \cite[Example 2 and Example 3, Section 5]{BCK} which provides larger classes of examples of this type.
\end{rem}

\begin{theo} \label{th-Doob2} Fix $\alpha, \theta, \eta, P_e,P,D,A$ and assume {\em A1-A2-A3}.  Under these assumptions there  are constants $c,C\in (0,\infty)$ (where $c,C$ depend only on  the parameters $\alpha, \theta, \eta, P_e,P, D$ and $ A$) such that
$$  1- \beta_0\le CR^{-\theta}$$
and
$$  1-\beta\ge c R^{-\theta}.$$
 Assume further that $1+\beta_-\ge cR^{-\theta}$.
 Then, for all $t\ge R^{\theta}$, we have
 $$\max_{x,y\in u}\left|\frac{K^t_{\phi_0}(x,y)}{\pi_{\phi_0}(y)} -1\right| \le C \exp\left(-c\frac{t}{R^{\theta}}\right).$$
Equivalently, in terms of the kernel $K_U$, this reads
$$\left|K^t_U(x,y) - \beta_0^t\phi_0(x)\phi_0(y)\pi_U(y)\right|\le C\beta_0^t\phi_0(x)\phi_0(y)\pi_U(y)e^{-ct/R^{\theta}},$$
for all $x,y\in U$ and $t\ge R^{\theta}$.
 \end{theo}
 \begin{rem} As for Theorem \ref{th-Doob1},
 part of the proof of Theorem \ref{th-Doob2} is to show that there are constants $C,\nu$ such that, for all $t\le R^{\theta}$ and $x,y\in U$,
$$\frac{K^t_{\phi_0}(x,y)}{\pi_{\phi_0}(y)}\le C (R^{\theta}/t)^\nu,$$
where $C,\nu$ depends only on the parameters $\alpha, \theta, \eta, P_e,P, D$ and $ A$. In terms of $K_U^t$, this gives the intrinsic ultracontractivity estimate
for all $t\le R^{\theta}$ and $x,y\in U$, 
$$\frac{K^t_{U}(x,y)}{\pi_U(y)}\le C (R^{\theta}/t)^\nu \phi_0(x)\phi_0(y).$$  
 \end{rem}
\begin{rem}Theorem~\ref{th-Doob2} gives a more satisfying result than Theorem \ref{th-Doob1} in that it does not involves the extra parameter $\omega$ (the two theorems have the same conclusion when $\omega=0$). However, Theorem \ref{th-Doob2} requires to verify Hypothesis A3, that is, to show that  $\pi_{\phi_0}$ is doubling. This is an hypothesis that is hard to verify, even for simple finite domains in $\mathbb Z^d$. At this point in this article, the only finite domains
in $\mathbb Z^2$ for which we could verify this hypothesis are those where we can compute $\phi_0$ explicitly such as cubes with sides parallel to the axes or the $45$ degree finite cone of Figure \ref{D0}. This shortcoming will be remedied in the next section when we show that finite  inner-uniform domains satisfy Hypothesis A3 (see Theorem~\ref{theo-Carleson}).
\end{rem}

\begin{exa} We can apply either of these two theorems to the one dimensional example of simple lazy random walk on $\{0,1,\dots,N\}$ with absorption at $0$ and reflection at $N$. This is the leading example of \cite{DM} where quantitative estimates for absorbing chains are discussed.  In this simple example, we know exactly the function $\phi_0$ and we can easily verify A1-A2-A3 and A4 with $\omega=0$. In terms of the Doob-transform chain $K_{\phi_0}$ and its invariant measure $\pi_{\phi_0}$, the result above proves convergence after order $N^2$ steps. This improves upon the results of \cite{DM} by a factor of $\log N$.
\end{exa}

\begin{exa} In the same manner, we can apply the two theorems above to the example discussed in the introduction (Figure \ref{D0}). The key is again the fact that we can find an explicit expression for the eigenfunction $\phi_0$ and that it follows that Assumptions A1-A1-A3 and A4 with $\omega=0$  are satisfied.  The conclusion is the same. In terms of the Doob-transform chain  $K_{\phi_0}$ and its invariant measure $\pi_{\phi_0}$, the result above proves convergence after order $N^2$ steps. 
\end{exa}

\begin{exa} Let us focus on the square grid $\mathbb Z^m$ in a fixed dimension $m$ and on the family of its finite $\alpha$-John domains for some fixed $\alpha\in (0,1]$.
In addition, for simplicity, we assume that the weight $\mu$ is constant equal to $1/4m$ one the grid edges and $\pi\equiv 1$ (this insure aperiodicity of $K$ and $K_U$).
Obviously, A1 is satisfied with $\theta=2$ and  A2  is assumed since $U$ is an $\alpha$-John domain. Theorem \ref{th-Doob2} does not apply here because we are not able to prove doubling of the measure $\pi_{\phi_0}$ (and in fact, doubling should probably not be expected in this generality). However, there is an $\omega$ (which depends only on the two fixed parameters $m$ and $\alpha$) such that A4 is satisfied (this follows from Lemma \ref{lem-RC} and Lemma \ref{lem-H}), and hence, we can apply Theorem~\ref{th-Doob1}.
\begin{theo} \label{th-JZd}
Fix $m$ and $\alpha\in (0,1]$. Let the square grid $\mathbb Z^m$ be equipped with the weights $\mu,\pi$ described above. There are constants $c=c(m,\alpha), C=C(m,\alpha)$ and $\omega=\omega(m,\alpha)$ such that, for any finite $\alpha$-John domain $U$ in $\mathbb Z^m$ with John radius $R_U=R(o,\alpha,U)$, the Doob-transform chain $K_{\phi_0}$ satisfies $$ cR_U^{-2}\le1-\beta_0\le CR_U^{-2},$$ 
$$1-\beta\ge c R_U^{-2-\omega},$$
and, for $t\ge R^{2+\omega}$ 
$$\max_{x,y\in u}\left|\frac{K^t_{\phi_0}(x,y)}{\pi_{\phi_0}(y)} -1\right| \le C \exp\left(-c\frac{t}{R_U^{2+\omega}}\right).$$
Equivalently, in terms of the kernel $K_U$, this reads 
$$\left|K^t_U(x,y) - \beta_0^t\phi_0(x)\phi_0(y)\pi_U(y)\right|\le C\beta_0^t\phi_0(x)\phi_0(y)\pi_U(y)e^{-ct/R^{2+\omega}},$$
for all $x,y\in U$ and $t\ge R^{2+\omega}$.Moreover, for $1\le t\le R^{\theta+\omega}$, we have
$$\frac{K_U^t(x,y)}{\pi_U(y)}\le C\left(R^{2+\omega}/t\right)\phi_0(x)\phi_0(y)\pi_U(y).
$$\end{theo}
It is an open question whether or not it is possible to prove the above theorem with $\omega(m,\alpha)=0$ for all finite $\alpha$-John domains in $\mathbb Z^m$ or, even more generally, for a general underlying structure $(\mathfrak X,\mathfrak E,\pi,\mu)$ under assumption $A1$ with $\theta=2$.
\end{exa} 

\begin{rem}
    Recall from Definition~\ref{eq-second-qs} that $\nu_x^t(y) = \mathbf{P}_x(X_t=y \ | \ \tau_u > t)$. Theorem~\ref{thm-nu-control} gives control on the rate of convergence of $\nu_x^t(y)$ in terms of the rate of convergence of $K_{\phi_0}^t(x,y)$. We can now apply Theorem~\ref{thm-nu-control} in each of the settings described above in Theorems~\ref{th-Doob1},~\ref{th-Doob2}, and~\ref{th-JZd}. For example, in the case of the square grid $\mathbb Z^m$ and for a fixed $\alpha\in (0,1)$, there exists $\omega=\omega(m,\alpha)\ge 0$ and $C=(m,\alpha), c=c(m,\alpha)>0$ such that, for any finite $\alpha$-John domain $U$ with John radius $R$,
    $$\forall\, t>CR^{2+\omega},\;\;  \left|\frac{\phi_0(y)\nu_x^t(y)}{\sum_U\phi_0}-1\right|\le e^{-ct/R^{2+\omega}}$$

\end{rem}

\section{Inner-uniform domains} \setcounter{equation}{0}
\label{sec-IU}

We now turn to the definition of inner $(\alpha,A)$-uniform domains. These domains form a subclass of the class of  $\alpha$-John domains.   They allow for a much more precise analysis of Metropolis-type chains and their Doob-transforms.

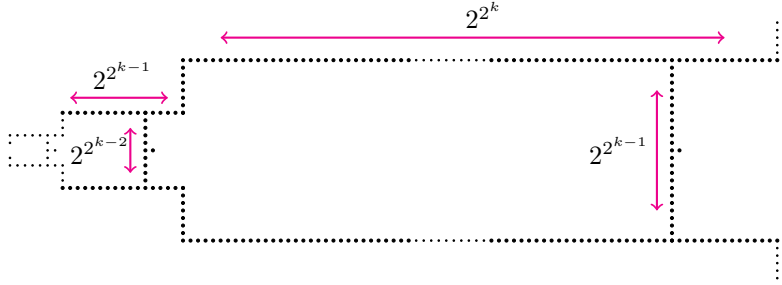
\begin{figure}[h] 
\begin{center}
\begin{tikzpicture}[scale=.1]

{\foreach \x in {4,5,6,7,8,9,10,11,12,13,14,15,16,17,18,19,20}
{\draw  [fill] (\x+5,25)  circle [radius=.2];
\draw  [fill] (\x+5,15)  circle [radius=.2];}}

{\foreach \x in {4,5,6,7,8,9,10,11}
{\draw  [fill] (25,\x+21)  circle [radius=.2];
\draw  [fill] (25,19-\x)  circle [radius=.2];}}

{\foreach \x in {21,22,23,24,25,26,27,28,29,30,31,32,33,34,35,36,37,38,39,40,41,42,43,44,45,46,47,48,49,50,61,62,63,64,65,66,67,68,69,70,71,72,73,74,75,76,77,78,79,80,81,82,83,84,85,86,87,88,89,90,91,92,93,94,95,96,97,98,99}
{\draw  [fill] (\x+5,32)  circle [radius=.2];
\draw  [fill] (\x+5,8)  circle [radius=.2];}}

{\foreach \x in {51,52,53,54,55,56,57,58,59,60}
{\draw  [fill] (\x+5,32)  circle [radius=.1];
\draw  [fill] (\x+5,8)  circle [radius=.1];}}

{\foreach \x in {2,3,4,5,6,7,8,9} {\draw [fill] (\x,22) circle[radius=.1];
 \draw [fill] (\x,18) circle [radius=.1];}}
\draw [fill] (9,17)  circle [radius=.1]; 
 \draw [fill] (9,16)  circle [radius=.1]; 
 \draw [fill] (9,23)  circle [radius=.1]; 
 \draw [fill] (9,24)  circle [radius=.1]; 
 \draw [fill] (8,20)  circle [radius=.1]; 
 \draw [fill] (7,19)  circle [radius=.1]; 
 \draw [fill] (7,20)  circle [radius=.1]; 
 \draw [fill] (7,21)  circle [radius=.1]; 
  \draw [fill] (2,21)  circle [radius=.1]; 
 \draw [fill] (2,19)  circle [radius=.1];  
  
 {\foreach \x in {0,1,2,3,4,5,6,7,8,9}
{\draw [fill] (20,16+\x) circle [radius=.2];
}}
{\foreach \x in {0,1,2,3,4,5,6,7,8,9,10,11}{
\draw [fill] (90,20+\x) circle [radius=.2];
\draw [fill] (90,8+\x) circle [radius=.2];
}}

{\foreach \x in {0,1,2,3,4}{
\draw [fill] (104,33+\x) circle [radius=.1];
\draw [fill] (104,7-\x) circle [radius=.1];
}}

\draw [fill] (21,20) circle [radius=.2];
\draw [fill] (91,20) circle [radius=.2];

\draw [thick,magenta,<->] (30,35) -- (97,35);
\node at (65,38) {$2^{2^k}$};

\draw [thick,magenta,<->] (88,12) -- (88,28);
\node at (83,20) {$2^{2^{k-1}}$};

\draw [thick,magenta,<->] (18,17) -- (18,23);
\node at (14,20) {$2^{2^{k-2}}$};

\draw [thick,magenta,<->] (10,27) -- (23,27);
\node at (17,30) {$2^{2^{k-1}}$};

\end{tikzpicture}
\caption{A graph in which, at large scales, some balls are not inner-uniform.
To the left, the graph ends after finitely many step with an origin $o$ which serves as the center of the balls to be considered. To the right, the indicated pattern is repeated infinitely many times at larger and larger scales. This graph is roughly linear. It satisfies doubling and Poincar\'e.} \label{fig-ballnotIU1}

\end{center}
\end{figure}

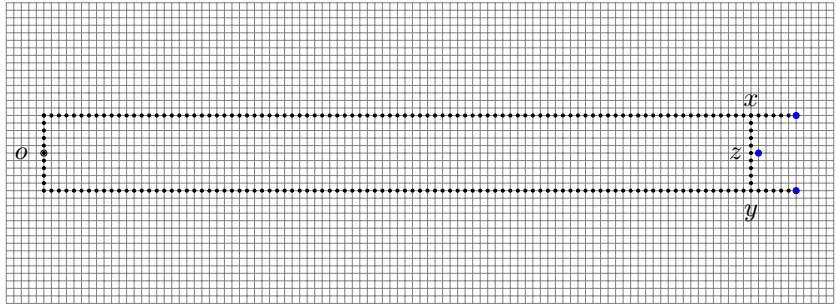
\begin{figure}[t] 
\begin{center}
\begin{tikzpicture}[scale=.1]
\draw [help lines] (0,0) grid (110,40);
{\foreach \x in {0,1,2,3,4,5,6,7,8,9,10,11,12,13,14,15,16,17,18,19,20,21,22,23,24,25,26,27,28,29,30,31,32,33,34,35,36,37,38,39,40,41,42,43,44,45,46,47,48,49,50,51,52,53,54,55,56,57,58,59,60,61,62,63,64,65,66,67,68,69,70,71,72,73,74,75,76,77,78,79,80,81,82,83,84,85,86,87,88,89,90,91,92,93,94,95,96,97,98,99,100}
{\draw  [fill] (\x+5,25)  circle [radius=.2];
\draw  [fill] (\x+5,15)  circle [radius=.2];}}

{\foreach \x in {0,1,2,3,4,5,6,7,8,9}
{\draw [fill] (5,16+\x) circle [radius=.2];
\draw [fill] (99,16+\x) circle [radius=.2];
}}
\draw (5,20) circle [radius=.4];
\draw [fill, blue] (105,15) circle [radius=.4];
\draw [fill, blue] (105,25) circle [radius=.4];
\draw [fill, blue] (100,20) circle [radius=.4];
\node at (2,20) {$o$};
\node at (99,27) {$x$};
\node at (99,12) {$y$};
\node at (97,20) {$z$};
\end{tikzpicture}
\caption{The basic model for the balls in the graph of Figure \ref{fig-ballnotIU1}. The shortest path from $x$ to $y$ is much shorter than any other path but its middle point is at distance $1$ from the boundary of the ball marked by blue dots.} \label{fig-ballNotIU2}

\end{center}
\end{figure}

Although the definition of inner-uniform domains given below appears to be quite similar to that of
John domains, it is in fact much harder, in general circumstances, to find inner-uniform domains than it is to find John domains. In the square lattice $\mathbb Z^d$, both classes of domains are very large and contain many interesting natural examples. Things are very different if one consider an abstract graph structure $(\mathfrak X,\mathfrak E)$ of the type used in this paper.  We noted earlier than any graph distance ball $B(o,r)$ in such a structure $(\mathfrak X,\mathfrak E)$ is a $1$-John domain. In particular, $\mathfrak X$ admits an exhaustion  $\mathfrak X=\lim_{r\rightarrow \infty}B(o,r)$ by finite $1$-John domains. We know of no constructions of an increasing  family of $\alpha$-inner-uniform domains in $(\mathfrak X,\mathfrak E)$, in general. Even if we assume additional properties such as doubling and Poincar\'e inequality on balls, we are not aware of a general method to construct inner-uniform subsets.  Of course, it may happen that, as in the case of $\mathbb Z^d$, graph balls turn out to be inner-uniform
(all for some fixed $\alpha>0$). But that is not the case in general.
Figures \ref{fig-ballnotIU1} and~\ref{fig-ballNotIU2} describe a simple planar graph in which, there are balls $B(o,r_i)$ with $r_i$ tending to infinity which each contains points 
$x_i,y_i$ such that $d_{B(o,r_i)}(x_i,y_i)=\rho_i=o(r_i)$ but the only path from $x_i$ to $y_i$ of length $O(\rho_i)$ has a middle point $z_i$ which is at distance $1$ from the boundary. all other paths from $x_i$ to $y_i$ have length at least $r_i/8$.  This implies that the inner-uniformity constant $\alpha_i$ of the ball $B(o,r_i)$ is $O(\rho_i/r_i)\le o(1)$.  The graph in question has a very simple structure and it satisfies doubling and the Poincar\'e inequality on balls at all scales.

\subsection{Definition and main convergence results} 
\begin{defin} A domain $U$  in $\mathfrak X$ is an inner $(\alpha,A)$-uniform domain (with respect to the graph structure $(\mathfrak X,\mathfrak E)$)
	if  for any two points $x,y\in U$ there exists a path $\gamma_{xy}=(x_0=x,x_1,\dots,x_k=y)$ joining $x$ to $y$ in  $(U, \mathfrak E_U)$ with the properties that:
	\begin{enumerate}
		\item  $k  \le A d_U(x,y)$;
		\item  For any $j \in \{0,\ldots, k\}$, $d (x_j,\mathfrak X \setminus U) \ge  \alpha (1+\min \{ j,k-j\})$. 
	\end{enumerate}
\end{defin}

\begin{rem}  Because the second condition must hold for all $x$, including those that are distance $1$ from the boundary, we see that $\alpha\in (0,1]$.
\end{rem}

We can think of an inner-uniform domain $U$ as being one where any two points are connected by a banana-shaped region. The entire banana must be contained within $U$. See Figure~\ref{Banana} for an illustration.

    There is an alternative and equivalent (modulo a change in $\alpha$)
	definition of inner-uniformity which uses distance instead of path-length in the second condition. More precisely, in this alternative definition, the condition
	``for any $j=0,\dots, k$, $d (x_j,\mathfrak X \setminus U) \ge  \alpha (1+\min \{ j,k-j\})$'' is replaced by  ``for any $j=0,\dots, k$, $d(x_j,\mathfrak X \setminus U)\ge \alpha' \min\{d_U(x_j,x),d_U(x_j,y)\}$''.  It is obvious that the definition we choose here easily implies the condition of the alternative definition (with $\alpha=\alpha'$). The reverse implication is much less obvious. It amounts to showing that it is possible to choose the path $\gamma_{xy}$  so that any of its segments $(x_i,x_{i+1},\dots,x_j)$ provide approximate geodesics between its end-points. This requires a modification (i.e., straightening) of the path $\gamma_{xy}$ provided by the definition because there is no reasons these paths have this property. See \cite{MS}.

The following lemma shows that all inner-uniform domains are John domains. However, the converse is not true. See Figure~\ref{IU-J}.

\begin{figure}[h] 
\begin{center}
\begin{tikzpicture}[scale=.17]
\draw[thick]  (6,10) -- (1,11) --(3,19) --(29,19)--( 29,2) -- (3,2) -- (1,9) -- (6,10);

\draw[fill] (2,8) circle [radius=.1];
\draw[fill] (3,5) circle [radius=.1];
\path[fill=yellow=20!, opacity=.8]  (2,8) to [out=-10, in= 90]  (5.5,6)  to [out=-90, in=0]  (3,5)  to [out =10, in=-90] (4.5,6)  to [out=90, in=-30] (2,8) ;

\draw[fill] (4,9) circle [radius=.1];
\draw[fill] (3,11) circle [radius=.1];
\path[fill=yellow=20!, opacity=.8]  (4,9) to [out=-40, in= -90]  (10.5,10)  to [out=90, in=40]  (3,11)  to [out =20, in=90] (8.5,10)  to [out=-90, in=-20] (4,9) ;

\draw[fill] (28,15) circle [radius=.1];
\draw[fill] (27,8) circle [radius=.1];
\path[fill=yellow=20!, opacity=.8]  (28,15) to [out=180, in= 90]  (22,11)  to [out=-90, in=-180]  (27,8)  to [out =160, in=-90] (24,11)  to [out=90, in=-160] (28,15) ;

\draw[thick]  (46,10) -- (41,11) --(43,19) --(69,19)--( 69,2) -- (43,2) -- (41,9) -- (46,10);

\draw[fill] (42,8) circle [radius=.1];
\draw[fill] (43,5) circle [radius=.1];
\path[fill=orange=10!, opacity=.2]  (42,8) to [out=-40, in= 180]  (56,9)  to [out=0, in =0] (56,13)  to [out =180, in=-20] (42,8) ;
\path[fill=orange=10!, opacity=.2]  (43,5) to [out=-10, in= 180]  (56,9)  to [out=0, in =0] (56,13)  to [out =180, in=10] (43,5) ;

\draw[fill] (44,9) circle [radius=.1];
\draw[fill] (43,11) circle [radius=.1];
\path[fill=orange=10!, opacity=.2]  (44,9) to [out=-55, in= 180]  (56,9)  to [out=0, in =0] (56,13)  to [out =180, in=-40] (44,9) ;
\path[fill=orange=10!, opacity=.2]  (43,11) to [out=30, in= 180]  (56,9)  to [out=0, in =0] (56,13)  to [out =180, in=45] (43,11) ;

\draw[fill] (68,15) circle [radius=.1];
\draw[fill] (67,8) circle [radius=.1];
\path[fill=orange=10!, opacity=.2]  (68,15) to [out=170, in= 0]  (56,13)  to [out=180, in =180] (56,9)  to [out =0, in=-160] (68,15) ;
\path[fill=orange=10!, opacity=.2]  (67,8) to [out=-170, in= -10]  (56,9)  to [out=180, in =1800] (56,13)  to [out =0, in=160] (67,8) ;

\draw[fill=red] (56,11) circle [radius=.2];

\end{tikzpicture}
\caption{On the left: The banana regions for arbitrary pairs of points which are the witnesses for the inner-uniform property. On the right: The carrot regions joining arbitrary vertices to the central point $o$ marked in red. They are witnesses for the John domain property.} \label{Banana}

\end{center}
\end{figure}
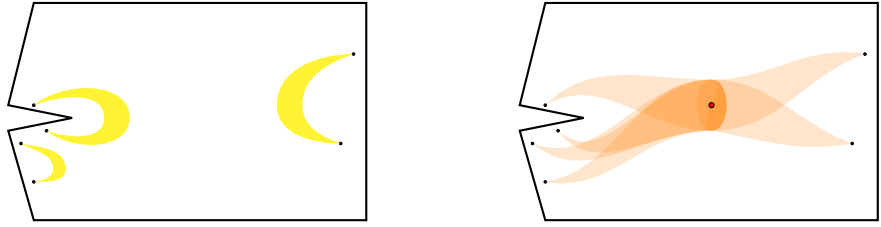

\begin{lem} \label{lem-IUJ}
	Suppose that $U$ is a finite  inner $(\alpha,A)$-uniform domain. Let $o$ be a point such that $d(o,\mathfrak X \setminus U)=\max\{d(x,\mathfrak X \setminus U):x\in U\}$, and let $R = d(o,\mathfrak X \setminus U)$. Then  
	$U\in J(\mathfrak X,\mathfrak E,o, \alpha^2/8,  (2/\alpha)R)$, that is, $U$ is an ($\alpha^2/8$)-John domain.
\end{lem}
\begin{proof}  Look at the mid-point  $z=x_{\lfloor k/2\rfloor}$ along $\gamma_{xo}$. We have ${R\ge d(z,\mathfrak X \setminus U)\ge \alpha k/2}$ so the   $k\le (2/\alpha)R$. 
	We consider three cases to find a lower bound on $d(x_j,\mathfrak X \setminus U)$ along $\gamma_{xo}$.
	\begin{enumerate}
		\item  When $j\le k/2$, then we have $d(x_j,\mathfrak X \setminus U)\ge \alpha (1+j)$.
		\item When $x_j\in B(o,R/2)$, then we have that 
		$$d(x_j,\mathfrak X \setminus U)\ge R/2 \ge (\alpha/4) k\ge  (\alpha/8)(1+j).$$
		\item When  $x_j\not \in B(o,R/2)$ and $j\ge k/2$, then $k-j\ge R/2$ and $$d(x_j,\mathfrak X \setminus U)\ge \alpha(1+ k-j)\ge \alpha (1+R/2)\ge  \alpha(1+ \alpha k/4)\ge (\alpha^2/4)(1+j).$$
	\end{enumerate}
	
\end{proof}

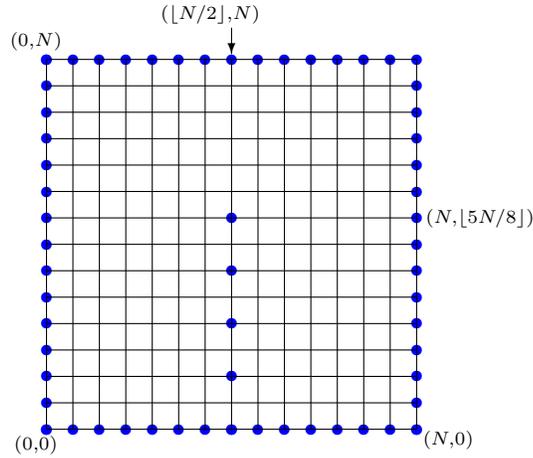
\begin{figure}[h]
	\begin{center}
		\begin{picture}(200,200)(-20,-20)

		{\color{blue}   \multiput(0,0)(10,0){15}{\circle*{4}}
			\multiput(0,0)(0,10){15}{\circle*{4}}
			\multiput(0,140)(10,0){15}{\circle*{4}}
			\multiput(140,0)(0,10){14}{\circle*{4}}
			
			\multiput (70,80)(0,-20){5}{\circle*{4}}  
			
		}

		\multiput(0,0)(10,0){15}{\line(0,1){140}}
		\multiput(0,0)(0,10){15}{\line(1,0){140}}

		\put(-6,-8){\makebox(4,4){$\scriptstyle (0,0)$}}
		\put(-6,145){\makebox(4,4){$ {\scriptstyle (0,N)}$}}
		\put(150,-6){\makebox(4,4){$\scriptstyle (N,0)$}}
		\put(162,78){\makebox(4,4){$\scriptstyle (N,\lfloor 5N/8\rfloor)$}}\put(60,155){\makebox(4,4){$\scriptstyle (\lfloor N/2\rfloor,N)$}}
		\put(70,152){\vector(0,-1){10}}
		
		\end{picture}
		\caption{A domain that is John but not inner-uniform. The blue dots are the boundary. Note that, on the middle vertical line, the blue dots are placed on every other vertex, up to the indicated height.}\label{IU-J}
		\end{center}\end{figure}

\begin{rem} The word ``inner'' in inner-uniform refers to the fact that the first condition 
	compares the length $k$ of the curve $\gamma_{xy}$ to the inner-distance $d_U(x,y)$ between $x$ and $y$. If, instead, the original distance $d(x,y)$ is used (i.e., the first condition in the definition becomes $k \leq Ad(x,y)$), then we obtain a much more restrictive class of domains called ``uniform domains.'' See Figure~\ref{IU-U}.
\end{rem}

\begin{figure}[h] \label{fig-ConvZ2IU}
\begin{center}
\begin{tikzpicture}[scale=.1]

\path [fill=lightgray]  (8,30) arc [radius=22, start angle=180, end angle=90] (30,52)  --
 (38,49) -- (53,33) -- (50,22) --   (38,10) -- (18,10) -- (8,30);
 \draw [fill] (30,30)  circle  [radius=.2];
 \draw (30,30)  circle  [radius=24];      \draw [help lines] (0,0) grid (60,60);   
  \draw (30,30)  circle  [radius=18];   
 \draw [fill] (48,37) circle [radius=.2];
 \draw [fill] (51, 30) circle [radius=.2];
 \foreach \x in {1,2,3,4,5,6,7,8,9}
 \draw [fill] (51-\x,30) circle [radius=.1];
 \draw [fill] (47,37) circle [radius=.1];
 \draw [fill] (46,37) circle [radius=.1];
 \draw [fill] (46,36) circle [radius=.1];
 \draw [fill] (45,36) circle [radius=.1];
 \draw [fill] (44,36) circle [radius=.1];    
 \draw [fill] (44,35) circle [radius=.1];  
 \draw [fill] (43,35) circle [radius=.1];  
 \draw [fill] (42,35) circle [radius=.1];     
 \draw [fill] (41,35) circle [radius=.1];    
 
   \foreach \x in {1,2,3,4,5}
 \draw [fill] (41,35-\x) circle [radius=.1]; 
 \draw ((30,30) -- (48,37);
  
   \end{tikzpicture}

\caption{Finite discrete ``convex subsets'' of $\mathbb Z^2$ are (inner-)uniform}
\end{center}
\end{figure}
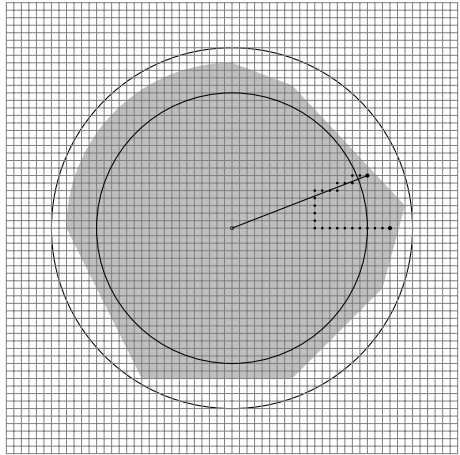

\begin{figure}[t]
	\begin{center} 
		\begin{picture}(200,200)(-20,-20)

		{\color{blue}   \multiput(0,0)(10,0){15}{\circle*{4}}
			\multiput(0,0)(0,10){15}{\circle*{4}}
			\multiput(0,140)(10,0){15}{\circle*{4}}
			\multiput(140,0)(0,10){14}{\circle*{4}}
			
			\multiput (40,130)(0,-10){4}{\circle*{4}}  
			\multiput (90,100)(0,-10){11}{\circle*{4}}  
			
		}

		\multiput(0,0)(10,0){15}{\line(0,1){140}}
		\multiput(0,0)(0,10){15}{\line(1,0){140}}

		\put(-6,-8){\makebox(4,4){$\scriptstyle (0,0)$}}
		\put(-6,145){\makebox(4,4){$ {\scriptstyle (0,N)}$}}
		\put(150,-6){\makebox(4,4){$\scriptstyle (N,0)$}}
		\put(33,155){\makebox(4,4){$\scriptstyle (\lfloor N/3\rfloor,N)$}}
		\put(40,152){\vector(0,-1){10}}
		\put(162,98){\makebox(4,4){$\scriptstyle (N,\lfloor 2N/3\rfloor)$}}
		
		\put(80,155){\makebox(4,4){$\scriptstyle (\lfloor 2N/3\rfloor,N)$}}
		\put(90,152){\vector(0,-1){10}}
		
		\end{picture}
		\caption{A domain that is inner-uniform but not uniform.}\label{IU-U}
		\end{center}\end{figure}
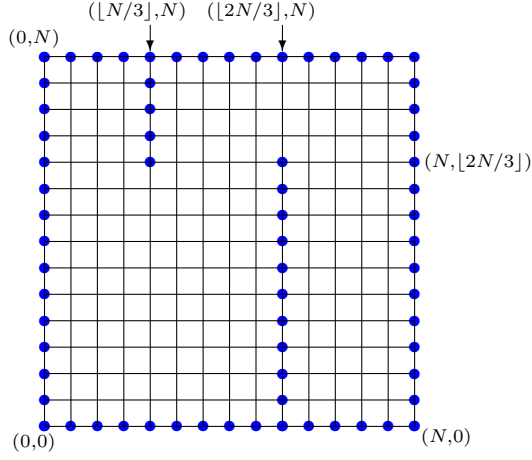

\begin{exa} The set $\mathfrak X_N$ in Figure~\ref{D0} (a forty-five degree finite cone in $\mathbb{Z}^2$)
	is a uniform domain, and hence, also an inner-uniform domain.  Finite convex sets in $\mathbb Z^d$ in the sense of Example~\ref{ex:convex} are uniform domains, all with the same fixed $(\alpha,A)$ depending only on the dimension $d$. The domain pictured in Figure \ref{D3} is a uniform domain, with the same fixed $(\alpha,A)$ for all $N$. Note that in this example, viewed as a subset of $\mathbb Z^2$, some of the boundary points are not killing points, but points where the process is reflected. This illustrates how variations of this type (i.e., with reflecting points) can be treated with our methods.
\end{exa}

\begin{exa} In Example \ref{exa-mb}, we observed that metric balls are always $1$-John domains. They are not always inner-uniform domains. See Figures~\ref{fig-ballnotIU1} and~\ref{fig-ballNotIU2}.
\end{exa}

\begin{exa} The discrete ``finite convex subsets" $U$ of $\mathbb Z^d$
satisfying \eqref{eq-alphaConv} and considered in Proposition \ref{pro-Conv} are inner-uniform with parameter $\bar{\alpha}>0$ 
depending only on the dimension $d$ and the parameter $\alpha$ in \eqref{eq-alphaConv}.  Note that the inner distance in such a finite connected set is comparable to the graph distance of $\mathbb Z^d$ with comparison constant depending only on the dimension $d$ and the parameter $\alpha$ in \eqref{eq-alphaConv} (i.e., these finite domains are {\em uniform}). 

Here is a rough description of the paths $\gamma_{xy}$ that demonstrate that such domains $U$ are inner-uniform. (See Figure \ref{fig-ConvZ2IU}). Let $r$
be the distance between $x$ and $y$ in $\mathbb Z^d$. Recall that $U$ has ``center'' $o$ and that we can go from $x$ (and $y$) to $o$ while getting away linearly from the boundary, roughly along a straight-line (see Proposition \ref{pro-Conv}). Let $\tilde{x}$ and $\tilde{y}$ be respective points along the paths joining $x$ and $y$ to $o$, respectively at distance $r$ from $x$ and from $y$. Convexity insures that there is a discrete path in $U$ joining $\tilde{x}$ to $\tilde{y}$ while staying close to the straight-line segment between these two points. This discrete path from $\tilde{x}$ to $\tilde{y}$ has length at most $Ar$ and stays at distance at least $ar$ from the boundary. This completely the discussion of the example.
\end{exa}

Now we return to the general setting. We define a special point $x_r$ for each point $x \in U$ and radius $r >0$. The meaning of this definition and the key geometric property of  $x_r$ is that
$x_r$ is a point which is essentially as far away from the boundary as possible while still being within a ball of radius $r$ of $x$, i.e., $d(x,x_r) \leq r$. Namely,  ${d(x_r,\mathfrak X \setminus U)\ge \alpha( 1+r)}$  
if $r\le R$ and $x_r=o$ otherwise.
\begin{defin} \label{def-x(r)}
	Let $U$ be a finite inner $(\alpha,A)$-uniform domain. Let $o$ be a point such that $d(o,\mathfrak X \setminus U)=\max\{d(x,\mathfrak X \setminus U):x\in U\}=R$.  Let $\gamma_{xy}$ be a collection of inner $(\alpha,A)$-uniform paths indexed by $x,y\in U$. For any $x\in U$ and $r>0$,
	let $x_r$ be defined by
	$$x_r= \left\{\begin{array}{cl}x_{\lfloor r\rfloor}  & \mbox{ if } \gamma_{xo}=(x=x_0,x_1,\dots, x_k=o) \mbox{ with } k\ge r ,\\
	o& \mbox{ if } \gamma_{xo}=(x=x_0,x_1,\dots, x_k=o) \mbox{ with } k< r.\end{array}\right.
	$$
\end{defin}

The following Carleson-type theorem, regarding the eigenfunction $\phi_0$, is the key to obtaining refined results for the convergence of the intrinsic Doob-transform chain on a finite inner-uniform domain. The context is as follows.
In addition to the geometric structure $(\mathfrak X,\mathfrak E)$, we assume we are given a measure $\pi$ and an edge weight $\mu$ such  that  
$(\mathfrak X,\mathfrak E, \pi,\mu)$ satisfies Assumption A1 with $\theta=2$, i.e., we assume that the measure $\pi$ is $D$-doubling, $\mu$ is adapted, $\pi$ domaintes $\mu$, and the pair  $(\pi,\mu)$ is elliptic and satisfies the $2$-Poincar\'e inequality on balls with constant $P$. 
\begin{theo}\label{theo-Carleson}
	Assume $(\mathfrak X, \mathfrak E, \pi, \mu)$ satisfies Assumption {\em A1} with $\theta=2$ and fix $\alpha,A$. There exists a constant $C_0$ depending only on $\alpha,A,D,P_e,P$ such that, for any finite inner $(\alpha,A)$-uniform domain $U$, the positive eigenfunction $\phi_0$ for the kernel $K_U$ in $U$  is $(1/8,C_0)$-regular and satisfies
	$$ \forall\,r>0,\;x\in U,\; z\in B_U(x, r/2),\;\;\phi_0(z)\le C_0\phi_0(x_r).$$
\end{theo}
\begin{cor} \label{cor-Carleson}
	Under the assumptions of {\em Theorem \ref{theo-Carleson}}, there are constants $D_0,D_1$ depending only on $\alpha,A,D,P_e,P$ such that 
	$$\forall\,x\in U,\;r>0,\;\;\pi_{\phi_0}(B_U(x,2r))\le D_0 \pi_{\phi_0}(B_U(x,r)).$$
	Moreover, for all $r\in [0,R]$,
	$$D_1^{-1}\phi_0(x_r)^2\pi_U(B(x_r,\alpha r) )\le \pi_{\phi_0}(B_U(x,r)) \le
	D_1\phi_0(x_r)^2\pi_U(B(x_r,\alpha r)).$$ \end{cor}

The following corollary gives a rate of convergence of the Doob transform chain to its stationary distribution in $L^{\infty}$.
\begin{cor}
\label{cor:KU-rate-of-conv}
Under the assumptions of {\em Theorem \ref{theo-Carleson}}, there are constants $C,c$ depending only on $\alpha,A,D,P_e,P$ such that, assuming that the lowest eigenvalue $\beta_-$ of the reversible Markov chain $(K_{\phi_0},\pi_{\phi_0}) $
	satisfies $1+\beta_-\ge cR^{-2}$, we have
	$$\max_{x,y\in u}\left|\frac{K^t_{\phi_0}(x,y)}{\pi_{\phi_0}(y)} -1\right| \le C \exp\left(-c\frac{t}{R^{2}}\right),$$
	for all $t\ge R^{2}$. In terms of the kernel $K_U$, this reads 
	$$\left|K^t_U(x,y) - \beta_0^t\phi_0(x)\phi_0(y)\pi_U(y)\right|\le C\beta_0^t\phi_0(x)\phi_0(y)\pi_U(y)e^{-ct/R^{2}},$$
	for all $x,y\in U$ and $t\ge R^{2}$. 
\end{cor}
\begin{proof} This follows from Theorem \ref{th-Doob2} because  the measure $\phi_0^2\pi_U$ is doubling by Theorem \ref{theo-Carleson} (and $U$ is a John domain by Lemma~\ref{lem-IUJ}).
\end{proof}

\begin{cor} \label{cor-Carleson2}
Under the assumptions of {\em Theorem \ref{theo-Carleson}}, there are constants $c,C$ depending only on $\alpha,A,D,P_e,P$ such that the second largest eigenvalue $\beta $ of the reversible Markov chain $(K_{\phi_0},\phi_{\phi_0})$ satisfies
	$$ c R^{-2}\le 1-\beta \le CR^{-2}.$$
	If $\beta_{U,1}< \beta_{U,0}=\beta_0$ denotes the second largest eigenvalue of 
	the kernel $K_U$ acting on on $L^2(U,\pi_U)$ then $\beta=\beta_{U,1}/\beta_0$
	and 
	$$ cR^{-2}\beta_0\le  \beta_0 -\beta_{U,1}\le C \beta_0  R^{-2}$$
	or equivalently
	$$   \beta_0(1-CR^{-2})\le \beta_{U,1}\le \beta_0(1-cR^{-2}).$$
	In particular, for all $t$,
	$$\max_{x\in U}\sum_{y\in U}| K_{\phi_0}^t(x,y)-\pi_{\pi_0}(y)|\ge  ce^{-Ct/R^2}.$$
\end{cor}

The following theorem is closely related to~\ref{theo-Carleson} and is to used to obtain explicit control on the function $\phi_0$. In Section~\ref{sec:examples} we demonstrate the power of this theorem in several examples.

\begin{theo}  \label{theo-comph}
	Assume {\em A1} with $\theta=2$ and fix $\alpha,A$. There exists a constant $C_1$ depending only on $\alpha,A,D,P_e,P$ such that, for any finite inner $(\alpha,A)$-uniform domain $U$, any point $x\in U$ and $r>0$ such that $$B_U(x,r)=\{y\in U: d_U(x,y)\le r\} \neq U$$ and any function $h$ defined in $U$
	and satisfying $K_Uh =h$ in $B_U(x,r)$, we have 
	$$\forall\,y,z\in B_U(x,r/2),\;\;\;  \frac{\phi_0(y)}{\phi_0(z)}\le C_1 \frac{h(y)}{h(z)}.$$
\end{theo}

\subsection{Proofs of Theorems \ref{theo-Carleson} and \ref{theo-comph}: the cable space with loops}
\label{sec-cable}
The statement in Theorem \ref{theo-Carleson} is a version of a fundamental inequality
known as a Carleson estimate \cite{Carleson} and was first derived in the study of analysis in Lipschitz domains \cite{Kemp} and \cite{Ancona,Dahl,Wu}.  For a modern perspective, sharp results, and references to the vast literature on the subject in the context of analysis on bounded domains, see \cite{BBB,Aik1,Aik2,Aik3,Aik4}. The generality and flexibility of the arguments developed by H. Aikawa in these papers and other works, based on the notion of 
``capacity width,'' is used in a fundamental way in \cite{Gyrya} and in \cite{LierlLSCOsaka,LierlLSCJFA,Lierl1,Lierl2} to extend  the result in the setting of (nice) Dirichlet spaces.

Given $(\mathfrak X,\mathfrak E, \mu, \pi) $ one can build an associated continuous space $\mathbf X$, known as the cable space for $(\mathfrak X,\mathfrak E,\pi,\mu)$. In many cases, it is more difficult to prove theorems in discrete domains than in continuous domains --- the cable space provides an important bridge, but allowing us to transfer known theorems from the continuous space $\mathbf X$ to its associated discrete space $\mathfrak X$. 

Topologically, the space $\mathbf X$ is a connected one-dimensional complex, that is, a union of copies of the interval $[0,1]$ with some identifications of end points.  The process of building the cable space from the discrete space $(\mathfrak X, \mathfrak E, \mu, \pi)$ is straightforward: the zero-dimensional points in the complex are given by the vertices $\mathfrak X$; two points $x$ and $y$ are then connected by a unit length edge $[0,1]$ if $\{x,y\} \in \mathfrak E$, with $0$ identified with $x$ and $1$ identified with $y$. For early references to the cable space, see the introduction to~\cite{Cattaneo}. This resource is particularly relevant because it discusses the spectrum of the discrete Laplacian.

But we need to allow for the addition of self-loops, copies of $[0,1]$ with $0$ and $1$ identified to each other and to some vertex $x \in \mathfrak X$. (Recall that $\mathfrak E$ has no self-loops.)  We will use the notation $(0,1)_{xx}$ for the self-loop at $x$ minus the point $x$ itself.  Let $\mathfrak L$ be the subset of those $x\in \mathfrak X$ where $\sum_y \mu_{xy} <\pi (x)$. Form a loop at each $x \in \mathfrak L$ and set the weight $\mu_{xx}$ on the loop to be equal to its ``deficiency,'' 
\begin{equation}
    \label{eq:loop-deficiency-weight}
\mu_{xx}=\pi(x)-\sum_y\mu_{xy}, \;\;x\in \mathfrak L.
\end{equation}
In what follows we will use the notation $xy$ as an index running over $\{x,y\}\in \mathfrak E$ when $x\neq y$ and $x\in \mathfrak L$ when $x=y$.

We need to use a simple (but rather interesting) variation on this construction. We introduce a {\em loop-parameter}, call it $\ell$. For any fixed $\ell\in [0,1]$, we construct the cable space $\mathbf X_\ell$ as described above but the self-loops have length $\ell$ instead of $1$ above. The other edges (non-self-loops) still have length 1. 

More precisely, the space $\mathbf X_\ell$ is obtained by joining any two  points $x,y$ in $\mathfrak X$ with $\{x,y\}\in \mathfrak E$ by a continuous edge $e_{xy}=(0,1)_{xy}$ isometric to the interval $(0,1)$ and adding a self-loop $e_{xx}=(0,\ell)_{xx}$ at each $x\in \mathfrak L$. Strictly speaking,  
$$\mathbf  X=\mathfrak X\cup \left(\bigcup_{\{x,y\}\in \mathfrak E}(0,1)_{xy}\right)\cup \left(\bigcup_{x\in \mathfrak L} (0,1)_{xx} \right) $$ 
with $e_{xy}$ being a copy of $(0,1)$ when $x\neq y$ and a copy of $(0,\ell)$ when $x=y$. See Figure \ref{XX}. The topology of this space is generated by the open subintervals of these many copies of $(0,1)$ and $(0,\ell)$, together with the star-shaped open neighborhoods of the vertices in
$\mathfrak X$.

\begin{figure}[h]
	\begin{center}
		\begin{picture}(250,170)(-20,0)
		
		\put(0,20){
			\multiput( 50,125)(40,0){4}{\circle*{3}}
			\multiput( 53,125)(40,0){3}{\multiput(0,0)(5,0){7}{\line(1,0){3}}}
			{\color{red}\put(45,127){\makebox{$2$}} \put(88,127){\makebox{$2$}} \put(128,127){\makebox{$3$}} \put(170,127){\makebox{$2$}} }
			{\color{black}\multiput(65,127)(40,0){3}{\makebox{$1$}}  } }
		
		\put(-1,15){
			\multiput( 50,75)(40,0){4}{\circle*{2}}
			\thicklines{\multiput( 50,75)(40,0){3}{\line(1,0){40}} }
			\put(50,66){\circle{16}}
			\put(130,66){\circle{16}} \put(170,66){\circle{16}}
			
			{\color{red}\put(45,77){\makebox{$2$}} \put(88,77){\makebox{$2$}} \put(128,77){\makebox{$3$}} \put(170,77){\makebox{$2$}} }
			{\color{black}\multiput(65,77)(40,0){3}{\makebox{$1$}}  }
			{\color{black} \put(40,59){\makebox{$1$}}
				\put(121.5,59){\makebox{$1$}} \put(163,59){\makebox{$1$}}}
		}

		\put(-1,-30){
			\multiput( 50,75)(40,0){4}{\circle*{2}}
			\thicklines{\multiput( 50,75)(40,0){3}{\line(1,0){40}} }
			\put(50,71){\circle{8}}
			\put(130,71){\circle{8}} \put(170,70){\circle{8}}
			
			{\color{red}\put(45,77){\makebox{$2$}} \put(88,77){\makebox{$2$}} \put(128,77){\makebox{$3$}} \put(170,77){\makebox{$2$}} }
			{\color{black}\multiput(65,77)(40,0){3}{\makebox{$1$}}  }
			{\color{black} \put(40,57){\makebox{$1$}}
				\put(121.5,57){\makebox{$1$}} \put(163,57){\makebox{$1$}}}
		}

		\end{picture}
		\caption{A simple example of $(\mathfrak X,\mathfrak E,\pi,\mu)$ and the associated cable spaces $\mathbf X_\ell$, where the edge weights $\mu$ are indicated in black and the vertex weights $\pi$ are indicated in red. The black weights on the loops indicate the ``deficiencies'' in the edge weights, as described in~\eqref{eq:loop-deficiency-weight}.}\label{XX}
		\end{center}\end{figure}
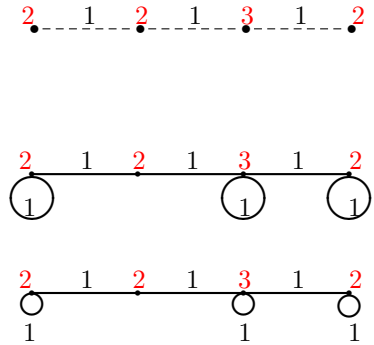

The cable Dirichlet space associated with the data   $(\mathfrak X,\mathfrak E, \mu, \pi,\ell) $  is obtained by equipping $\mathbf X_\ell$ with 
its natural distance function  $\mathbf d_\ell :\mathbf X_{\ell}\times\mathbf X_{\ell}\rightarrow [0,\infty)$, the length of the shortest path between two points. The space $\mathbf X_\ell $ is also equipped
with a measure   $\boldsymbol{\pi}$ equal to $\mu_{xy} dt$ on each  interval $e_{xy}$ (including the intervals  $e_{xx}$), and  with the Dirichlet form obtained by closing  the form 
$$\mathcal E_{\mathbf X_\ell}(f,f)= \sum_{{xy}} \mu_{xy}\int_{e_{xy}} |f_{e_{xy}}'(t)|^2dt, \;\; f\in \mathcal D_0(\mathbf X_\ell),$$
where $\mathcal D_0(\mathbf X_\ell)$ is the space of all compactly  supported continuous functions on $\mathbf X_\ell$ which have a bounded continuous derivative  $f'_{e_{xy}}$ on each open edge $e_{xy}$ and $e_{xx}$. (Note that  the values of these various edge-derivatives at a vertex do not have to match in any sense.)  The domain of 
$\mathcal E_{\mathbf X_\ell}$, $\mathcal D(\mathcal E_{\mathbf X_{\ell}})$, is the closure of $\mathcal D_0(\mathbf X_\ell)$ under the norm $$\|f\|_{\mathcal E_{\mathbf X_{\ell}}} = \left(\int_{\mathbf X_\ell}|f|^2d\boldsymbol{\pi}+\mathcal E_{\mathbf X_\ell}(f,f)\right)^{1/2}.$$ 

The cable Dirichlet space $(\mathbf X_\ell, \boldsymbol{\pi}, {\mathcal E}_{\mathbf X_\ell})$ is  a regular strictly local Dirichlet space (see, e.g., \cite{FOT,Gyrya}) and its  intrinsic distance
is the shortest-path distance $\mathbf d_\ell$ described briefly above. This Dirichlet space is actually quite elementary in the sense that it is possible to describe concretely the domain of the associated Laplacian, the generator of the associated Markov semigroup of operators acting on $L^2(\mathbf X_\ell,\boldsymbol{\pi})$.  First, we recall that this Laplacian is the self-adjoint operator $\Delta_\ell$ with domain $\mathcal D(\Delta_\ell)$ in $L^2(\mathbf X_\ell,\boldsymbol{\pi})$ defined by
$$\mathcal D(\Delta_\ell)=\{u\in \mathcal D(\mathcal E_{\mathbf X_\ell}): \exists C \mbox{ such that}, \forall f\in \mathcal D_0({\mathbf X_\ell}), \mathcal E_{\mathbf X_\ell}(u,f)\le C\|u\|_2\}.$$
For  any function $u\in \mathcal D(\Delta_\ell)$ there exists a unique function $v\in L^2({\mathbf X_\ell},\boldsymbol{\pi})$ such that $\mathcal E_{\mathbf X_\ell}(u,f)=\int_{\mathbf X_\ell} vf d\boldsymbol{\pi}$ (from the Riesz representation theorem) and we set
$$\Delta u = -v.$$
This implies that  $$\mathcal E_{\mathbf X_\ell}(u,f)= -\int  f \Delta _\ell u d\boldsymbol{\pi}$$
for all $u\in \mathcal D(\Delta_\ell)$ and all $f\in \mathcal D_0(\mathbf X_\ell)$ (equivalently, all $f\in \mathcal D(\mathcal E_{\mathbf X_\ell})$).

From the above abstract definition, we can now derive a concrete description of $\mathcal D(\Delta_\ell)$.  We start with a concrete description of  $\mathcal D(\mathcal E_{\mathbf X_\ell})$. A function $f$ is in $\mathcal D(\mathcal E_{\mathbf X_\ell})$ if it is continuous on $\mathbf X_\ell$, belongs to $L^2(\mathbf X_\ell,\boldsymbol{\pi})$ and the restriction $f_{e_{xy}}$ of $f$ to  any open edge $(0,1)_{xy}$, has a distributional derivative which can be represented by a square-integrable function $f'_{e_{xy}}$
satisfying   $$\sum_{xy} \mu_{xy}\int_{e_{xy}} |f'_{e_{xy}}|^2 dt<\infty.$$
The key observation is that, because of the one-dimensional nature of $\mathbf X$,  on any edge $e_{xy}$ (or subinterval of $e_{xy}$)
on which $f'$ is defined in the sense of distributions and represented by a square integrable function, we have
$$ |f(s_2)-f(s_1)|\le \sqrt{|s_2-s_1|}\left(\int_{s_1}^{s_2}|f'(s)|^2 ds\right)^{1/2}.$$

We now give a (well-known) concrete description of $\mathcal D(\Delta_\ell)$. A  function $u\in L^2(\mathbf X_\ell,\boldsymbol{\pi})$
is in $\mathcal D(\Delta_\ell)$ if and only if  
\begin{enumerate} 
	\item The function $u\in L^2(\mathbf X_\ell,\boldsymbol{\pi})$ admits a continuous version, which, abusing notation, we still call  $u$.
	\item  On each open edge  $e_{xy}$, the restriction $u_{e_{xy}}$ of  $u$ to $e_{xy}$ has a continuous first derivative $u'_{e_{xy}}$ with limits at the two end-points  and such that
	$$\sum_{xy}\mu_{xy}\int_{e_{xy}}|u'_{e_{xy}}|^2dt<\infty.$$ Furthermore
	$u_{e_{xy}}$ has  a 
	second derivative in the sense of distributions which can be represented by  a square-integrable function $u''_{e_{xy}}$ and $$\sum_{xy}\mu_{xy}\int_{e_{xy}}|u''_{e_{xy}}|^2dt<\infty.$$
	\item  At any vertex $x\in \mathfrak X$, Kirchhoff's law
	$$\sum_{y:\{x,y\}\in \mathcal E}\mu_{xy}\vec{u}_{e_{xy}}(x) +\sum_{x\in \mathfrak L}\mu_{xx}(\vec{u}_{e_{xx}}(0)-\vec{u}_{e_{xx}}(\ell))=0$$
	holds.  Here,  for $\{x,y\}\in \mathfrak E$,
	$\vec{u}_{e_{xy}}(x)$ is the  (one-sided) derivative of $u$ at $x$ computed along $e_{xy}$ oriented from $x$ to $y$ and, for $x\in \mathfrak L$, $\vec{u}_{e_{xx}}(0)$ and $\vec{u}_{e_{xx}}(\ell)$ are the (one-sided) derivatives of $u_{e_{xx}}$ on $(0,\ell)_{xx}$ at $0$ and at $\ell$. \end{enumerate}

We say that a function $u$ defined on a subset $\Omega$ is locally in $\mathcal D(\Delta_\ell)$ if it satisfies the above properties over $\Omega$ except for the global square integrable conditions on $u,u'$ and $u''$.  For such a function, 
$\Delta_\ell  u$ is defined as the locally square integrable function 
$\Delta_\ell  u =u''$ where $u''=u''_{e_{xy}}$ on  $e_{xy}\cap \Omega$.

\begin{rem}
    The stochastic process associated with the Dirichlet form $\mathcal{E}_{\mathbf{X}_{\ell}}$ can be explicitly constructed using Brownian motion. More specifically, starting at a vertex in the cable space, one performs Brownian excursions along adjacent edges until reaching another vertex. For a detailed description see~\cite{RevuzYor, Folz}. See~\cite{quantum-graphcs} for a description of the related quantum graphs.
\end{rem}

\begin{defin}\label{def-UL}
	To any finite domain $U$
	in $(\mathfrak X,\mathfrak E)$ we  associate the domain $\mathbf U=\mathbf U_\ell$ in ${\mathbf X_\ell}$
	formed by all the vertices $x$ in $U$  and all  the open edge $e_{xy}$ with at least one end point in $U$, including the loops $e_{xx}$ with $x\in U$.
\end{defin}

See Figure \ref{XXU} for an example of Defintion~\ref{def-UL}. As another example, consider the trivial finite domain $U=\{x\}$. To it, we associate the domain $\mathbf U$ formed by the vertex $x$ and all the open edges containing $x$, i.e., an open star around $x$, perhaps with a self-loop of length $\ell$,
whose branches are in one to one correspondence with the $y\in \mathfrak X$ such that $\{x,y\}\in \mathfrak E$. 

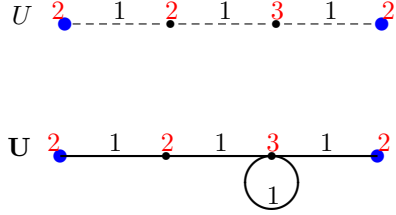
\begin{figure}[t]
	\begin{center}
		\begin{picture}(200,100)(-20,0)
		
		\put(30,75){\makebox{$U$}}
		{\color{blue} \multiput( 50,75)(120,0){2}{\circle*{5}}}
		\multiput( 90,75)(40,0){2}{\circle*{3}}
		\multiput( 53,75)(40,0){3}{\multiput(0,0)(5,0){7}{\line(1,0){3}}}
		
		{\color{red}\put(45,77){\makebox{$2$}} \put(88,77){\makebox{$2$}} \put(128,77){\makebox{$3$}} \put(170,77){\makebox{$2$}} }
		{\color{black}\multiput(65,77)(40,0){3}{\makebox{$1$}}  }
		
		\put(-5,0){  \put(30,25){\makebox{$\mathbf U$}}
			{\color{blue} \multiput( 50,25)(120,0){2}{\circle*{5}}}
			\multiput( 90,25)(40,0){2}{\circle*{3}}\thicklines{\multiput( 50,25)(40,0){3}{\line(1,0){40}} }
			\put(130,15){\circle{20}} 
			
			{\color{red}\put(45,27){\makebox{$2$}} \put(88,27){\makebox{$2$}} \put(128,27){\makebox{$3$}} \put(170,27){\makebox{$2$}} }
			{\color{black}\multiput(65,27)(40,0){3}{\makebox{$1$}}  }
			{\color{black}
				\put(121.5,7){\makebox{$1$}}}
		}\end{picture}
		\caption{ $U$ and $\mathbf U$  in black with their boundaries in blue.}\label{XXU} 
		\end{center}\end{figure}

With this definition, the discrete finite domain $U$ is inner-uniform if and only if the
domain $\mathbf U$ is inner-uniform in the metric space $(\mathbf X_\ell,\mathbf{d}_\ell)$. Following \cite[Definition 3.2]{Gyrya} we say that a continuous domain $\mathbf U$ is inner-uniform in the metric space $(\mathbf X_{\ell},\mathbf d_{\ell})$ if there exists constants $A^c$ and $\alpha^c$ such that, for each $\xi,\zeta \in \mathbf U$, there exists a continuous curve $\gamma_{\xi\zeta}:[0,\tau] \rightarrow \mathbf U$ (called an \emph{inner-uniform path}) contained in $\mathbf U$ with $|\gamma_{\xi\zeta}| = \tau$ such that (1) $\gamma_{\xi\zeta}(0) = \xi$ and $\gamma_{\xi\zeta}(\tau) = \zeta$, (2) $|\gamma_{\xi\zeta}| \leq A^c\mathbf{d}_{\mathbf U}(\xi,\zeta)$ and (3) for any $t\in[0,\tau]$,
$$\mathbf{d}_{\ell}(\gamma_{\xi\zeta}(t),\mathbf X_{\ell} \setminus \mathbf{U}) \geq \alpha^c \min\{t,\tau_U - t\}$$ where $\mathbf{d}_{\mathbf{U}}$ is the distance in $\mathbf{U}$.

The important constants $A^d,\alpha^d$
and $A^c,\alpha^c$  ($d$ for discrete, $c$ for continuous)
capturing the key properties of an inner-uniform domain in both cases are within factors
of  $8$ from each others.  (Very large self-loops would be problematic, but we restrict to $\ell\in [0,1]$.)  
In fact, for any pair of points $\xi,\zeta$ in $\mathbf U$ we can define
an inner-uniform path $\gamma_{\xi\zeta}$ from $\xi$ to $\zeta$ as follows. If the two points satisfy $\mathbf d_{\mathbf U}(\xi,\zeta)=\tau\le 1$, i.e., they are either on the same edge or on two adjacent edges, then we set $\gamma_{\xi\zeta}$ to be the obvious path  from $\xi$ to $\zeta$, parametrized by arc-length (one can easily check that this path satisfies  $\mathbf{d}_{\ell}(\gamma_{\xi\zeta}(t),\mathbf{X}_{\ell} \setminus \mathbf{U})\ge  \min\{t,\tau-t\}$).  When $\mathbf d_{\mathbf U}(\xi,\zeta)> 1$, one can join them in $\mathbf U$ by first finding the closest points $x(\xi)$ and $x(\zeta)$ in $U$ (if there are multiple choices, pick one) and then use the obvious continuous extension of the discrete inner-uniform path from $x(\xi)$ to $x(\zeta)$, which is, again, parametrized by arc-length.  

Finally we extend Definition \ref{def-x(r)} from $U$ to $\mathbf U$ as follows.

\begin{defin} Let $U$ be a finite inner-uniform domain equipped with a central point $o\in U$ such that $d(o,\mathfrak X \setminus U)=\max\{d(x,\mathfrak X \setminus U):x\in U\}$.
	For any point $\xi\in \mathbf U$, let  $\gamma_{\xi o}$ be the inner-uniform continuous path defined above joining $\xi$ to $o$ in $\mathbf U$.
	For any $\xi\in \mathbf U$ and $r>0$, let $\xi_r$ be defined by
	$$\xi_r= x(\xi)_r   \mbox{ if } r\ge 1$$
	where $x(\xi)$ is the (chosen) closest point to $\xi$ in $U$ and $x(\xi)_r$ is given by Definition \ref{def-x(r)},   and
	$$\xi_r =\gamma_{\xi o}(\min\{r,\tau\}) \mbox{ if }   r\in (0,1) \mbox{ and }  \gamma_{\xi o}(\tau)=o,.$$ 
\end{defin}

\begin{rem} \label{rem-xr}
The two key properties of the point $\xi_r\in \mathbf U$ are  as follows. There are two constants $C,\epsilon$ which depends only on the
	inner-uniform constants $A,\alpha$ of $U$ such that
	\begin{enumerate}
		\item The inner-distance  $\mathbf d_{\mathbf U}(\xi,\xi_r)$ is no larger than $C r$; 
		\item  The  distance $\mathbf d_\ell (\xi_r,\mathbf \mathfrak X \setminus U)$  is at least  $\epsilon r$.
	\end{enumerate}
	In the present case, we chose the points $\xi_r$ so that, for $r\ge 1$, they actually belong to $U$ and coincide with $x(\xi)_r$ from Definition \ref{def-x(r)}. 
\end{rem}

The heat diffusion with Dirichlet boundary condition on the bounded inner-uniform domain $\mathbf U=\mathbf U_\ell$
is studied in \cite{LierlLSCJFA,LierlLSCOsaka}. The heat diffusion semigroup with Dirichlet boundary condition on the domain $\mathbf U$ is the semigroup associated with the Dirichlet form obtained by closing the (closable) form
$$ \mathcal E_{\mathbf U,D}(f,f)= \int _{\mathbf U} |f'|^2 d\boldsymbol{\pi}$$
defined on continuous functions $f$ in $\mathbf U$ that are locally in $\mathcal D(\mathcal E_{\mathbf X_\ell})$  and have compact support in $\mathbf U$ (for such function, $f'=f'_{e_{xy}}$ on $e_{xy}\cap \Omega$).  The subscript $D$ in this notation stands for {\em Dirichlet condition}.
Let
$H_t^{\mathbf U,D}= e^{t \Delta_{\mathbf U,D}}$ be the associated self-adjoint semigroup on $L^2(\mathbf U,\boldsymbol{\pi}_{\mathbf U})$
with infinitesimal generator $\Delta_{\mathbf U,D}$.  Here, $\boldsymbol{\pi}_{\mathbf U}$ is the normalized restriction of $\boldsymbol{\pi}$ to $\mathbf U$
$$ \boldsymbol{\pi}_{\mathbf U}= \boldsymbol{\pi}(\mathbf U)^{-1} \boldsymbol{\pi}|_{\mathbf U}.$$
The domain of $\Delta_{\mathbf U, D}$  is  exactly the set of 
functions $f$ that are locally in  $\mathcal D(\Delta_\ell)$ in $\mathbf U$, have limit $0$ 
at the boundary points of $\mathbf U$ and satisfy $\int _{\mathbf U}|u''|^2 d\boldsymbol{\pi}<\infty$.   Also the parameter $\ell$ does not appear explicitly in the notation we just described, but all these objects depend on the choice of $\ell$. 

Just as in the discrete setting, the key to the study of $H^{\mathbf U,D}_t$ is the Doob-transform technique which involves the positive eigenfunction $\boldsymbol{\phi}_{\ell,0}$ associated to the smallest eigenvalue $\boldsymbol{\lambda}_{\ell,0}$ of $-\Delta_{\mathbf U,D}$ in $\mathbf U$. 
This function is defined by the following equations:
\begin{enumerate}
	\item $\boldsymbol{\lambda}_{\ell,0}= \inf \left\{\int_{\mathbf U} |f'|^2 d\boldsymbol{\pi}_{\mathbf{U}}: f\in \mathcal D(\mathcal E_{\mathbf U,D}),\int_{\mathbf U}|f|^2d\boldsymbol{\pi}_{\mathbf{U}}=1\right\}$;
	\item $\boldsymbol{\phi}_{\ell,0}\in \mathcal D(\Delta_{\mathbf U,D})$ and $\Delta_{\mathbf U,D} \boldsymbol{\phi}_{\ell,0}=-\boldsymbol{\lambda}_0 \boldsymbol{\phi}_{\ell,0}$;
	\item $\int_{\mathbf U}|\boldsymbol{\phi}_{\ell,0}|^2d\boldsymbol{\pi}=1.$
\end{enumerate}

\begin{pro} \label{pro-*}
	Assume  that $(\mathfrak X,\mathfrak E,\pi,\mu)$ is such that $\mu$ is adapted  and  $\mu$ is subordinated to $\pi$.  
	Let $U$ be a finite domain in $(\mathfrak X,\mathfrak E)$.   There exists a value 
	$$\ell_0=\ell_0(\mathfrak X,\mathfrak E,\pi,\mu,U)\in [0,1]$$ 
	of the loop-parameter $\ell$ such that the following properties hold true.
	
	Let $\mathbf U$ be the bounded domain in $\mathbf X_{\ell_0}$ associated to $U$. Let $\phi_0$, $\beta_0$ be the Perron-Frobenius eigenfunction and eigenvalue of $K_U$. Let $\boldsymbol{\phi}_0$, $\boldsymbol{\lambda}_0$ be the eigenfunction and bottom eigenvalue of $\Delta_{\mathbf U,D}$  for the parameter $\ell_0$ as defined above.  There exists  a constant $\kappa>0$ such that 
	\begin{enumerate}
		\item $\beta_0= \cos(\sqrt{\boldsymbol{\lambda}_0})$,
		\item $\phi_0(x)= \kappa \boldsymbol{\phi}_0(x)$ for all vertex $x\in U$.
	\end{enumerate}
\end{pro}
\begin{proof}   First, we study the function $\boldsymbol{\phi}_{\ell,0}$ for an arbitrary $\ell \in [0,1]$. 
	On each edge $e_{xy}$ in $\mathbf U$, the function $\boldsymbol{\phi}_{\ell,0}$ 
	satisfies $$\left(\frac{\partial }{\partial s}\right)^2[\boldsymbol{\phi}_{\ell,0}]_{e_{xy}}=-\boldsymbol{\lambda}_{\ell,0}[\boldsymbol{\phi}_{\ell,0}]_{e_{xy}},$$ and this implies  
	$$[\boldsymbol{\phi}_{\ell,0}]_{e_{xy}}(s)= \frac{\boldsymbol{\phi}_{\ell,0}(y)-\cos(\sqrt{\boldsymbol{\lambda}_{\ell,0}} \ell_{xy})\boldsymbol{\phi}_{\ell,0}(x)}{\sin (\sqrt{\boldsymbol{\lambda}_{\ell,0}}\ell_{xy})} \sin (\sqrt{\boldsymbol{\lambda}_{\ell,0}} s) +\boldsymbol{\phi}_{\ell,0}(x)\cos (\sqrt{\boldsymbol{\lambda}_{\ell,0}} s)$$
	where $s\in (0,\ell_{xy})$ parametrizes $e_{xy}$ from $x$ to $y$  with 
	$$\ell_{xy} = \begin{cases} 1 & \text{ when } x\neq y \\ \ell & \text{ when } x = y. \end{cases}$$ When $x=y$,
	$$[\boldsymbol{\phi}_{\ell,0}]_{e_{xx}}(0)=[\boldsymbol{\phi}_{\ell,0}]_{e_{xx}}(\ell)=\boldsymbol{\phi}_{\ell,0}(x),$$
	and the function $[\boldsymbol{\phi}_{\ell,0}]_{e_{xx}} $  on the edge $(0,\ell)_{xx}$ satisfies 
	$$[\boldsymbol{\phi}_{\ell,0}]_{e_{xx}}(s)=[\boldsymbol{\phi}_{\ell,0}]_{e_{xx}}(\ell-s).$$
	To express Kirchhoff's law  at $x\in U$, we compute, for $x \neq y$,
	$$[\vec{\boldsymbol{\phi}}_{\ell,0}]_{e_{xy}}(0)=\frac{\sqrt{\boldsymbol{\lambda}_{\ell,0}}}{\sin (\sqrt{\boldsymbol{\lambda}_{\ell,0}})}(\boldsymbol{\phi}_{\ell,0}(y )-\cos(\sqrt{\boldsymbol{\lambda}_{\ell,0}}) \boldsymbol{\phi}_{\ell,0}(x)),$$
	and, for $x=y$,
	\begin{eqnarray*}
		[\vec{\boldsymbol{\phi}}_{\ell,0}]_{e_{xx}}(0)-[\vec{\boldsymbol{\phi}}_{\ell,0}]_{e_{xx}}(1)&=&
		2[\vec{\boldsymbol{\phi}}_{\ell,0}]_{e_{xx}}(0)\\ &=&2\frac{\sqrt{\boldsymbol{\lambda}_{\ell,0}}}{\sin (\sqrt{\boldsymbol{\lambda}_{\ell,0}}\ell)}(1-\cos(\sqrt{\boldsymbol{\lambda}_{\ell,0}} \ell))
		\boldsymbol{\phi}_{\ell,0}(x).\end{eqnarray*}
	It follows that Kirchhoff's law gives
	\begin{eqnarray*} \lefteqn{
			\sum_{y:\{x,y\}\in \mathfrak E}\mu_{xy}(\boldsymbol{\phi}_{\ell,0}(y )-\cos(\sqrt{\boldsymbol{\lambda}_{\ell,0}}) \boldsymbol{\phi}_{\ell,0}(x)) } \hspace{1in}&&\\
		&&+2\mu_{xx}\frac{\sin( \sqrt{\boldsymbol{\lambda}_{\ell,0}})}{\sin( \sqrt{\boldsymbol{\lambda}_{\ell,0}}\ell)}(1-\cos(\sqrt{\boldsymbol{\lambda}_0}\ell))\boldsymbol{\phi}_{l,0}(x)=0.\end{eqnarray*}
	
	Recall that $K_U(x,y)= \mu_{xy}/\pi(x)$ for $x,y\in U$ with $\{x,y\}\in \mathfrak E$ and $K_U(x,x)=\mu_{xx}/\pi(x)$. It follows that, for $x\in U$,
	$$K_U\boldsymbol{\phi}_{\ell,0} (x)= \frac{1}{\pi(x)}\sum_{y} \mu_{xy} \boldsymbol{\phi}_{\ell,0}(y)=  \frac{1}{\pi(x)}\left(
	\sum_{y:\{x,y\}\in \mathfrak E} \mu_{xy} \boldsymbol{\phi}_{\ell,0}(y) +\mu_{xx} \boldsymbol{\phi}_{\ell,0}(x) \right),$$
	and Kirchhoff laws for $\boldsymbol{\phi}_{\ell,0}$  yields
	\begin{eqnarray*}\lefteqn{
			\frac{K_U\boldsymbol{\phi}_{\ell,0}(x) }{ \boldsymbol{\phi}_{\ell,0}(x)}=
			K_U(x,x) +\left(1-K_U(x,x)\right)\cos (\sqrt{\boldsymbol{\lambda}_{\ell,0}})} &&\\&&\hspace{1.3in} -2K_U(x,x) \frac{\sin( \sqrt{\boldsymbol{\lambda}_{\ell,0}})}{\sin( \sqrt{\boldsymbol{\lambda}_{\ell,0}}\ell)}(1-\cos(\sqrt{\boldsymbol{\lambda}_0}\ell))\\
		&=&  \cos (\sqrt{\boldsymbol{\lambda}_{\ell,0}})\\
		&&+K_U(x,x)(1-\cos(\sqrt{\boldsymbol{\lambda}_{\ell,0}}))\left(  1-2   \frac{\sin( \sqrt{\boldsymbol{\lambda}_{\ell,0}})}{\sin( \sqrt{\boldsymbol{\lambda}_{\ell,0}}\ell)}\frac{(1-\cos(\sqrt{\boldsymbol{\lambda}_{\ell,0}}\ell))}{  (1-\cos(\sqrt{\boldsymbol{\lambda}_{\ell,0}}))} \right)
	\end{eqnarray*}
	
	Given the uniqueness of the Perron-Frobenius eigenvalue and the fact that the associated positive eigenfunction is unique up to a multiplicative constant,
	the proposition follows from the previous computation if there exists 
	$\ell_0\in [0,1] $ at which the function
	$$ F(\ell)= 1-2   \frac{\sin( \sqrt{\boldsymbol{\lambda}_{\ell,0}})}{\sin( \sqrt{\boldsymbol{\lambda}_{\ell,0}}\ell)}\frac{(1-\cos(\sqrt{\boldsymbol{\lambda}_{\ell,0}}\ell))}{  (1-\cos(\sqrt{\boldsymbol{\lambda}_{\ell,0}}))} $$
	vanishes.    But, by an easy inspection,  $F(0)= 1$  and $F(1)=-1$.  If we can prove that the function
	$$\ell\mapsto \boldsymbol{\lambda}_{\ell,0}$$
	is continuous, then $F$ must vanish somewhere between $l=0$ and $l=1$, so we are done. 
	
	Fix $\ell_1,\ell_2$. Any function $f$ on $\mathbf X_{\ell_1}$ is turned into a function
	$\tilde{f}$ on $\mathbf X_{\ell_2}$ by setting 
	$$\tilde{f}_{e_{xy}}(s)=\left\{\begin{array}{cl} f_{e_{xy}}(s) &\mbox{ if } x\neq y\\
	f_{e_{xx}}(\ell_2s/\ell_1) &\mbox{ if } y=x.\end{array}\right.$$
	Further,
	$$\int_{\mathbf X_{\ell_2}}|\tilde{f}|^2d\boldsymbol{\pi}=\int_{\mathbf X_{\ell_1}}|f|^2d\boldsymbol{\pi} +((\ell_1/\ell_2)-1)\sum_{x\in \mathfrak L}\mu_{xx}\int_{e_{xx}}|f_{e_{xx}}|^2dt$$
	and
	$$\mathcal E_{\mathbf X_{\ell_2}}(\tilde{f},\tilde{f})=\mathcal E_{\mathbf X_{\ell_1}}(f,f)+((\ell_2/\ell_1)-1)\sum_{x\in \mathfrak L}\mu_{xx}\int_{e_{xx}}|f'_{e_{xx}}|^2dt.$$

	Applying this to the function $\boldsymbol{\phi}_{\ell_1,0}$, normalized so that
	$\int_{\mathbf X_{\ell_1}}|\boldsymbol{\phi}_{\ell_1,0}|^2d\boldsymbol{\pi} =1$,
	we find that 
	$$\boldsymbol{\lambda}_{\ell_2,0}\le  \frac{\max\{1,\ell_1/\ell_2\}}{\min\{1,\ell_2/\ell_1\}} 
	\boldsymbol{\lambda}_{\ell_1,0}.$$
	Exchanging the role of $\ell_1,\ell_2$ yields thet complementary inequality
	$$\boldsymbol{\lambda}_{\ell_2,0}\ge  \frac{\min\{1,\ell_1/\ell_2\}}{\max\{1,\ell_2/\ell_1\}} 
	\boldsymbol{\lambda}_{\ell_1,0}.$$
	This proves the continuity of $\ell\mapsto \boldsymbol{\lambda}_{\ell,0}$ as desired.
\end{proof}

\begin{rem} When the quantity $K_U(x,x)$ is constant, say, $K_U(x,x)=\theta$ for all $x\in U$, then every function $\boldsymbol{\phi}_{\ell,0}$ for $\ell\in [0,1]$ satisfies
	$\phi_0(x)=\kappa_\ell \boldsymbol{\phi}_{\ell,0}(x)$ at vertices $x \in U$, and 
	we have  
	$$\beta_0=1- (1-\cos(\sqrt{\boldsymbol{\lambda}_{\ell,0}}))\left(1-\theta\left(1-2   \frac{\sin( \sqrt{\boldsymbol{\lambda}_{\ell,0}})}{\sin( \sqrt{\boldsymbol{\lambda}_{\ell,0}}\ell)}\frac{(1-\cos(\sqrt{\boldsymbol{\lambda}_{\ell,0}}\ell))}{  (1-\cos(\sqrt{\boldsymbol{\lambda}_{\ell,0}}))}  \right)\right).$$
	The function of $l$ on the right-hand side is equal to the constant $\beta$.
\end{rem}

\begin{theo}[{Special case of \cite[Proposition 5.10]{LierlLSCJFA}}]  \label{th-JL}
	Assume {\em A1} with $\theta=2$ and fix $\alpha,A$. There exists a constant $C_0$ depending only on $\alpha,A,D,P_e,P$ such that, for any finite inner $(\alpha,A)$-uniform domain $U$ and loop parameter $\ell\in [0,1]$, the positive eigenfunction $\boldsymbol{\phi}_{\ell,0}$ for the $\Delta_{\mathbf U_\ell,D}$ in $\mathbf U_\ell$  is $(1/8,C_0)$-regular and satisfies
	$$ \forall\,r>0,\;\xi\in \mathbf U_\ell,\; z\in B_{\mathbf U_\ell}(\xi, r/2),\;\;\boldsymbol{\phi}_{\ell,0}(z)\le C_0\boldsymbol{\phi}_{\ell,0}(\xi_r).$$
\end{theo}
\begin{proof} The domain $\mathbf U=\mathbf U_\ell$ in $(\mathbf X_\ell,\mathbf d_\ell)$ is inner-uniform and the Dirichlet space $(\mathbf X_\ell,\boldsymbol{\pi},\mathcal E_{\mathbf X_\ell})$ is a Harnack space in the sense of \cite{Gyrya} and \cite{LierlLSCJFA}. The most basic case of \cite[Proposition 5.10]{LierlLSCJFA} provides the desired result. Technically speaking, the definition of the map $(x,r)\mapsto \xi_r$ here and in \cite{LierlLSCJFA} are slightly different but these differences are inconsequential. \end{proof}

\begin{proof}[Proof of Theorem \ref{theo-Carleson}]   Together, Theorem \ref{th-JL} and Proposition \ref{pro-*} obviously yield Theorem \ref{theo-Carleson}.
\end{proof}

\begin{proof}[Proof of Theorem~\ref{theo-comph}] We use the  same method as in the proof of Theorem \ref{theo-Carleson} and extract this result from the similar result for the cable process
	with the proper choice $\ell_0$ of loop length.  Local harmonic functions for the cable process (with Dirichlet boundary condition at the boundary of $U$) 
	are always in a one-to-one  straightforward  correspondance with  local  harmonic functions for $K_U$, independently of the choice of the loop parameter $\ell$. Therefore, the stated result follows from \cite[Theorem 5.5]{LierlLSCJFA}.
	
\end{proof}

\subsection{Point-wise kernel bounds}
 \label{sec-HPWb}
In this section, we describe how to obtain the following detailed point-wise estimates on the iterated kernels $K^t_U$ and $K^t_{\phi_0}$ when $U$ is inner-uniform. Recall that $V(x,r)=\pi(B(x,r))$ and $x_{\sqrt{t}}$ is a point such that ${d(x_{\sqrt{t}},\mathfrak X \setminus U)\ge \alpha( 1+\sqrt{t})}$  
if $\sqrt{t} \le R$ and $x_{\sqrt{t}}=o$ otherwise.
\begin{theo}\label{theo-DH}
	Assume {\em A1} with $\theta=2$ and fix $\alpha,A$.  In addition, assume that the pair $(\pi,\mu)$ is such that $\sum_y\mu_{xy}\le (1-\epsilon)\pi$ with $\epsilon>0$ (this means that $\min_{x\in \mathfrak X}\{K_{\mu}(x,x)\}\ge \epsilon $). 
	There exist constants $c_1,c_2,C_1,C_2\in (0,\infty)$which depend only on $\alpha,A,D,P_e,P$ and are such that, for any finite inner $(\alpha,A)$-uniform domain $U$,  integer $t$ and $x,y\in U$ such that $d_U(x,y)\le t$,
	\begin{eqnarray*} \lefteqn{
			\frac{C_1\exp(-c_1 d_U(x,y)^2/t) } {\sqrt{V(x,\sqrt{t})V(y,\sqrt{t})} \phi_0(x_{\sqrt{t}})\phi_0(y_{\sqrt{t}})} } &&  \\
		&\leq& \frac{K_{\phi_0}^t(x,y)}{\phi_0(y)^2\pi(y)}\\
		&\leq& \frac{C_2 \exp(-c_2 d_U(x,y)^2/t)} {\sqrt{V(x,\sqrt{t})V(y,\sqrt{t})} \phi_0(x_{\sqrt{t}})\phi_0(y_{\sqrt{t}})}  .\end{eqnarray*}
\end{theo}
\begin{rem}   When $t$ is larger than $R^2$ then $x_{\sqrt{t}}=o$ and the two-sided estimate above states that 
	$K_{\phi_0}(x,y)$ is roughly of order $\pi_{\phi_0}(y)^2\pi_U(y)$ because $\phi_0(o)^2\simeq 
	\sum_U \phi_0^2\pi_U=1$. The convergence result stated earlier give better estimates in this case.  When $t\le R^2$, the statement provides a useful estimate of the iterated kernel before the equilibrium is reached.
\end{rem}
The following corollary simply translates Theorem \ref{theo-DH} in terms of the iterated kernel $K_U^t$.
\begin{cor}
\label{cor:DH}
	Assume {\em A1} with $\theta=2$ and fix $\alpha,A$.  In addition, assume that the pair $(\pi,\mu)$ is such that $\sum_y\mu_{xy}\le (1-\epsilon)\pi$ with $\epsilon>0$ (which implies that $\min_{x\in \mathfrak X}\{K_{\mu}(x,x)\}\ge \epsilon $). 
	There exist constants $c_1,c_2,C_1,C_2\in (0,\infty)$which depend only on $\alpha,A,D,P_e,P$ and are such that, for any finite inner $(\alpha,A)$-uniform domain $U$,  for any integer $t$ and any $x,y\in U$ such that $d_U(x,y)\le t$,
	\begin{eqnarray*} \lefteqn{
			\frac{C_1\beta_0^t\phi_0(x)\phi_0(y)\exp(-c_1 d_U(x,y)^2/t) } {\sqrt{V(x,\sqrt{t})V(y,\sqrt{t})} \phi_0(x_{\sqrt{t}})\phi_0(y_{\sqrt{t}})}}  \\
		&\leq& \frac{K_{U}^t(x,y)}{\pi(y)} \\
		&\le &  \frac{C_2\beta_0^t \phi_0(x)\phi_0(y)\exp(-c_2 d_U(x,y)^2/t)} {\sqrt{V(x,\sqrt{t})V(y,\sqrt{t})} \phi_0(x_{\sqrt{t}})\phi_0(y_{\sqrt{t}})}  .\end{eqnarray*}
\end{cor}

\begin{proof}[Outline of the proof of Theorem \ref{theo-DH}]   To simplify notation, set 
	$$\widetilde{K}= K_{\phi_0}, \;\;\widetilde{\pi}= \phi_0^2\pi |_U.$$ 
	The estimates stated above and which we are going to obtain  for $\widetilde{K}^t=K^t_{\phi_0}$ do  not depend on the exact scaling 
	of $\phi_0$ and $\pi|_U$ as long as 
	the given choice made is used consistently.  The first key point of the proof is the fact that $\widetilde{K}=K_{\phi_0}$ is Markov (i.e., satisfies 
	$\sum_{y\in U} \widetilde{K}(x,y)=1$ for each $x\in U$) and reversible with respect to $\widetilde{\pi}=\phi_0^2 \pi|_U$. (Normalizing is optional.)
	Also, the reversible Markov chain $(\widetilde{K},\widetilde{\pi})$ satisfies  $\widetilde{K}(x,x)\ge \epsilon $ and  the ellipticity condition
	$\widetilde{K}(x,y)\ge 1/\widetilde{P}_e$  where $\widetilde{P}_e= \beta^{-1}_0 P_e\max\{\phi_0(x)/\phi_0(y): \{x,y\}\in \mathfrak E_U\}$. The constant  $\widetilde{P}_e$ is bounded above in terms of the constants $\alpha,A,D,P, P_e,\epsilon$ only.
	
	It is well-known (see~\cite[Theorem 6.34]{BalrlowLMS} or~\cite{Delm-PH}) that the two-sided Gaussian-type estimate stated in Theorem \ref{theo-DH}  for the reversible Markov chain $(\widetilde{K},\widetilde{\pi})$
	is equivalent to the conjunction of two more geometric properties which are (a) the doubling property 
	$$\forall\,x\in U,\;r>0,\;\;\; \widetilde{V}(x,2r)\le \widetilde{D}\widetilde{V}(x,r)$$
	of the volume function 
	$$\widetilde{V}(x,r)= \widetilde{\pi}(B_U(x,r))= \sum_{y\in B_U(x,r)} \phi_0^2(y)\pi|_U(y),$$
	and (b)  the Poincar\'e inequality  
	$$\min_{\xi}\sum_{B_U(x,r)} |f(y)-\xi|^2 \widetilde{\pi}(y)\le \widetilde{P} r^2 \sum_{y,z\in B_U(x,r)} |f(y)-f(z)|^2 \widetilde{K}(z,y)\widetilde{\pi}(z),$$
	for all $x\in U$, $r>0$ and all $f$ defined over $B_U(x,r)$.   See \cite{Delm-PH}.

	Theorem \ref{theo-Carleson} shows that 
	\begin{equation}
	\label{eq:phi0-volume}
	\widetilde{V}(x,r)\simeq  \phi_0(x_r)^2V(x,r)
	\end{equation}
	and the doubling property of $\widetilde{V}$ follows from Corollary \ref{cor-Carleson}.  The proof of the Poincar\'e inequality on the balls $B_U(x,r)$ follows from a variation on the argument developed in Section \ref{sec-PQPJD}
	which uses the additional  property of inner-uniform domains. See \cite{Gyrya} for the proof in the context of strictly local  Dirichlet spaces and \cite{Kelsey} for the case of discrete graphs. 
\end{proof}

The following useful corollary to Theorem~\ref{theo-DH} is illustrated in several different examples in Section~\ref{sec:examples}. 

\begin{cor}
\label{cor:exit-time-estimate}
Given the setup of Theorem~\ref{theo-DH},
$$ c\beta_0^t\frac{\phi_0(x)}{\phi_0(x_{\sqrt{t}})}\le \mathbf P_x(\tau_U>t)  \le   C \beta_0^t \frac{\phi_0(x)}{\phi_0(x_{\sqrt{t}})} ,$$
where $\tau_U$ is the random time that the process $(X_t)$ exists $U$, and $c,C>0$ are constants which depend only on $\alpha,A,D,P_e,P$.
\end{cor}
\begin{proof} Remark \ref{rem-xr} 
gives us a constant $c$ such  $d(x_r,\mathfrak X\setminus U)\ge cr.$  Note that for any $y\in B(x_{\sqrt{t}},c\sqrt{t}/2)$, we have
$\phi_0(y)\le C\phi_0(x_{\sqrt{t}})$ and $\phi_0(y_{\sqrt{t}})\ge C^{-1}\phi_0(x_{\sqrt{t}})$.  Furthermore, Theorem \ref{theo-Carleson} gives that
$$\widetilde{V}(x,\sqrt{t})\approx \widetilde{V}(y,\sqrt{t})\approx V(x_{\sqrt{t}},c\sqrt{t}/2)\approx \phi_0(x_{\sqrt{t}})^2 V(x,\sqrt{t}).$$

Now, we use the lower bound concerning $K^t_{U}$ from Corollary~\ref{cor:DH} and the previous observations to obtain
    \begin{align}
        \mathbf{P}_x(\tau_U > t) &= \sum_{y \in U} K_U^t(x,y) 
        \geq \sum_{y \in B(x_{\sqrt{t}},c\sqrt{t}/2)} K^t_U(x,y)\nonumber \\
        &\geq c'_1\beta_0^t\frac{\phi_0(x)}{\phi_0(x_{\sqrt{t}})}. \label{eq:exit-time-1}
    \end{align}
  
  For the upper bound, also using Corollary~\ref{cor:DH},  
    \begin{align}
        \mathbf{P}_x(\tau_U > t) &= \sum_{y \in U} K_U^t(x,y) \nonumber \\
        &\le C_2 \beta_0^t\frac{\phi_0(x)}{\phi_0(x_{\sqrt{t}})} \sum_{y \in U}  \frac{\phi_0(y)}{\phi_0(y_{\sqrt{t}})}  \frac{e^{-c_2d^2_U(x,y)}}{\sqrt{V(x,\sqrt{t})V(y,\sqrt{t})}} \pi(y) \nonumber \\
        & \le C'_2 \beta_0^t\frac{\phi_0(x)}{\phi_0(x_{\sqrt{t}})} .
    \end{align}
The last inequality holds because $\phi_0(y)\le C\phi_0(y_{\sqrt{t}})$ by Theorem \ref{theo-Carleson}, and 
$$\sum_{y \in U} \frac{e^{-c_2d^2_U(x,y)}}{\sqrt{V(x,\sqrt{t})V(y,\sqrt{t})}} \pi(y) \le C $$
on any doubling space.
\end{proof}

\section{Some explicit examples}
\label{sec:examples}
In this section, we consider explicit families of finite domains  indexed by a size parameter $N$ which is
comparable to the diameter of the relevant domain. Each finite domain $U$ is an $\alpha$-inner-uniform domain with a chosen ``center'' $o$ which is just a point in $U$ at maximal distance $R=R_U$ from the boundary (see Lemma \ref{lem-IUJ}).  Within each family, the inner-uniformity parameter, $\alpha\in (0,1)$, is fixed.  

The underlying 
weighted graph $(\mathfrak X,\mathfrak E,\pi,\mu)$ for these examples satisfies  A1 with $\theta=2$. In fact, in this section, the underlying space is the square grid $\mathbb Z^d$ of some fixed dimension $d$ (or some simple modification of it).  

We normalize the Perron-Frobenius eigenfunction $\phi_0$ by $\pi_U(\phi_0^2)=1$.  Because of Theorem \ref{theo-Carleson},  we have  $$\max\{\phi_0\}\le C_0\phi_0(o)$$ and (see the $(1/8,C_0)$-regularity of $\phi_0$), 
$$C_0\min_{B(o,R/2)}\{\phi_0\}\ge \phi_0(o).$$
Furthermore, $\pi_U(B(o,R/2))\ge c_0\pi(U)$.  It follows that
$$\forall y\in B(o,R/2),\;\;\phi_0(y)\approx  \phi_0(o)\simeq 1$$
uniformly within each family of examples considered.   In fact, in many examples, the choice of the point $o$ is somewhat arbitrary because one could as well pick any point $\tilde{o}$ with the property that 
$$d(\tilde{o},\mathfrak X\setminus U)\ge  \frac{1}{2}\max_{x\in U}\{d(x,\mathfrak X\setminus U)\}=\frac{R}{2}.$$
Any such point $\tilde{o}$ has the property that
$$\forall y\in B(\tilde{o},R/4),\;\;\phi_0(y)\approx  \phi_0(\tilde{o})\approx \phi_0(o)\simeq 1$$
uniformly over $\tilde{o}$ and within each family of examples considered. See Figure~\ref{B0}.

\begin{figure}[h]
	\begin{center}
		\begin{tikzpicture}[scale=.05]
		
		\draw [help lines] (0,0) grid (50,50);
		\draw [thick, blue]  (2,25) -- (12, 30)  --  (25,48)  -- (30,35)  --(48,25) -- (25,2) --  (17, 20) -- (2,25); 
		\draw [orange, fill=orange!40, opacity=.6]  (27,24)  circle  [radius= 5];

		\draw [help lines] (60,0) grid (110,50);
		\draw [thick, blue]  (62,25) -- (85,48)  -- (108,25) -- (85,2) -- (62,25); 
		\draw [thick, blue, fill]  (82,23) -- (82,26)  -- (88,26) -- (88,23) -- (82,23); 		
		\draw [name path=A, orange] (71,25)  -- (85,39) -- (99,25) -- (85, 11) -- (71,25) ;
		\draw [name path=B, orange]	 (77,23)  --  (77,27)-- (83,33) --  (87,33)-- (93,27) -- (93, 23) -- (87,17)	-- (83,17) -- (77,23);
		\tikzfillbetween[of=A and B]{orange!40, opacity=.6};	
		\end{tikzpicture}
		\caption{In light orange, regions where $\phi_0$ is approximately equal to $1$. On the left, an example in which there is essentially one central point $o$. On the right, an example in which the ``center'' $o$ can be placed in a variety of different location.}\label{B0}
	\end{center}
\end{figure}
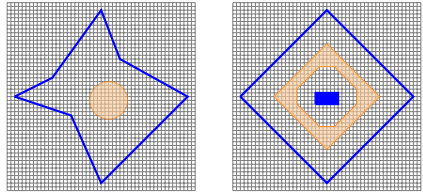

\subsection{Graph distance balls in $\mathbb Z^2$} 

\begin{figure}[h]
	\begin{center}
		\begin{tikzpicture}[scale=.05]
		
		\draw [help lines] (0,0) grid (50,50);
		\draw [thick, blue]  (2,25) -- (25,48)  -- (48,25) -- (25,2) -- (2,25); 
		\draw [fill] (25,25)  circle  [radius=.1];
		\end{tikzpicture}
		\caption{$B(N)$ in $\mathbb Z^2$}\label{B1}
	\end{center}
\end{figure}
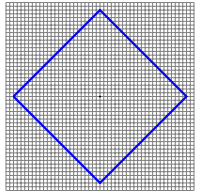

In $\mathbb Z^2$, let  $U=B(N)=\{x=(p,q)\in \mathbb Z^2: |p|+|q|\leq N\}$. This is the graph ball around $0$ in $\mathbb Z^2$.  
Equip $\mathbb Z^2$ with  the counting measure $\pi$ and with edge weights 
$$\mu_{xy}\begin{cases} 1/8 & \text{ if } |p_x-p_y|+|q_x-q_y|=1 \\ 0 & \text{ otherwise.} \end{cases}$$
The Markov kernel $K_\mu$ drives a lazy random walk on the square lattice, with holding probability $1/2$ at each vertex.  We are interested in the kernel 
$$K_U(x,y)= K_\mu(x,y)\mathbf 1_U(x)\mathbf 1_U(y)$$
which we view as defining an operator  on $L^2(U,\pi_U)$ where $\pi_U$ is the uniform probability measure on $U$.    This set is clearly inner-uniform (in fact, it is uniform because the inner distance between any two points in $U$ is the same as the distance between these point in $\mathbb Z^2$).

Let us introduce the Perron-Frobenius eigenfunction $\phi_0$ and its eigenvalue~$\beta_0$. Obviously, they depend on $N$.  This is one of the rare cases when
$\phi_0$ and $\beta_0$ can be determined explicitly:
$$\phi_0(x)= \kappa_N\cos \left(\frac{\pi}{2(N+1)} (p+q) \right) \cos \left(\frac{\pi}{2(N+1)}(p-q) \right)$$
with
$$\beta_0 = \frac{1}{2}\left(1 + \cos^2\left(\frac{\pi}{2(N+1)}\right)\right).$$
The normalizing constant $\kappa_N$ is of order $1$. Here we need to recall that $\phi_0$ vanishes on points at graph distance $N+1$ from
the origin in $\mathbb Z^2$.

To illustrate our result for estimating $\mathbf{P}_x(\tau_U > t)$ without writing long formulas, let us consider the probabilities 
$\mathbf P_{(p,0)}(\tau_U>t)$ and $\mathbf P_{(p,p)}(\tau_U>t)$  that
a random walk started at $x=(p,0)$ (for $0\le p\le N)$ and $x=(p,p)$ (for $0\le p\le N/2$), respectively, has not yet been killed by time  $t$. For all $t\le N^2$, we have
\begin{equation}
    \label{eq:exit-time-ball-1}
\mathbf P_{(p,0)}(\tau_U>t) \approx   \left(\frac{N-p}{N-p+\sqrt{t}}\right)^2, \;0\le p\le N.
\end{equation}
This comes from applying Corollary~\ref{cor:exit-time-estimate} to the eigenfunction above,
\begin{align*}
    \mathbf{P}_{(p,0)}(\tau_U > t) &\approx \frac{\phi_0((p,0))}{\phi_0((p,0)_{\sqrt{t}})} \\
    &\approx \frac{\phi_0((p,0))}{\phi_0((p-\sqrt{t},0))} \\
    &\approx  \frac{(\cos (\frac{\pi}{2N}p))^2}{\cos(\frac{\pi}{2N}(p-\sqrt{t}))^2}.
\end{align*}
Now, use that $\cos\left(\frac{\pi}{2N}x\right)= \sin \left(\frac{\pi}{2N}(N-x)\right)\sim \frac{\pi}{2N}(N-x)$.   In particular, for any fixed $0<t\le N^2$, $P_{(p,0)}(\tau_U>t)$ vanishes asymptotically like $\frac{(N-p)^2}{t}$ as $p$ tends to $N$. 

Similarly, for  $0<t\le N^2$,
$$\mathbf P_{(p,p)}(\tau_U>t) \approx   \left(\frac{N-2p}{N-2p+\sqrt{t}}\right),\;0\le 2p\le N.$$
In this case,  for any fixed $0<t\le N^2$,  $\mathbf P_{(p,p)}(\tau_U>t)$ vanishes like $\frac{N-2p}{\sqrt{t}}$  when $p$ tends to $N/2$.

\begin{rem} While our results apply equally well to the graph distance balls of $\mathbb Z^d$ for $d>2$, they are much more complicated in that case and there is no explicit formula for $\phi_0$ or the eigenvalue $\beta_0$. The ball is  a polytope with faces of dimension $0,1,\dots,d$.  The vanishing of $\phi_0$ near each of these faces is described by a power function of the distance to the particular face that is considered
	and the exponent depends on the dimension of the face and on the  angles made by the higher dimensional faces meeting at the given face (the exponent is always $1$ when approaching the highest dimensional faces).
\end{rem}

\subsection{$B(N)\setminus\{(0,0)\}$ in $\mathbb Z^2$} 
\begin{figure}[h]
	\begin{center}
		\begin{tikzpicture}[scale=.1]
		\draw [thick, blue] (25,25)  circle  [radius=.4];
		\draw [red]  (25,25) circle [radius=7];
		\node at (21,27) {$2$};
		\draw [red] (19,19) circle [radius=5];
		\node at (18,18) {$1$};
		\draw [red] (34,34) circle [radius=7];
		\node at (33,33) {$3$};
		\draw [red] (42,25) circle [radius=7];
		\node at (42,25) {$4$};
		\draw [help lines] (0,0) grid (50,50);
		\draw [thick, blue]  (2,25) -- (25,48)  -- (48,25) -- (25,2) -- (2,25); 
		\draw [fill] (25,25)  circle  [radius=.1];
		\end{tikzpicture}
		\caption{$B(N)\setminus \{0\}$ in $\mathbb Z^2$ (the blue central point is part of the boundary)} \label{B2}
	\end{center}
\end{figure}
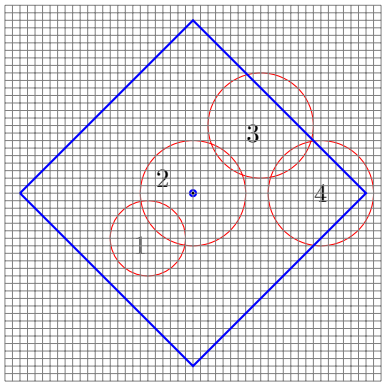
The case when $U=B(N)\setminus\{(0,0)\}$ is interesting because we are able to describe precisely the behavior of $\phi_0$ even though there is no
explicit formula available.  First, we note again that this is an inner-uniform domain (there is no preferred point $o$ in this case, since any point at distance of order $N/2$
from $(0,0)$ will do).   Theorem \ref{theo-comph} will play a key part in allowing us to describe the behavior of $\phi_0$.  First, we can use path arguments and an appropriate test function to show that
$$1-\beta_0 \approx N^{-2}.$$
For the upper bound, use the test function $$f((p,q))=\min\{d((0,0),(p,q)),N+1-d(0,0),(p,q))\}$$
which vanishes at all boundary points for $U$

Second, we show that
$$\phi_0((p,q))\approx  \frac{(N-|p+q|)(N-|p-q|))\log (1+|p|+|q|)}{N^2\log N }.$$
To obtain this result, cover $U$ by a finite number (independent of $N$) of  $\mathbb Z^2$ balls
$\{B_j\}$ of radius of order $N$ so that the trace of $U$ in each of the balls $2B_j$ is of one of the following four types: (1) no intersection with the boundary of $U$; (2)  the intersection with the boundary  of $U$ is $\{(0,0)\}$; (3) the intersection with the boundary of $U$ is  a subset of $\{(p,q): p+q=N\}$ or $\{(p,q): p-q=N\}$ of $\{(p,q): p+q=-N\}$ or $\{(p,q): p-q=-N\}$; and (4) the intersection with the boundary is a corner formed by two of the previously mentioned lines.  See Figure \ref{B2} for an illustration of these four types.  In case (1), we know that $\phi$ is approximately constant in $B_j$. Moreover, this approximately constant value must be (approximately) the maximum value of $\phi_0$ because of Theorem \ref{theo-Carleson}, and this constant must be approximatively equal to $1$ because $\phi_0$ is normalized by $\pi_U(\phi_0^2)=1$. This is compatible with the proposed formula describing $\phi_0$. In case (2), Theorem \ref{theo-comph} allows us to compare
$\phi_0((p,q))$ to the harmonic function $h((p,q))$ equal to the discrete modified Green's function $$A((0,0),(p,q))=\sum_0^\infty(M^t((0,0),(p,q))-M^t((0,0),(0,0))$$ on $\mathbb Z^2\setminus \{(0,0)\}$
Here $M$ is the Markov kernel of aperiodic simple random walk on $\mathbb Z^2$.
 It is well-known that this function is comparable to $\log(1+|p|+|q|)$ (See \cite[Chapter 3]{Spitzer} from which we borrowed the notation $A(x,y)$. More precise estimates are available using sharp version of the local limit theorem, but this is enough for our purpose). Because the ball $B_j$ in question must contain a point at distance of order $N$ from the boundary of $U$ at which $\phi_0$ is of order $1$, we find that, in such a ball, 
$$\phi_0((p,q))\approx \frac{\log (1+|p|+|q|)}{\log N}.$$
Again, this estimate is compatible with the proposed formula. 
In case (3), we easily have a linear function $h$ vanishing on the (flat) portion of the boundary contains in that ball and positive discrete harmonic in $U$. Thanks to Theorem~\ref{theo-comph}, this provides  the estimate
$$\phi_0((p,q))\approx \frac{ d_U((p,q)),\mathfrak X \setminus U)}{N}$$
in balls of this type, which has the form suggested by the proposed formula. Finally, in case (4), and, for definiteness, in the case the ball $B_j$ is centered at the corner of intersection of the line ${\{(p,q): p+q=N\}}$ and ${\{(p,q): p-q=N\}}$, the function $h((p,q))= (N-p-q)(N-p+q) $ vanishes on these two lines and is discrete harmonic.  This gives (again,using Theorem \ref{theo-comph})
$$\phi_0((p,q))\approx \frac{ (N-p-q)(N-p+q)}{N^2}$$
as desired.

\subsection{$B(N)\setminus \{\mathbf 0\}$ in $B(N)$, in dimension $d>1$} \label{pointedball}

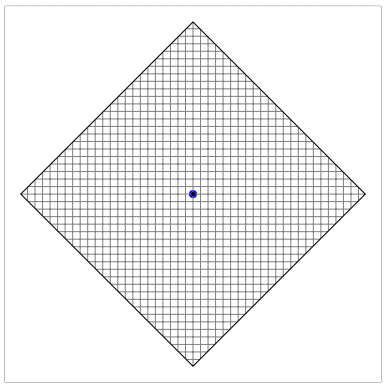
\begin{figure}[h]
	\begin{center}
		\begin{tikzpicture}[scale=.1]
		\draw [blue, thick] (25,25)  circle  [radius=.4];
		\draw [help lines] (0,0) grid (50,50);
		\draw [thick]  (2,25) -- (25,48)  -- (48,25) -- (25,2) -- (2,25); 
		\draw [fill] (25,25)  circle  [radius=.2];
		\fill [white, even odd rule] (0,0) -- (0,50)  -- (50,50) -- (50,0) -- (0,0)
		(2,25) -- (25,48)  -- (48,25) -- (25,2) -- (2,25); 
		\end{tikzpicture}
		\caption{$B(N)\setminus \{0\}$ in $B(N)$ (The blue central point is the entire boundary.)} \label{B3}
	\end{center}
\end{figure}

First we explain the title of this subsection. Consider the simple random walk in the ball $B(N)\subset \mathbb Z^d$, with any reasonable reflection
type hypothesis on the boundary of $B(N)$. Our aim is to study absorption at $0$ for this random walk on the finite set  $B(N)$.  To put this example in our general framework,
we set $\mathfrak X_N= B(N)$ equipped with the edge set $\mathfrak E_N$ induced by the underlying square lattice, that is the collection of all lattice edges with both end points in $B(N)$. The measure $\pi$ on $\mathfrak X_N=B(N) $ is the counting measure and each lattice edge $e$ in $\mathfrak E$ is given the weight $\mu(e)=1/(2d)$.  This means that the Markov kernel $K_\mu$ for our underlying walk has no holding at point $x\in B(N-1)\subset B(N)$ and holding probability $\nu(x)/(2d)$ where $\nu(x)= 2d- \#\{y\in B(N): \{x,y\}\in \mathfrak E_N\}$ when $x\in B(N)\setminus B(N-1)$ (this holding probability at the boundary is always at least $1/2$).   The domain  $U_N$ of interest to us here is $U_N=B(N)\setminus \{\mathbf 0\}$ (inside $B(N)$) whose sole outside boundary point is 
the center $\mathbf 0$.  When the dimension $d$ is at least $2$, this is an inner-uniform domain in $(\mathfrak X_N,\mathfrak E_N)$ (there is no canonical center but any point at distance at least $N/2$ from $\mathbf 0$ can be chosen to be the center $o$).   

Because the domain  $U_N$ is inner-uniform  (uniformly in $N$), Theorem~\ref{theo-comph} yields 
$$\mathbf P_x(\tau_U >t)\approx  \frac{\beta_0^{t} \phi_0(x)}{\phi_0(x_{\sqrt{t}})} $$
and, for $t \geq N^2$, Corollary~\ref{cor:KU-rate-of-conv} gives,
$$|K^t_U(x,y) - \phi_0(x)\phi_0(y)\beta_0^{t} |U|^{-1}|  \le C \beta_0^t\phi_0(x)\phi_0(y) e^{-t/N^2}.$$

As in the previous examples, the key is to obtain further information on $\beta_0$ and $\phi_0$.
For that we need to treat the cases $d=2$ and $d>2$ separately. In both cases, we use Theorem \ref{theo-comph} to estimate $\phi_0$.

\subsubsection{Case $d=2$}    

The first task is to estimate $1-\beta_0$  from above and below. This is done by using the same argument explained in \cite[Example 3.2.5: The dog]{LSCStF}. See Subsection \ref{par-eig} below where we spell out
the main part of the argument in question.  The upshot is that    $1-\beta_0\approx 1/N^2\log N$.   We know that $\phi_0(x)\approx 1$  when $x$ is at graph distance at least $N/2$ from 
$\mathbf 0$ (see the outline describe in Example \ref{pointedball} for type 1 balls).  To estimate $\phi_0$ at other points, we compare it with the global positive harmonic function from $\mathbb Z^2\setminus \{\mathbf 0\}$
given the so-called modified Green's function $h(x)= A(\mathbf 0,x)= \sum_{t=0}^\infty [M^t(\mathbf 0,x)-M^t(\mathbf 0,\mathbf 0)]$ where $M$ stands here for the Markov kernel of aperiodic simple random walk in $\mathbb Z^2$ as in Example \ref{pointedball}.  Note that $h$ vanishes at $0$. Classical estimates (e.g., \cite{Spitzer}) yield $h(x) \approx \log |x|$.    This, together with Theorem \ref{theo-comph} and the estimate when $x$ is at distance at least $N/2$ from $\mathbf 0$,  gives
$$\phi_0(x)  \approx  \frac{\log |x|}{\log N}.$$

\subsubsection{Case $d>2$}   The case $d>2$ is perhaps easier although the arguments are essentially the same.  The eigenvalue $\beta_0$ is estimated by $1-\beta_0\approx 1/N^d$ and the harmonic function  $h(x)=\sum_0^\infty M^t(\mathbf 0,x)-\sum_0^\infty M^t(\mathbf 0,\mathbf 0)$ (these sums converge separately because $d>2$) is estimated by  $h(x)\approx  \left(1-  1/(1+|x|)^{d-2}\right)$. This gives
$$\phi_0(x) \approx    \frac{\left(1-  1/(1+|x|)^{d-2}\right) }{ \left(1-  1/(1+N)^{d-2}\right)} \approx 1.$$

\subsubsection{Discussion}  The first thing to observe in these examples is the fact that $1-\beta_0=o(1/N^2)$.  For $t\ge N^2$  we have
$$|K^t_U(x,y) - \phi_0(x)\phi_0(y)\beta_0^{t} |U|^{-1}|  \le C \beta_0^t\phi_0(x)\phi_0(y) e^{-t/N^2}.$$  
In the case $d=2$, if $\epsilon>0$ is fixed and  $x,y$ are at distance greater than $N^\epsilon$ from the origin, we can without loss of information, 
simplify the above statement and write
$$|K^t_U(x,y) - \phi_0(x)\phi_0(y)\beta_0^{t} |U|^{-1}|  \le C e^{-t/N^2}.$$  
Because $\beta_0^t$ decays significantly slower than  $e^{-t/N^2}$, this provides a good example of a quasi-stationary distribution during the time interval 
$t\in (N^2, N^2\log N)$. 

In the case $d>2$, the same phenomenon occurs, only in an even more tangible way.  For any $x,y\in U_N$, $\phi_0(x),\phi_0(y)$ are uniformly bounded away from $0$  (even for the neighbors of the origin, $\mathbf 0$).  Moreover, $1-\beta_0\approx 1/N^{d}=o(1/N^2)$.  For $t\ge N^2$  and $x,y\in U_N$,
$$|K^t_U(x,y) - \phi_0(x)\phi_0(y)\beta_0^{t} |U|^{-1}|  \le C e^{-t/N^2}.$$ 
On intervals of the type $t \in (TN^2,N^d/T)$ with $T$ large enough,  $  K^t_U(x,y) $ is well approximated  by $ \phi_0(x)\phi_0(y) |U|^{-1} $ because, on such intervals, $\beta_0^t$ remains close to $1$.

\subsection{ $B(N)\setminus B_2(L)$  in $B(N)$, in dimension $d>1$}

We work again in $\mathfrak X_N=B(N)$ with the weighted graph structure explained above.   We use $B_2(r)$ to denote
the trace on the lattice $\mathbb Z^d$ of the Euclidean (round) ball centered at the origin, $\mathbf 0$. The domain we wish to investigate is $U_{N,L}=B(N)\setminus B_2(L)$ with  $L=o(N)$ so that the number of points in $U_{N,L}$ is of order $N^d$ and $U_{N,L}$ is inner-uniform (uniformly in all choices of $N,L$). Again, the chosen center $o$ in $U_{N,L}$ can be any point  at graph distance $N$ from $\mathbf 0$.  All the estimates described below are uniform in $N,L$ as long as $L=o(N)$.

\subsubsection{Estimating $\beta_0$} \label{par-eig}
First we explain how to estimate $\beta_0$ for
$U=B(N)\setminus B_2(L)$  in $B(N)$ using and argument very similar to those used in
\cite[Example 3.2.5: The dog]{LSCStF}. For each point $x\in U$ 
fix a graph geodesic discrete path $\gamma_x$ that joins $x$ to the origin in $\mathbb Z^d$ while staying as close as possible to the straight line
from $x$ to the origin. We stop $\gamma_x$ whenever it reaches a point in $B_2(L)$. 
\begin{figure}[h]
	\begin{center}
		\begin{tikzpicture}[scale=.1]
		\draw [fill=blue] (25,25)  circle  [radius=3];
		\draw [help lines] (0,0) grid (50,50);
		\draw [thick]  (2,25) -- (25,48)  -- (48,25) -- (25,2) -- (2,25); 
		\draw [fill] (25,25)  circle  [radius=.2];
		\fill [white, even odd rule] (0,0) -- (0,50)  -- (50,50) -- (50,0) -- (0,0)
		(2,25) -- (25,48)  -- (48,25) -- (25,2) -- (2,25); 
		\node at (44,31) {$x$};
		\draw  (25,25) -- (42,30);
		\draw[thick] (42,30) -- (41,30) --(40,30) --(40,29)-- (39,29) -- (38,29) --(37,29) -- (37,28) -- (36,28) --(35,28) -- (34,28)
		-- (34,27) --(33,27) -- (32,27) --(31,27)--(30,27)-- (30,26)
		--(29,26) --(28,26) --(27,26)
		;
		\draw[fill] (42,30) circle [radius=.2];
		\draw[yellow, thick, fill=yellow!20, opacity=.6] (25,25) -- (44,34) -- (46, 28) --(25,25); 
		\draw[thick, red] (30,27)-- (30,26);
		
		\end{tikzpicture}
		\caption{Paths to the origin in $B(N)\setminus B_2(L)$} \label{B4path}
	\end{center}
\end{figure}
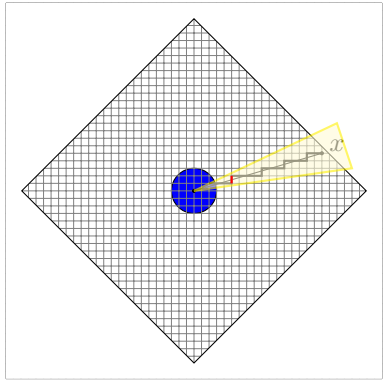

 Given a function $f$ on $B(N)$ which is equal to zero on $B_2(L)$ and a directed edge $e=(x,y)$, set $df(e)=f(y)-f(x)$. The edges along a path $\gamma_x$ are all directed toward the origin. Using this notation, we have
$$|f(x)|^2\le |\sum_{e\in \gamma_x} df(e)|^2
\le |\gamma_x|_w \sum_{e\in \gamma_x} |df(e)|^2 w(e)$$
where $w$ is a weight function on the edge $e$ which will be chosen later
and $|\gamma|_w= \sum_{e\in \gamma} w(e)^{-1}$. 
Summing over all $x\in U$, we obtain
$$\sum_U|f|^2 \le  2d\sum_{e\in \mathfrak E}\left(\sum_{x: \gamma_x\ni e}|\gamma_x|_w w(e)\right) \frac{|df(e)|^2  }{2d}
\le C_w(d,N,L) \mathcal E_\mu(f,f) $$
where $$C_w(d,N,L)= 2d \max_{e\in \mathfrak E}\left\{w(e)\sum_{x:\gamma_x\ni e}|\gamma_x|_w\right\}. $$
Using the Raleigh quotient formula for $1-\beta_0$, we obtain the eigenvalue estimate
$$\beta_0\le 1 - 1/C_w(d,N,L)$$
for any choice of the weight $w$.  Here we choose $w(e)$ to be the Euclidean distance of the edge $e$ to the origin raised to the power $d-1$. This implies that  
$$|\gamma_x|_w\le C_d\times \begin{cases} \log (N/L) &\mbox{ when } d=2,\\
L^{-d+2} & \mbox{ when } d>2 \end{cases}$$
for some constant $C_d$ which depends on the dimension $d$. It remains to count how many $x$ use a given edge $e$. Because we use paths that remain close to the straight line from $x$ to the origin,
the vertices $x$ that use and given edge $y$ at Euclidean distance $T$
from the origin must be in a cone of aperture bounded by $C_d/T$. The number of these vertices is at most $C_d N\times (N/T)^{d-1} $
where the constant $C_d$ changes from line to line. See Figure \ref{B4path}. Recall that $w(e)\approx T^{d-1}$. Putting things together yields
$$C_w(d,N,L)\le C_d \times  \begin{cases}{cl} N^2 \log (N/L) &\mbox{ when } d=2,\\
N^d L^{-d+2} & \mbox{ when } d>2. \end{cases}$$
In terms $\beta_0$ this gives
$$1-\beta_0\ge C^{-1}_d \times  \begin{cases}{cl} 1/N^2\log (N/L) &\mbox{ when } d=2,\\
L^{d-2}/N^d & \mbox{ when } d>2. \end{cases} $$ 

The upper-bound is a simple computation using a test function which take the value $0$ on $B_2(L)$ and increase linearly at rate $1$ until taking the value $L$. After that the test function remains constant equal to $L$.  Note that this bound interpolates between the case $L=1$ (more or less, the previous case) when $1-\beta_0\approx 1/N^d$
and the case when $L$ is a fixed small fraction of $N$, in which case  $1-\beta_0\approx 1/N^2$.

\subsubsection{Estimating $\phi_0$ in the case $d=2$}

\begin{figure}[h]
	\begin{center}
		\begin{tikzpicture}[scale=.1]
		\draw [fill=blue] (25,25)  circle  [radius=3];
		\fill [yellow, even odd rule] (25,25) circle[radius=6] circle[radius=3];
		\draw [help lines] (0,0) grid (50,50);
		\draw [thick]  (2,25) -- (25,48)  -- (48,25) -- (25,2) -- (2,25); 
		\draw [fill] (25,25)  circle  [radius=.2];
		\fill [white, even odd rule] (0,0) -- (0,50)  -- (50,50) -- (50,0) -- (0,0)
		(2,25) -- (25,48)  -- (48,25) -- (25,2) -- (2,25); 
		\end{tikzpicture}
		\caption{$B(N)\setminus B_2(L)$: In the yellow region of width $L$ around $B_2(L)$, $\phi_0(x)\approx (\frac{\log L}{\log N})d(x,B_2(L))$.} \label{B4}
	\end{center}
\end{figure}
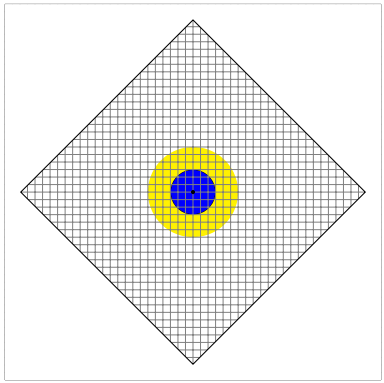

The technique is the same as the one described below for the case $d>2$. Here we omit the details and  only describe the findings. The behavior of the function $\phi_0$ is best described by considering two zones.
See Figure \ref{B4}.
The  first zone is  $B_2(2L)\setminus B_2(L)$ in which the function $\phi_0$
is roughly  linearly increasing as the distance from $B_2(L)$ increases and satisfies 
$$\phi_0(x)\approx \frac{ \log L}{\log N}d(x,B_2(L)).$$
The second zone is $B(N)\setminus B_2(2L)$ in which  $\phi_0$ satisfies
$$\phi_0(x)\approx \frac{\log |x|}{\log N}.$$

\subsubsection{Estimating $\phi_0$ in the  case $d>2$}
Because of the basic known property of $\phi_0$ discussed earlier, it satisfies $\phi_0\approx 1$ on the portion of  $U_{N,L}$  which is at distance of order $N$  from $B_2(L)$   (the outer-part of $U_{N,L}$).  The function $\phi_0$ is also bounded on $U_N$, uniformly in $N,L$.  One key step is to find out the region in $U_{N,L}$ over which $\phi_0$ is bounded below by a  fixed small $\epsilon$.   For this purpose we use, a simple comparison with the Green's function 
$G(\mathbf 0,y)=\sum K^t(\mathbf 0,y)$, of the simple random walk on $\mathbb Z^d$.   First, find the smallest positive $T=T(L)$  such that
$$B_2(L)\subset \{x\in \mathbb Z^d: G(\mathbf 0,x)\ge T\}.$$
Recall that 
\begin{equation}
G(\mathbf 0,x)\approx 1/(1+|x|)^{d-2}  \label{Green=}
\end{equation}
This shows that $T\approx  1/L^{d-2}$  (the implied constants in this estimate depend on $d$ because we are using both the Euclidean norm and the graph distance). 

We are going to compare $\phi_0$ to a multiple of the harmonic function $$v(x)=1-G(\mathbf 0,x)/T.$$   It is clear that $v\approx 1$  when $|x|=N$ (uniformly over $N,L$).  It follows that there is a constant $a>0$,
independent of $N,L$, such that $ \phi_0-av$ is  greater or equal to $4$ on the boundary of $V_{N,L}=B(N)\setminus \{z: G(\mathbf 0,z)\ge T\}$  (the constant $a$ is chosen so that this is true on the outer-boundary whereas, on the inner-boundary, $v=0$, $\phi_0>0$).   Suppose that $\phi_0-av$ attains a minimum at an interior point $x_0$ in $V_{N,L}$.  This would imply that    $\phi_0(x_0)-av(x_0)\le \beta_0 \phi_0(x_0)-av(x_0)$, that is, $1\le \beta_0$, a contradiction.
It follows that $\phi_0\ge a v$  on $V_{N,L}$.  Because of the known estimate for $G$ recalled above and of the general properties of $\phi_0$, this shows that 
$$\phi_0\approx 1 \mbox{ over }B(N)\setminus B_2(2L).$$  

All the statements and arguments given so far would work just as well if we where considering $B(N)\setminus B(L)$ instead of $B(N)\setminus B_2(L)$.
These two cases differ  only  in the behavior of their respective $\phi_0$ near the interior boundary.   For $U_{N,L}=B(N)\setminus B_2(L)$,
it is possible to show that 
$$\phi_0(x) \approx  \frac{d( x, B_2(L)}{L}.$$
The fundamental reason for this is the (uniform) smoothness of the boundary of the Euclidean ball $B_2(L)$ (viewed at scale $L$).  
The result is a consequence of one of  the main result in \cite{VarMilan1} (see also \cite{varMilan2,VarMilan3}).

\subsection{$B(N)\setminus B(L)$, $d=2$}

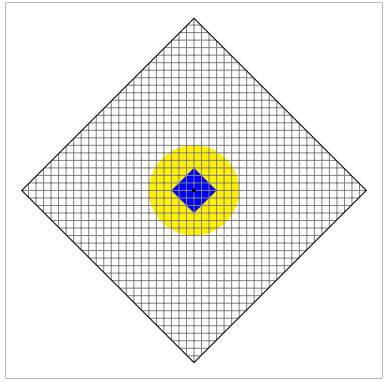
\begin{figure}[h]
	\begin{center}
		\begin{tikzpicture}[scale=.1]
		\draw [fill=blue] (25,25)  circle  [radius=3];
		\fill [yellow, even odd rule]  (25,25) circle [radius=6]
		(22,25) -- (25,28)  -- (28,25) -- (25,22) -- (22,25); 
		\draw [help lines] (0,0) grid (50,50);
		\draw [thick]  (2,25) -- (25,48)  -- (48,25) -- (25,2) -- (2,25); 
		\draw [fill] (25,25)  circle  [radius=.2];
		\fill [white, even odd rule] (0,0) -- (0,50)  -- (50,50) -- (50,0) -- (0,0)
		(2,25) -- (25,48)  -- (48,25) -- (25,2) -- (2,25); 
		\end{tikzpicture}
		\caption{$B(N)\setminus B(L)$} \label{BNL}
	\end{center}
\end{figure}  

Next we consider $B(N)\setminus B(L)$, $L<N/2$, in dimension $d=2$.
We have again
$$\beta_0\approx 1/ N^2\log (N/L)$$ 
In the zone $B(N)\setminus B_2(2L)$ (outside the yellow area in Figure \ref{BNL}), the function $\phi_0$ is estimated by
$$\phi_0(x)\approx \frac{\log |x|}{\log N }.$$
We note here that the exact outer shape of the yellow region is unimportant (we could have drawn a diamond instead of a round ball). In order to describe the function $\phi_0$ is the yellow zone ($B_2(2L)\setminus B(L)$),  it is convenient to split the region into eight areas, each of which is of one of two types.  See Figure \ref{BNL2}
where the two red circles describes the two types of region that we will consider.
The estimates described below are compatible when two regions intersect. In the type $1$ regions, because the relevant piece of the boundary at scale $L$ is flat, 
$$\phi_0(x)\approx \frac{\log L}{L\log N} d(x,B(L)).$$

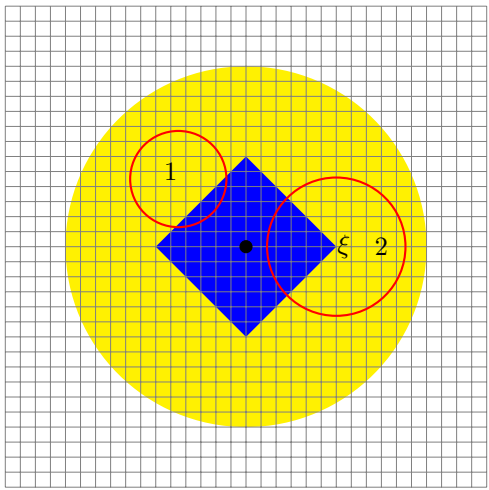
\begin{figure}[h]
	\begin{center}
		\begin{tikzpicture}[scale=.2]
		\draw [fill=blue] (16,16)  circle  [radius=6];
		\fill [yellow, even odd rule]  (16,16) circle [radius=12]
		(10,16) -- (16,22)  -- (22,16) -- (16,10) -- (10,16); 
		\draw [help lines] (0,0) grid (32,32);
		\draw [fill] (16,16)  circle  [radius=.4];
		\draw [thick, red] (11.5,20.5) circle [radius=3.2];
		\draw [thick, red] (22,16) circle [radius=4.6];
		\node at (11,21) {$1$};
		\node at (22.5,16) {$\xi$};
		\node at (25,16) {$2$};
		\end{tikzpicture}
		\caption{The yellow zone in  $B(N)\setminus B(L)$} \label{BNL2}
	\end{center}
\end{figure}   

In the type $2$ regions, centered around one of the corner of $B(L)$,
$$\phi_0(x)\approx   \frac{\log L}{\log N}  (\rho/L)^{2/3} \cos \left( 4\theta/3\right),\;\;
x=(x_1,x_2), x-\xi=\rho e^{i \theta}$$
Here $\xi$ is the tip of the diamond $B(L)$ around which the region of type 2 is centered, $\theta$ is the angle in $[-\pi,\pi)$ measured from the median semi-axis through the tip.  This last estimate  is obtained by using the results of \cite{VarMilan1}  to derive the behavior of discrete harmonic function in a type $2$ region from the behavior of the
analogous classical harmonic function in the analogous domain in $\mathbb R^2$ (a cone with aperture $3\pi/2$).

\section{Summary and concluding remarks}
This article gives detailed quantitative estimates describing the behavior of Markov chains on certain finite sub-domains of a large class of underlying graphs before the chain exits the given sub-domain. There are two types of key assumptions. 

The first set of assumptions concern the underlying graph (before we consider a particular sub-domain). This underlying graph belongs to a large class of graphs whose properties mimic those of the square grid $\mathbb Z^m$. This class of graphs can be defined in a variety of known equivalent different ways: it satisfies, uniformly at all scales and locations, the doubling volume condition and Poincar\'e inequality on balls; equivalently, the iterated kernel of simple random walk satisfies detailed two-sided ``Gaussian or sub-Gaussian bounds''; or, equivalently, it satisfies a certain type of parabolic Harnack inequality for (local) positive solutions of the discrete heat equation. See the books \cite{BalrlowLMS,GrigBook} for details and pointers to the literature.  It is perfectly fine for the reader to concentrate attention on the case of the square grid $\mathbb Z^m$. However, even if the reader concentrates on this special case, the techniques that are then used to study the behavior of the chain in sub-domains are the same techniques as the ones needed to understand the more general class of graphs we just alluded to. 

The second set of assumptions concerns the finite sub-domains of the underlying graph that can be treated. These sub-domains are called John domains and inner-uniform domains, and both are defined using metric properties. For John domains (the larger class), there is a central point $o$ and any other point of the domain can be joined to the central point $o$ by a carrot-shaped region that remains entirely contained in the domain. The inner-uniform  condition (a strictly more restrictive condition) requires that any pair of point in the domain can be joined by a banana-shaped region that is entirely contained in the domain.   It is not easy to get a good precise understanding of the type of regions afforded by these conditions because they allow for very rugged domains (e.g., in the Euclidean plane version, the classical snowflake). They do cover many interesting examples. 

It is worth emphasizing here that the strength of the results obtained in this article comes from the conjunction of the two types of assumptions described above.  Under these assumptions, one can describe the results of this paper by saying that any question about the behavior of the chain until it exits the given sub-domain boils down (in a technically precise and informative way) to estimating the so-called Perron-Frobenius eigenvalue and eigenfunction of the domain.  Let us stress here that it is quite clear that it is necessary to understand the Perron-Frobenius pair in order to get a handle on the behavior of the chain until it exits the domain.  What is remarkable is the fact
that it is essentially sufficient to understand this pair in order to answer a host of seemingly more sophisticated and intricate questions. This idea is not new as it is the underlying principle of the method known as the Doob-transform technique which has been used by many authors before. Under two basic types of assumptions described above, this idea works remarkably well. In  different contexts (diffusion, continuous metric measure spaces, Dirichlet forms and unbounded domains) this same idea is the basis for many of the developments in \cite{Pinsky,Gyrya}. 

For inner-uniform domains, the more restrictive class of domains, the results obtained are rather detailed and complete. For John domains, the  results obtained, which depend on the notion of moderate growth (see Lemma~\ref{lem-Q}), are less detailed and leave interesting questions open.

We conclude with pointing out to further potential developments. This article focuses on the behavior before the exit time of the given finite domain.  In the follow-up paper \cite{DHSZ2}, we discuss, in the case of inner-uniform domains, the implications of these results on the problem of understanding the exit position. This can be framed as an extension of the classical Gambler's ruin problem.  In a spirit similar to what was said above,
\cite{DHSZ2} shows how Gambler's ruin estimates on inner-uniform domains reduce to an understanding of the Perron-Frobenius eigen pair of the domain. Much less is known for John domains in this direction.  

Having reduced a certain number of interesting questions to the problem of estimating the Perron-Frobenius eigenfunction $\phi_0$ of a given finite domain, we owe the reader to observe that this task, estimating $\phi_0$, remains extremely difficult. There are plenty of interesting results in this direction and many more natural open problems. An illustrative example is the following: consider the cube of side length $2N$ in $\mathbb Z^3$ with the three main coordinate axes going through the center removed; this is an inner-uniform domain and we would like to estimate the eigenfunction $\phi_0$. Another example, less mysterious, is to find precise estimates for $\phi_0$ for the graph balls in $\mathbb Z^m$ with $m\ge 3$.

For finite domains in $\mathbb Z^m$ with diameter $R$, we have proved that the key convergence parameter for the quasi-stationarity problems considered here is order $R^2$ for $\alpha$-inner-uniform domains and no more than $R^{2+\omega}$ for $\alpha$-John domains where $\omega\ge 0$ depends only 
on the dimension $m$ and John parameter $\alpha$. It is an interesting open question to decide
whether or not $\omega$ can be taken to be always equal to $0$. Even if there are John domains where $\omega$ must be positive, it is clear that there is a class of John domains that is strictly larger than the class of all inner-uniform domains and for which one can take $\omega=0$. Elucidating this question is an interesting open problem in the present context and in the context of analysis in Euclidean domains.  

\section*{Acknowledgements}
The authors thank Tianyi Zheng for her contributions to the early stage of this work and, in particular, for pointing out that the idea of moderate growth is useful in treating John domains. Laurent Saloff-Coste was partially supported by NSF grant DMS-1707589. Kelsey Houston-Edwards was partially supported by NSF grants DMS-0739164 and DMS-1645643.

\bibliographystyle{plain}

\bibliography{}

\end{document}